\documentclass[11pt,twoside, a4paper, english, reqno]{amsart} 
\usepackage[foot]{amsaddr}
\usepackage{a4wide}
\usepackage{amscd}
\usepackage{amssymb}
\usepackage{amsthm}
\usepackage{graphicx}
\usepackage{amsmath,calrsfs}
\usepackage{latexsym}
\usepackage{enumitem,dsfont}

\usepackage{upref}
\usepackage[colorlinks]{hyperref}
\usepackage{color,graphics}
\usepackage{upref,hyperref,color}

\usepackage[usenames,dvipsnames]{xcolor}
\usepackage{pdfsync}
\usepackage{ulem,cancel,pgf}

\usepackage{frcursive}
\def\k{\text{\begin{cursive}
\hspace{-0,1cm}\textcal{k}\hspace{-0,15cm}
\end{cursive}}}

\theoremstyle{plain}
\newtheorem{theorem}{Theorem}[section]
\theoremstyle{plain}
\newtheorem{lemma}[theorem]{Lemma}
\newtheorem{prop}[theorem]{Proposition}
\newtheorem{cor}[theorem]{Corollary}

\theoremstyle{definition}
\newtheorem{definition}{Definition}[section]

\newtheorem{remark}{Remark}[section]

\newtheorem*{maintheorem*}{Main Theorem}
\newtheorem*{maincorollary*}{Main Corollary}

\DeclareFontFamily{U}{BOONDOX-calo}{\skewchar\font=50 }
\DeclareFontShape{U}{BOONDOX-calo}{m}{n}{
	<-> s*[1.05] BOONDOX-r-calo}{}
\DeclareFontShape{U}{BOONDOX-calo}{b}{n}{
	<-> s*[1.05] BOONDOX-b-calo}{}
\DeclareMathAlphabet{\mathcalb}{U}{BOONDOX-calo}{m}{n}
\SetMathAlphabet{\mathcalb}{bold}{U}{BOONDOX-calo}{b}{n}
\DeclareMathAlphabet{\mathbcalb}{U}{BOONDOX-calo}{b}{n}

\newcommand{\R}{\ensuremath{\mathbb{R}}}

\newcommand{\E}{\ensuremath{\mathbb{E}}}

\newcommand{\goto}{\ensuremath{\rightarrow}}

\def\e{{\text{e}}}

\numberwithin{equation}{section} \allowdisplaybreaks

\title[Optimal control  of  third grade fluids  with multiplicative noise  ]
{Optimal control  of   third grade fluids  with multiplicative noise  }

\date{\today}

\author[Yassine Tahraoui]{ Yassine Tahraoui}
\address[ Yassine Tahraoui]{
	Center for Mathematics and Applications (NovaMath),  NOVA	SST,	Portugal}
\email[Yassine Tahraoui]{tahraouiyacine@yahoo.fr}

\author[Fernanda Cipriano]{Fernanda Cipriano}
\address[Fernanda Cipriano]{
	Center for Mathematics and Applications (NovaMath), NOVA	SST and Department of Mathematics, NOVA	SST,	Portugal}
\email[Fernanda Cipriano]{cipriano@fct.unl.pt}

\allowdisplaybreaks[1]

\usepackage{comment}
\begin{document}

	\begin{abstract}
This work aims to control the dynamics of certain non-Newtonian fluids 	in a bounded domain of $\mathbb{R}^d$, $d=2,3$  perturbed by a multiplicative Wiener noise,  the control acts	as	a predictable distributed random force, and the goal is to achieve a predefined velocity profile under a minimal cost.
  Due to the strong nonlinearity of the stochastic state equations, strong solutions are available just locally in time, and  the cost functional 	includes an appropriate  stopping time. First,	we  show the existence of an optimal pair.   Then,	we	show that the solution of the stochastic forward	linearized equation coincides with the G\^ateaux derivative of the control-to-state mapping,	after	establishing	some	stability results.	Next,	
   we analyse the backward stochastic adjoint equation;	where the uniqueness of  solution holds	only	when	$d=2$.
 Finally, we establish a duality relation  and deduce the necessary  optimality conditions.
  		\end{abstract} 
  	
	\maketitle
		\textbf{Keywords:}  Third grade fluid, Navier-slip boundary conditions,	Stochastic	PDE, Optimal control, Necessary optimality condition,	Multiplicative	noise.\\[2mm]
	\hspace*{0.45cm}\textbf{MSC:} 35R60,	49K20,	76A05, 76D55,	60H15. \\
		

\tableofcontents              
     \section{Introduction}

The purpose of this article is to control the random  velocity field $y$   of a non-Newtonian fluid, which fills a  bounded domain  $D\subset\mathbb{R}^d, d=2,3$  with smooth boundary, and moves under the action of  random forces.
More precisely, we consider a tracking  problem and the goal is  to minimize cost functionals of the following form
$
\mathcal{J}_M(U,y)=f_M(y,\tau^U_M)+ \dfrac{\lambda}{p}\E\int_0^T\Vert U\Vert_{(¨H^1(D))^d}^pdt,	
$
where  $f_M$ is	a	given	functional of the physical state $y$, which lives up to a certain stopping time $\tau^U_M$, $\E$ denotes the average over all possible outcomes of the physical system, the constante $\lambda>0$
sets the intensity of the cost and  $p>2(d+1)$. 
The control acts through the external force $U,$ and the random velocity field $y$ is constrained to satisfy
the incompressible third grade fluid equations driven by a multiplicative Wiener-noise,	see	\eqref{I}	and	Section \ref{Sec2} for the precise mathematical framework. 

Most studies on fluid dynamics have been devoted to Newtonian fluids, which are characterized by the classical Newton's  law of viscosity.
However, there exist many real and industrial fluids  with nonlinear viscoelastic behavior  that does not obey Newton's law of viscosity, and consequently  cannot be described by the classical  viscous Newtonian fluids model. Among these fluids, we can find  natural biological fluids, geological flows,  industrial oils, fluids arising in food processing and cosmetic production,  and many others,	see \textit{e.g} \cite{DR95,FR80}. Therefore, it is necessary to consider  more general fluid models. Recently, the class of non-Newtonian fluids of differential type has received a special attention  since it	includes	second	grade	fluids,	which can be related to Camassa and Holm equations, possessing interesting geometric properties.
The  third	grade	fluids	correspond	to	the	next	step	of	modeling fluids	of differential type,	which includes	larger	class	of	fluids,	see \textit{e.g.}	\cite{Busuioc, Hol-Mar-Rat-98}. 
Consequently,  the  mathematical analysis of third grade fluids equations should be relevant to  predict and control the  behavior of these fluids, in order  to design optimal flows that can be successfully used and applied in the industry. 
From mathematical point of view, 
 third grade fluids constitute an hierarchy of fluids with  increasing complexity and more  nonlinear terms, which is  more complex  and require more involved analysis.

Without	exhaustiveness,	 the third grade fluids equations  with Dirichlet  boundary condition were  studied  in  \cite{AC97, SV95}.	Later on \cite{Bus-Ift-1, Bus-Ift-2},  
supplementing the equation with a Navier-slip boundary condition, the authors  
established the existence of a global solution for initial conditions in $H^2$
and proved that uniqueness holds in 2D.
  In \cite{ Cip-Did-Gue},  the authors extended  the later deterministic results  to  stochastic models in 2D. Recently, we proved  in \cite{yas-fer-2}  the existence and uniqueness of local adapted		solution to the stochastic third grade fluids equations with Navier-slip boundary conditions.  The sample paths take values in the Sobolev space $H^3$		and	are defined	 up to a certain positive stopping time,	in	2D	and	3D.

The control problems  for Newtonian fluids  (first	grade	fluids), where the flows are   described by  Navier-Stokes  equations, have been extensively studied in the literature, let us refer for instance   \cite{Abergel90,Delos-Griesse, Hinze-Kunisch}. 
 We emphasize that directing the velocity field to a desired velocity profile over time,  through a tracking-type cost functional,  has a wide range of applications.
To the best of our knowledge, the control problem for the second grade fluids (differential type) has been addressed for the first time in \cite{Arada-Cip}, where the authors proved the existence of a 2D-deterministic  optimal control and deduced the first order optimality conditions. Later	on, the control problem for 2D-stochastic second grade fluid models have been studied in \cite{Chem-Cip-2, CP19}.  Recently	in \cite{yas-fer},  the authors  tackled the control problem 
for 2D-deterministic third grade fluids. Namely, they showed 
the existence of an optimal solution, they deduced the 
 first order optimality conditions and proved the  uniqueness  of the
coupled system.

It is worth mentioning that the majority	of	the works in the literature considered the	optimal	control	problem	in	2D. 
The 3D-problem is much more difficult due to the fact that 
 the existence and uniqueness of the solution of  Newtonian	or	Non-Newtonian	fluid equations 
holds just  locally in time (defined	only up	to	a	certain	stopping	time	in	the	stochastic	case).	Taking	into	account	 the existence and uniqueness result of the solution,	
we 	define	the 	cost	functional	including		the	stopping	time	and	obtain	a	nonconvex	optimization	problem,	see	\cite{Benner-Trautwein2019}.	In	order	to	establish	the	existence	of	the optimal	control	and	to derive the	optimality	system,	it is necessary to	analyze	the	dependence	of	the	cost	functional	on	the	stopping	time, which is not an easy issue.
In	\cite{Benner-Trautwein2019},	the	authors	proved	the	well-posedeness	of	local	mild	solution	to	the	stochastic	Navier-Stokes	equations	with	multiplicative	Levy	noise.	Then,	they	proved	the	existence	and	uniqueness	of	optimal	control,	by	analyzing	some	cost	functional	depending	on	stopping	time.	Recently,	they	studied	in	\cite{Benner-Trautwein2021},	the	optimality	system	and	established	the	necessary	and	sufficient	optimality	conditions	with	a	multiplicative	noise	driven	by	a		$Q$-Wiener	process.	It is  worth mentioning that	the approach in	\cite{Benner-Trautwein2019,Benner-Trautwein2021}	was	based	on	 the theory of semigroups.	
   
Here, we aim to control stochastic non-Newtonian fluids is 2D and 3D.
Namely, we will solve a	stochastic	optimal	control	problem	constrained	by	the stochastic third	grade	fluid	equations. We will 
show the existence of the optimal control,	and	establish	the	optimality	system	via	variational	approach.
 The stochastic third grade fluid equations is highly	nonlinear which requires a subtil analysis, as well as the corresponding stochastic linearized and backward stochastic adjoint equations.	Moreover,	we should  say that our work improves substantially the 2D result obtained in	\cite{Chem-Cip-2} in the case $\beta=0$,	since  one	can	perform	the	analysis	of	the	optimality system
	without the special weights used in \cite{Chem-Cip-2}. Finally,  we show  that a similar  analysis applies  to derive the optimality system for cost functional including velocity field derivatives, which
	can be relevant to control the turbulence.

The article is organized as follows:
in Section \ref{Sec2}, we	precise	the	mathematical	framework	and	the	assumptions	on	the	data.	Then,	we	recall	the	defintion	of	stochastic	local	solution	and	a	result	about	the		existence	and	uniqueness.	Section	\ref{Sec3}	is	devoted	to	establish	some	stability	results.	In	Section	\ref{Sec4},	we	formulate	the	optimal	control	problem	and	we	state	the	main	results	of	this	work.	Then,	we	show	the	existence	and	uniqueness	of	optimal	solution.	Section \ref{Sec5} is devoted to show the existence and the uniqueness of the solution to the stochastic linearized state equation.	In	section	\ref{Sec6},	we  show that	the	G\^ateaux	derivative		of the control-to-state mapping	concides	with	the	solution	of	the	linearized	equations.	Then,	we	write	the	variation	of	the	cost	functional.	In	Section	\ref{Sec7},	we	study the	backward stochastic 	adjoint	equations	in	2D	and	3D,		where	the	uniqueness	holds	only	in	2D.		Section	\ref{Section-Optimality-condition}	concernes	the		proof	of	 a duality relation between the solution of the linearized equation and the adjoint state.	Then, we deduce the first order optimality  condition. We propose in Section \ref{Sec-V-control} the extension of our analysis to the case of cost functional including derivatives. 	Finally,	we gather in Section \ref{Technical-results} some technical lemmas, that we used repeatedly in our analysis.	
\section{Content of the study}\label{Sec2}
Let	$\mathcal{W}$ be a cylindrical  Wiener process defined on a complete probability space	$(\Omega,\mathcal{F},P)$,  endowed with  the right-continuous filtration $\{\mathcal{F}_t\}_{t\in[0,T]}$ generated by $\{\mathcal{W}(t)\}_{t\in [0,T]}.$ We assume that $\mathcal{F}_0$ contains all the P-null subset of $\Omega$ (see Subsection \ref{Noise-section} for the assumptions on the noise).
The goal  is to study the optimal  control of a third grade fluid where the control is introduced via the  external forces. The fluid  fills a bounded  and simply connected domain $D  \subset  \mathbb{R}^d,\; d=2,3,$ with regular boundary $\partial D$, and 
its dynamics is 
governed by the following equations

\begin{align}\label{I}
\begin{cases}
d(v(y))=\big(-\nabla \textbf{P}+\nu \Delta y-(y\cdot \nabla)v(y)-\sum_{j}v^j(y)\nabla y^j+(\alpha_1+\alpha_2)\text{div}(A^2) \hspace*{-3cm}&\\
\hspace*{2cm}+\beta \text{div}(|A|^2A)+U\big)dt+ G(\cdot,y)d\mathcal{W} \quad &\text{in } D \times (0,T)\times\Omega,\\
\text{div}(y)=0 \quad &\text{in } D \times (0,T)\times\Omega,\\
y\cdot \eta=0, \quad \left(\eta \cdot \mathbb{D}(y)\right)\big\vert_{\text{tan}}=0  \quad &\text{on } \partial D \times (0,T)\times\Omega,\\
y(x,0)=y_0(x) \quad &\text{in } D \times\Omega,
\end{cases}
\end{align}
where $y:=(y^1,\dots, y^d)$ is the velocity of the fluid, $\textbf{P}$ 
is the pressure and $U$ corresponds to the  external force.
The operators   $v$, $A$, $\mathbb{D}$ are defined by
$v(y)=y-\alpha_1 \Delta y:=(y^1-\alpha_1 \Delta y^1,\dots,y^d-\alpha_1 \Delta y^d)$ 
and 
$ A:=A(y)=\nabla y+\nabla y^T=2\mathbb{D}(y)$.
The vector $\eta$ denotes the outward normal to the boundary $\partial D$ and  $u\vert_{\text{tan}}$
represents the tangent component of a vector $u$ defined on  $\partial D$. 
In addition,  $\nu$ denotes  the viscosity of the fluid
and  $\alpha_1,\alpha_2$, $\beta$  
are material moduli satisfying 
\begin{align}\label{condition1}
\nu \geq 0, \quad \alpha_1\geq 0, \quad |\alpha_1+\alpha_2 |\leq \sqrt{24\nu\beta}, \quad  \beta \geq 0.
\end{align}
It is worth noting that  \eqref{condition1} allows the motion of the fluid to be compatible with thermodynamic laws,	namely	it	ensures		that the Helmholtz free energy density be at a minimum value when the
fluid is locally at rest,	we	refer	to	\cite{FR80}	for	more	details.
The diffusion coefficient $G$ will be specified in Subsection \ref{Noise-section}.
\subsection{Functional spaces and notations}
For a Banach space $E$, we define
$$  (E)^k:=\{(f_1,\cdots,f_k): f_l\in E,\quad l=1,\cdots,k\}\;\text{ for	positive integer }	k.$$
Let us introduce the following spaces:
\begin{equation}
\begin{array}{ll}
H&=\{ y \in (L^2(D))^d \,\vert \text{ div}(y)=0 \text{ in } D \text{ and } y\cdot \eta =0 \text{ on } \partial D\}, \nonumber\\[1mm]
V&=\{ y \in (H^1(D))^d \,\vert \text{ div}(y)=0 \text{ in } D \text{ and } y\cdot \eta =0 \text{ on } \partial D\}, \\[1mm]
W&=\{ y \in V\cap (H^2(D))^d\; \vert\, (\eta \cdot 
\mathbb{D}(y))\big\vert_{\text{tan}} =0 \text{ on } \partial D\},\quad
\widetilde{W}=(H^3(D))^d\cap W. \nonumber
\end{array}
\end{equation}

First, we recall the Leray-Helmholtz projector $\mathbb{P}: (L^2(D))^d \to H$, which is a linear bounded operator characterized by the following $L^2$-orthogonal decomposition
$v=\mathbb{P}v+\nabla \varphi,\;  \varphi \in H^1(D). $

Now, let us introduce  the scalar product between two matrices  $
A:B=tr(AB^T)$
and denote $\vert A\vert^2:=A:A.$
The divergence of a  matrix $A\in \mathcal{M}_{d\times d}(E)$ is given by 
$(\text{div}(A)_i)_{i=1}^{i=d}=(\displaystyle\sum_{j=1}^d\partial_ja_{ij})_{i=1}^{i=d}¨. $
The space $H$ is endowed with the  $L^2$-inner product $(\cdot,\cdot)$ and the associated norm $\Vert \cdot\Vert_{2}$. We recall that
\begin{align*}
(u,v)&=\sum_{i=1}^d\int_Du_iv_idx, \quad  \forall u,v \in (L^2(D))^d,\quad
(A,B)&=\int_D A: Bdx ; \quad  \forall A,B \in \mathcal{M}_{d\times d}(L^2(D)).
\end{align*}
On the functional spaces $V$, $W$ and  $\widetilde{W}$,
we will consider  the following inner products 
\begin{equation}
\begin{array}{ll} 
(u,z)_V&:=(v(u),z)=(u,z)+2\alpha_1(\mathbb{D}u,\mathbb{D}z),\\[1mm]
(u,z)_W&:=(u,z)_V+(\mathbb{P}v(u),\mathbb{P}v(z)),\\[1mm]
(u,z)_{\widetilde{W}}&:=(u,z)_V+(\text{curl}v(u),\text{curl}v(z)),
\end{array}
\end{equation}
and denote by $\Vert \cdot\Vert_V,\Vert \cdot\Vert_W$ and $\Vert \cdot\Vert_{\widetilde{W}}$ the corresponding norms. The usual norms on the classical Lebesgue and Sobolev spaces $L^p(D)$ and $W^{m,p}(D)$ will be denoted by   $\|\cdot \|_p$ and 
$\|\cdot\|_{W^{m,p}}$, respectively.
In addition, given a Banach space $X$, we will denote by $X^\prime$ its dual.\\

For the sake of simplicity, we do not distinguish between scalar, vector or matrix-valued   notations when it is clear from the context. In particular, $\Vert \cdot \Vert_E$  should be understood as follows
\begin{itemize}
	\item $\Vert f\Vert_E^2= \Vert f_1\Vert_E^2+\cdots+\Vert f_d\Vert_E^2$ for any $f=(f_1,\cdots,f_d) \in (E
	)^d$.
	\item $\Vert f\Vert_{E}^2= \displaystyle\sum_{i,j=1}^d\Vert f_{ij}\Vert_E^2$ for any $f\in \mathcal{M}_{d\times d}(E)$.
\end{itemize}

Throughout the article,  we consider $Q= D\times [0,T], \quad \Omega_T=\Omega\times [0,T],$  and denote by $C,C_i, i\in \mathbb{N}$,   generic constants, which may vary from line to line.\\
$\quad$	The  results on the  following modified Stokes problem will be very usefull to  our analysis
\begin{align}\label{Stokes}
\begin{cases}
h-\alpha_1\Delta h+\nabla \textbf{P}=f, \quad 
\text{div}(h)=0 \quad &\text{in } D,\\
h\cdot \eta=0, \quad (\eta \cdot \mathbb{D}(h))\big\vert_{\text{tan}}=0  \quad &\text{on } \partial D.	\end{cases}
\end{align}
The solution $h$ will be denoted by   $h=(I-\alpha_1\mathbb{P}\Delta)^{-1}f$.
We recall the existence and the uniqueness results, as well as the regularity of the solution $(h,\textbf{P})$. 
\begin{theorem}(\cite[Theorem 3]{Busuioc})
Suppose that  $f \in (H^m(D))^d,\, m \in \mathbb{N}$. Then there exists a unique (up to a constant for $\textbf{P}$) solution 	$(h,\textbf{P}) \in (H^{m+2}(D))^d\times H^{m+1}(D)$  of the Stokes problem \eqref{Stokes} such that
	$$ \Vert h \Vert_{H^{m+2}}+\Vert \textbf{P} \Vert_{H^{m+1}} \leq C(m)\Vert f\Vert_{H^m},\;\text{	
		where	}  C(m)\text{	is a positive constant.	}$$
\end{theorem}

We also consider the trilinear form 
$$
b(\phi,z,y)=(\phi\cdot \nabla z,y)=\int_D(\phi\cdot \nabla z)\cdot y dx, \quad \forall \phi,z,y \in (H^1(D))^d,$$
which verifies $b(y,z,\phi)=-b(y,\phi,z),\quad\forall y \in V; \forall z,\phi \in (H^1(D))^d$.


\subsection{The stochastic setting}\label{Noise-section}

Let us consider a cylindrical Wiener process $\mathcal{W}$ 
defined on  $(\Omega,\mathcal{F},P)$, which can be written as 
$\mathcal{W}(t)= \sum_{\k \ge 1} e_\k \beta_\k(t),$
where $(\beta_\k)_{\k\ge 1}$ is a sequence of  mutually independent real valued standard Wiener processes and $(e_\k)_{\k\ge 1}$ is a complete orthonormal system in a separable Hilbert space $\mathbb{H}$.
Recall that 
the sample paths  of $\mathcal{W}$ take values 
in a larger Hilbert space $H_0$ such that   $\mathbb{H}\hookrightarrow H_0$ defines a Hilbert–Schmidt
embedding. For example, the space $H_0$ can be defined as follows  
$$ H_0=\bigg\{ u=\sum_{\k \ge 1}\gamma_{\k}e_\k\;\vert\;  \sum_{ \k\geq 1} \dfrac{\gamma_\k^2}{\k^2} <\infty\bigg\}
\text{	endowed with the norm	}
 \Vert u\Vert_{H_0}^2=\sum_{ \k\geq 1} \dfrac{\gamma_\k^2}{\k^2}, \;  u=\sum_{\k \ge 1}\gamma_ke_\k.$$
Hence,  $P$-a.s. the trajectories of $\mathcal{W}$ 
belong to the space $C([0,T],H_0)$ (cf. \cite[Chapter 4]{Daprato}).\\

In order to define the stochastic integral in the infinite dimensional framework, let us consider  another Hilbert space $E$ and denote by $L_2(\mathbb{H},E)$ the space of Hilbert-Schmidt operators from $\mathbb{H}$ to $E$, which is the subspace of the linear operators  defined as follows
$$L_2(\mathbb{H},E):=\bigg\{ G:\mathbb{H} \to E \;\vert \quad  \Vert G\Vert^2_{L_2(\mathbb{H},E)}:=\sum_{\k \ge 1}\Vert G \e_\k \Vert_E^2 < \infty\bigg\}.$$
Given a $L_2(\mathbb{H},E)-$valued predictable \footnote{$\mathcal{P}_{T}:=\sigma(\{ ]s,t]\times F_s \vert 0\leq s < t \leq T,F_s\in \mathcal{F}_s \} \cup \{\{0\}\times F_0 \vert F_0\in \mathcal{F}_0 \})$ (see \cite[p. 33]{Liu-Rock}). Then, a process defined on $\Omega_T$ with values in a given space $E$ is predictable if it is $\mathcal{P}_{T}$-measurable.} process $G\in L^2(\Omega;L^2(0,T;L_2(\mathbb{H},E)))$, and taking $\sigma_\k=Ge_\k$, we may define the It\^o stochastic integral by
$ \displaystyle\int_0^tGd\mathcal{W}=\sum_{\k\ge 1} \int_0^t\sigma_\k d\beta_k, \; \forall t\in [0,T].$\\
Next, we will precise the assumptions on the data.
\subsection{Definition of the diffusion coefficient and assumptions}
Let us consider 
a family of Carath\'eodory functions
 $\sigma_\k: [0,T]\times \R^d\mapsto \R^d, \; \k \in\mathbb{N},$  
  which  satisfies
$\sigma_\k(t,0)=0$,  and   there exists $L > 0$  such that for a.e. $t\in (0,T)$,  
\begin{align}
\label{noise1}
\quad &\sum_{\k\ge 1}\big| \sigma_\k(t,\lambda)-\sigma_\k(t,\mu)\big|^2  \leq L  |\lambda-\mu|^2;	\quad	\forall	\lambda,\mu \in \R^d.
\end{align}
In addition, we assume that there exists a sequence $(a_k)\subset \mathbb{R}_0^+$  such that
\begin{align}
\label{noiseV}
\vert \nabla \sigma_\k(\cdot,\cdot)\vert  \leq a_k, \quad \sum_{\k\ge 1} a_k^2 <\infty,
\end{align}
and $\forall \mu\in \mathbb{R}^d$, the mapping
$
(t,\lambda) \to \nabla \sigma_\k(t,\lambda) \mu
$
is also a Carath\'eodory function.

We notice that, in particular, \eqref{noise1} gives 
$\displaystyle \,\mathbb{G}^2(t,\lambda):= \sum_{\k\ge 1}
 |\sigma_\k(t,\lambda)|^2\le L\,|\lambda|^2.
$ 

For each $t\in[0,T]$ and $y\in V$, we introduce the  Hilbert-Schmidt operator
$$
G(t,y): \mathbb{H}\goto (H^1(D))^d, \qquad G(t,y)e_\k= \{ x \mapsto \sigma_\k\big(t,y(x)\big)\},
\quad \k \ge 1.
$$
	Since $G$ satisfies $\sigma_\k=Ge_\k$,   the stochastic integral 
	$\displaystyle\int_0^tG(\cdot,y)d\mathcal{W}=\sum_{\k\ge 1}\int_0^t\sigma_\k(\cdot,y)d\beta_\k$ 
	is a well-defined $(\mathcal{F}_t)_{t\in [0,T]}$-martingale with values in	$(H^1(D))^d$ (resp. $(L^2(D))^d$).
\begin{itemize}
	\item[$\mathcal{H}_0:$] In addition, for any $t\in [0,T]$, we assume that $\sigma_\k(t,\cdot):\mathbb{R}^d\to \mathbb{R}^d$ is G\^ateaux differentiable, satisfying		for	any	$v\in	\mathbb{R}^d$
	\begin{align}\label{Gateaux-derivative-noise}
		\dfrac{1}{\delta}[\sigma_\k(t,y+\delta v)-\sigma_\k(t,y)]=\nabla_y\sigma_\k(t,y) v+R_\k^\delta(t, v), \quad \forall \k \geq 1,
\end{align}
with  $\vert R_\k^\delta(t, v)\vert \leq b_\k\vert v\vert \delta^{\gamma}, \gamma >0 $, where  $(b_k)\subset \mathbb{R}_0^+$ and $\sum_{ \k\geq 1}b_\k^2 <\infty$.\\
\end{itemize}
\begin{remark}
	The	additional assumption  $\mathcal{H}_0$ on
	 $(\sigma_\k)_\k$  is assumed to	obtain	the	necessary	optimality	condition	(see	Section	\ref{Sec6}	and	Section	\ref{Section-Optimality-condition}	).	
\end{remark}
\subsection{Assumptions on the   data.}
\begin{itemize}
		\item[$\mathcal{H}_1:$]
		we consider $y_0: \Omega \to \widetilde{W}$ and $ U:\Omega\times[0,T] \to (H^1(D))^d$ such that
		 \begin{align}
		&\bullet \; y_0\;\text{ is }\mathcal{F}_0\text{-measurable and}\; U\;\text{ is predictable}.\notag\\
			  &\bullet\; 			U\in L^p(\Omega\times (0,T); (H^1(D))^d),\quad y_0\in L^p(\Omega,\widetilde{W}),\quad 	\text{for fixed}\quad  p>2(d+1).\, \label{data-assumptions}
			\end{align}

		\end{itemize}


\begin{remark}
		The	 assumption $p>2(d+1)$ ensures uniform estimates for the solution of finite dimensional linearized and adjoint equations without the special weights  used in 
\cite{Chem-Cip-2}.
\end{remark}

Unless otherwise stated,	$p$	always	satisfies	$p>2(d+1)$	in	the	following.

\subsection{Existence of local strong solution}
Let us recall  the notion of the local solution and the pathwise uniqueness. 
This	subsection	is	based	on the  results from \cite{yas-fer-2}.
	\begin{definition}\label{Def-strong-sol-main}
		Let $(\Omega,\mathcal{F},(\mathcal{F}_t)_{t\geq 0},P)$ be a stochastic basis and $\mathcal{W}(t)$ be a $(\mathcal{F}_t)$-cylindrical Wiener process. We say that 
		a pair $(y,\tau)$  is a  local strong (pathwise) solution to \eqref{I} iff:	
		\begin{itemize}
			\item $\tau$ is an a.s. strictly positive $(\mathcal{F}_t)$-stopping time.
			\item The velocity $y$ is a $W$-valued predictable process satisfying
			$$y(\cdot \wedge \tau) \in L^p(\Omega;\mathcal{C}([0,T];W^{2,4}(D)))\cap L^p_w(\Omega;L^\infty(0,T;\widetilde{W}))	 \footnote{$L^p_w(\Omega;L^\infty(0,T;\widetilde{W}))=\{ 	u:\Omega\to	L^\infty(0,T;\widetilde{W})	\text{	is	 weakly-* measurable		and	}	E\Vert	u\Vert_{L^\infty(0,T;\widetilde{W})}^p<\infty\}.$}
			, \quad \text{	for	}	p > 4.$$
			 			\item  $P$-a.s. for all $t\in [0,T]$
			\begin{align*}
			&(y(t\wedge \tau),\phi)_V=(y_0,\phi)_V+\displaystyle\int_0^{t\wedge \tau}\big(\nu \Delta y-(y\cdot \nabla)v(y)-\sum_{j}v(y)^j\nabla  (y)^j
			+(\alpha_1+\alpha_2)\text{div}[A(y)^2]
			\nonumber\\
			&\hspace*{4cm}
			+\beta\text{div}[|A(y)|^2A(y)]+ U,\phi\big)
			ds +\displaystyle\int_0^{t\wedge \tau}(G(\cdot,y),\phi)d\mathcal{W}	\text{ for all } \phi\in V. \nonumber 
			\end{align*}
		\end{itemize} 
	\end{definition}

	\begin{definition}
	\label{def-uniq}
		\begin{itemize}
			\item[i)] 	We say that local pathwise uniqueness holds if for any   given pair $(y^1,\tau^1), (y^2,\tau^2)$ of local strong solutions of \eqref{I},
			 we have $$ P\big(y^1(t)=y^2(t);\; \forall t\in [0,\tau^1\wedge \tau^2]\big)=1.$$
			\item[ii)]  We say that  $((y_M)_{M\in \mathbb{N}}, (\tau_M)_{M\in \mathbb{N}},\mathbf{t})$ is a maximal strong local solution to \eqref{I} if  for each $M \in \mathbb{N}$, the pair $(y_M,\tau_M)$ is a local strong solution, $(\tau_M)_M$ is an increasing sequence of stopping times such that  
			$\mathbf{t}:=\displaystyle\lim_{M\to \infty} \tau_M >0, \quad \text{P-a.s.}$ 	and
						\begin{align}\label{maximal-stopping}
		\sup_{t\in [0,\tau_M]}\Vert y(t)\Vert_{W^{2,4}} 
			\geq M \text{ on }\quad  \{\mathbf{t} <T\}, \quad \forall M\in \mathbb{N},\;	\text{ P-a.s.	}
			\end{align}
		\end{itemize}
	\end{definition}
\begin{theorem}\label{main-thm-exist}
	There exists a unique maximal strong (pathwise) local solution to \eqref{I}	with
		\begin{align}\label{stopping-time}
	\tau_M=\inf\{ t\geq 0: \Vert  y(t)\Vert_{W^{2,4}} \geq M\} \wedge T; \quad M\in \mathbb{N}.\end{align}
\end{theorem}

\begin{remark}\label{Rmq-y-yM}
	We	recall that   the local strong solution  of \eqref{I} have been  constructed by using $y_M$ (see \cite[Lemma 5.2]{yas-fer-2}).  In	other	words, 
	   we have: $\forall	M\in \mathbb{N}: y(t):=y_M(t), \; \forall  t\in [0,\tau_M], \,\text{  P-a.s.} $
\end{remark}
Following \cite{yas-fer-2}, we have the following 
estimates for  the solution of \eqref{I}.
\begin{prop}\label{Estimate-y-proposition}
	There exists  $K:=K(L,M,T,\Vert  y_0\Vert_{L^p( \Omega;\widetilde{W})}, \Vert  U\Vert_{L^p( \Omega\times[0,T];(H^1(D))^d})>0	\text{	such that 	}$ 
	\begin{align*}
	&\E \sup_{s\in [0,\tau_M]} \Vert y\Vert_V^2+4\nu\E \int_0^{\tau_M}\Vert D  y\Vert_{2}^2dt+\dfrac{\beta}{2}\E\int_0^{\tau_M}\int_D| A(y)|^4dxdt \leq e^{cT} ( \E\Vert y_{0}\Vert_V^2	+\E\int_0^{T}\Vert U\Vert_2^2dt),\\
	&\E\sup_{s\in [0,\tau_M]}\Vert  y\Vert_{\widetilde{W}}^2:=	\E\sup_{s\in [0,\tau_M]}[\Vert \text{curl }v(y)\Vert_2^2+\Vert y\Vert_V^2]\leq K,\\
	&\E \sup_{ [0,\tau_M]} \Vert y \Vert_{\widetilde{W}}^{p} \leq  K(M,T,p)(1+\E\Vert y_0\Vert_{\widetilde{W}}^{p}+\E\int_0^T\Vert U\Vert_2^{p}ds+ \E\int_0^T\Vert \text{ curl }  U\Vert_2^{p} ds), \quad \forall p >2.
	\end{align*}
\end{prop}

\section{Stability results }\label{Sec3}
In this section, we assume that the initial data $y_0$ and the force $U$ satisfy $\mathcal{H}_1$, and  consider 
 two  strong local solutions
 $(y_1,\tau_M^1)$ and $(y_2,\tau_M^2)$  to \eqref{I}	in the sense of Definition \ref{Def-strong-sol-main} with  the initial conditions $y_0^1,y_0^2$ and the forces $U_1,U_2$, respectively. In addition, we denote $y=y_1-y_2, y_0=y_0^1-y_0^2$ and $U=U_1-U_2$.
 
  First, we state a result, which follows from   \cite[Lem. 5.1]{yas-fer-2}.
\begin{lemma}\label{Lemma-V-stability}  For	all	$ p\geq 2${\color{blue},}  there exists $C(M,L,T,p)>0$ such that
	\begin{align*}
	\E\sup_{s\in [0, \tau_M^1\wedge\tau_M^2]}\Vert y_1(s)-y_2(s)\Vert_V^p
	\leq C(M,L,T,p)\big[\E\Vert y_0^1-y_0^2\Vert_{V}^p+\E\int_0^{\tau_M^1\wedge\tau_M^2}\Vert U_1(s)-U_2(s)\Vert_{2}^pds\big].
	\end{align*}
\end{lemma}
\begin{proof}
	 Let us consider $1\leq p< \infty$ and $t\in[0,T]$.
	For any 
	$s\in [0,t\wedge\tau_M^1\wedge \tau_M^2]$,  from the proof of \cite[Lemma 5.1]{yas-fer-2},  there exists $M_0>0$ such that 
\begin{align}\label{Stab-V-new}
\Vert y(s)\Vert_V^2+4\nu \int_0^s\Vert \mathbb{D}y\Vert_{2}^2dr&\leq \Vert y_0\Vert_{V}^2+M_0\int_0^s(\Vert y_1\Vert_{W^{2,4}}+\Vert y_2\Vert_{W^{2,4}}+1)\Vert  y\Vert_{V}^2dr+\int_0^s\Vert U_1-U_2\Vert_{2}^2dr\nonumber\\
&\qquad+2\int_0^s(G(\cdot,y_1)-G(\cdot,y_2), y)d\mathcal{W}.
\end{align}
Therefore, we have
\begin{align}
\label{7:1:23}
\Vert y(s)\Vert_V^{2p}&\leq C(p,\tau_M^1\wedge \tau_M^2)\big\{ \Vert y_0\Vert_{V}^{2p}+M_0\int_0^s(\Vert y_1\Vert_{W^{2,4}}+\Vert y_2\Vert_{W^{2,4}}+1)^p\Vert  y\Vert_{V}^{2p}dr
+\int_0^s\Vert U\Vert_{2}^{2p}dr\nonumber\\
&\qquad+\vert\int_0^s(G(\cdot,y_1)-G(\cdot,y_2), y)d\mathcal{W}\vert^p\big\},
\quad \forall s\in [0,t\wedge\tau_M^1\wedge \tau_M^2].
\end{align}
Using Burkholder-Davis-Gundy inequality, 	we		infer that
\begin{align*}
\E\sup_{s\in [0,  t\wedge\tau_M^1\wedge\tau_M^2]}\Vert y(s)\Vert_V^{2p}&\leq C(p,L,T)\{  \E\Vert y_0\Vert_{V}^{2p}+\E\int_0^{ t\wedge\tau_M^1\wedge\tau_M^2}\Vert U\Vert_{2}^{2p}ds\\&\qquad+M_0(2M+1)^p\int_0^{ t}
\E\sup_{s\in [0,  r\wedge\tau_M^1\wedge\tau_M^2]}\Vert y(s)\Vert_V^{2p}dr\}.
\end{align*}
Finally, Gronwall's inequality ensures  the result of Lemma \ref{Lemma-V-stability}.
\end{proof}
In order to  derive the necessary optimality condition, we need 
 a stability result with respect to $W$-norm. This is the aim of the next proposition.



\begin{prop}\label{local-stability-thm}

	There exists  $K(M,\epsilon,p)>0$, which depends only on the data such	that 
\begin{align}\label{stability-estimate}
\E\sup_{s\in [0, \tau_M^1\wedge\tau_M^2]}\Vert (y_1-y_2)(s)\Vert_{W}^{2p}&\leq K(M,\epsilon) \big\{  \E\Vert y_0^1-y_0^2\Vert_W^{2p}+\E\int_0^{\tau_M^1\wedge\tau_M^2}\Vert U_1-U_2\Vert_{2}^{2p} ds\\
&\qquad+ \Vert y_1\Vert_{L^{2p(d+1+\epsilon)}(\Omega\times(0,\tau_M^1\wedge\tau_M^2);(H^3)^d)}^{2p(d+\epsilon)}\Vert y\Vert_{L^{2p(d+1+\epsilon)}(\Omega\times(0,\tau_M^1\wedge\tau_M^2);V)}^{2p}\big\},\nonumber
\end{align}
where $p\in [1,\infty[$ and $\epsilon\in ]0,1]$.
\end{prop}
\begin{proof}
	For any $t\in [0,\tau_M^1\wedge \tau_M^2]$,	we have 
	\begin{align}\label{Ito-No-appro}
	v(y(t))-v(y_0)=&-\int_0^t\nabla (p_1-p_2)ds+\nu \int_0^t\Delta y-[(y_1\cdot \nabla)v(y_1)-(y_2\cdot \nabla)v(y_2)]ds\\
	&\hspace*{-2.5cm}-\displaystyle\sum_{j=1}^d\int_0^t[v(y_1)^j\nabla y_1^j-v(y_2)^j\nabla y_2^j]ds+(\alpha_1+\alpha_2)\int_0^t[\text{div}(A(y_1)^2)-\text{div}(A(y_2)^2)]ds
	\nonumber\\
	&\hspace*{-2.5cm}+\beta\int_0^t[\text{div}(|A(y_1)|^2A(y_1))-\text{div}(|A(y_2)|^2A(y_2))]ds+\int_0^tUds+ \int_0^t[G(\cdot,y_1)-G(\cdot,y_2)]d\mathcal{W}(s)\nonumber. 
	\end{align}
From	\eqref{Stab-V-new}	we	have
\begin{align}\label{stability-V-y1y2}
d\Vert y\Vert_V^2 +4\nu \Vert D(y)\Vert_2^2 dt 
\leq K(1+\Vert y_1\Vert_{W^{2,4}}^2+\Vert y_2\Vert_{W^{2,4}}^2)\Vert y\Vert_{W}^2dt
+2(G(\cdot,y_1)-G(\cdot,y_2), y)d\mathcal{W}+\Vert U\Vert_{2}^2 dt.
\end{align}	
	On	the	other	hand,	
	applying the operator $\mathbb{P}$ to \eqref{Ito-No-appro} and  using It\^o formula,
	 we derive
			\begin{align*}
	&d\Vert\mathbb{P}v(y)\Vert_{2}^2=2\nu\big( \Delta(y),\mathbb{P}v(y)\big)dt -2\big([(y_1\cdot \nabla)v(y_1)-(y_2\cdot \nabla)v(y_2)],\mathbb{P}v(y)\big)dt\\&\quad-2\big(\sum_{j}[v(y_1)^j\nabla y_1^j-\theta_v(y_2)^j\nabla y_2^j],\mathbb{P}v(y)\big)dt
+\big(\alpha_1+\alpha_2)([\text{div}(A(y_1)^2)-\text{div}(A(y_2)^2)],\mathbb{P}v(y)\big)dt\\
	&\quad+2\beta\big( [\text{div}(|A(y_1)|^2A(y_1))-\text{div}(|A(y_2)|^2A(y_2))],\mathbb{P}v(y)\big)dt+2\big(U,\mathbb{P}v(y)\big)dt\\
	&\quad+\sum_{ \k\geq 1}\Vert \mathbb{P}[\sigma_\k(\cdot,y_1)-\sigma_\k(\cdot,y_1)]\Vert_2^2dt+2\big(G(\cdot,y_1)-G(\cdot,y_2),\mathbb{P}v(y)\big)d\mathcal{W}. 
	\end{align*}
	Let us  estimate the terms in the right hand side of this equation. First, note	that
	\begin{align*}
	2\nu\big( \mathbb{P}\Delta(y),\mathbb{P}v(y)\big)
	\leq \dfrac{-2\nu}{\alpha_1}\Vert\mathbb{P}v(y)\Vert_{2}^2+\dfrac{2\nu}{\alpha_1}\Vert y\Vert_W^2
\quad\text{and}\quad 	2\big(U,\mathbb{P}v(y)\big) \leq \Vert U\Vert_2^2+\Vert y\Vert_{H^2}^2.
	\end{align*}
	The properties of the projection $\mathbb{P}$ and \eqref{noise1} give
	$\sum_{ \k\geq 1}\Vert  \mathbb{P}[\sigma_\k(\cdot,y_1)-\sigma_\k(\cdot,y_1)]\Vert_2^2 \leq L \Vert  y_1-y_2\Vert_{2}^2.
	$	Next, Sobolev embedding theorem (see e.g \cite[Ch. 1]{Roubicek}) and \cite[Lem. 5]{Bus-Ift-2}  ensure that		\begin{align*}
	&\big(\{(y_1\cdot \nabla)v(y_1)-(y_2\cdot \nabla)v(y_2)\},\mathbb{P}v(y)\big)=[b(y,v(y_1),\mathbb{P}v(y))+b(y_2,v(y)-\mathbb{P}v(y),\mathbb{P}v(y))]\\
	&\leq \Vert y\Vert_\infty \Vert y_1\Vert_{H^3}\Vert y\Vert_{H^2}+C\Vert y_2\Vert_{\infty}\Vert y\Vert_{H^2}\Vert y\Vert_{H^2}  \leq \Vert y\Vert_\infty \Vert y_1\Vert_{H^3}\Vert y\Vert_{H^2}+C\Vert y_2\Vert_{W^{2,4}}\Vert y\Vert_{W}^2,
\\
&\text{and	}\\
&\sum_{j=1}^d\big([v(y_1)^j\nabla y_1^j-v(y_2)^j\nabla y_2^j],\mathbb{P}v(y)\big)=b(\mathbb{P}v(y),y_2,v(y))+b(\mathbb{P}v(y),y,v(y_1))\\
	&\quad\leq \Vert y\Vert_{H^2}\Vert y_2\Vert_{W^{1,\infty}}\Vert y\Vert_{H^2}+\Vert y\Vert_{H^2}\Vert y\Vert_{W^{1,4}}\Vert y_1\Vert_{W^{2,4}}\leq K(\Vert y_1\Vert_{W^{2,4}}+\Vert y_2\Vert_{W^{2,4}})\Vert y_1-y_2\Vert_{W}^2.
	\end{align*}
	On the other hand,  using the embeddings	$H^2(D) \hookrightarrow W^{1,4}(D)$, 	
$	W^{2,4}(D)\hookrightarrow	W^{1,\infty}(D)$, we show the existence of  
	$K>0$	such	that
	\begin{align*}
	&\big(\text{div}(A(y_1)^2)-\big(\text{div}(A(y_2)^2),\mathbb{P}v(y)\big)=\big( \text{ div}(A(y) A(y_1))+\text{ div}(A(y_2)A(y)),\mathbb{P}v(y))\big)\\
	&\leq K(\Vert y_1\Vert_{W^{2,4}}+\Vert y_2\Vert_{W^{2,4}})\Vert y\Vert_{W}^2.
	\end{align*}

The H\"older inequality and the embeeding	$W^{2,4}(D)\hookrightarrow	W^{1,\infty}(D)$ yield
 	\begin{align*}
	&\big(\text{div}(|A(y_1)|^2A(y_1))-\text{div}(|A(y_2)|^2A(y_2)),\mathbb{P}v(y)\big)\\
	&=\big(\text{div}(|A(y_1)|^2A(y)),\mathbb{P}v(y)\big)+\big(\text{div}([A(y_1):A(y)+A(y):A(y_2)]A(y_2)),\mathbb{P}v(y)\big)\\
		&\leq K(\Vert y_1\Vert_{W^{2,4}}^2+\Vert y_2\Vert_{W^{2,4}}^2)\Vert y\Vert_{W}^2.	\end{align*}	Hence
		\begin{align}\label{stability-W-eqn}
	\qquad&d\Vert\mathbb{P}v(y)\Vert_{2}^2+\dfrac{2\nu}{\alpha_1}\Vert\mathbb{P}v(y)\Vert_{2}^2dt\leq K(1+\Vert y_1\Vert_{W^{2,4}}^2+\Vert y_2\Vert_{W^{2,4}}^2)\Vert y\Vert_{W}^2 \nonumber\\
	&\qquad+2(G(\cdot,y_1)-G(\cdot,y_2), \mathbb{P}v(y))d\mathcal{W}+\Vert U\Vert_{2}^2 dt+ K\Vert y\Vert_\infty \Vert y_1\Vert_{H^3}\Vert y\Vert_{H^2} dt.
	\end{align}
	Gathering \eqref{stability-V-y1y2} and \eqref{stability-W-eqn},
	for any $t\in [0, \tau_M^1\wedge\tau_M^2]$, we deduce the following relation
	\begin{align}\label{take-p-power}
	&\Vert y(t)\Vert_{W}^2+\min(1,\dfrac{1}{2\alpha_1})4\nu \int_0^t\Vert y\Vert_W^2 ds
		\leq \Vert y_0\Vert_W^2+ K\int_0^t(1+\Vert y_1\Vert_{W^{2,4}}^2+\Vert y_2\Vert_{W^{2,4}}^2)\Vert y\Vert_{W}^2ds\nonumber\\
		&+2\vert \int_0^t(G(\cdot,y_1)-G(\cdot,y_2), \mathbb{P}v(y)+y)d\mathcal{W}\vert +2\int_0^t\Vert U\Vert_{2}^2 ds+K\int_0^t\Vert y\Vert_\infty \Vert y_1\Vert_{H^3}\Vert y\Vert_{H^2} ds.
	\end{align}
	Let $p\geq 1$, thanks to the Burkholder-Davis-Gundy inequality, for any $\delta>0$, we have
	\begin{align*}
2\E\sup_{s\in [0, t\wedge\tau_M^1\wedge\tau_M^2]}
&\Big\vert \int_0^s (G(\cdot,y_1)-G(\cdot,y_1), y+\mathbb{P}v(y))d\mathcal{W}\Big\vert^p
\\ &\leq \delta\E\sup_{s\in [0, t\wedge\tau_M^1\wedge\tau_M^2]}\Vert y\Vert_2^{2p}+\delta\E\sup_{s\in [0,t\wedge \tau_M^1\wedge\tau_M^2]}\Vert \mathbb{P}v(y)\Vert_2^{2p}+C_\delta\E\int_0^{t\wedge\tau_M^1\wedge\tau_M^2}\Vert y\Vert_2^{2p}dr,
\end{align*}

	Taking the $p^{th}$ power of \eqref{take-p-power} and the expectation, we write
		\begin{align*}
	&\E\sup_{s\in [0, t\wedge\tau_M^1\wedge\tau_M^2]}\Vert y(s)\Vert_{W}^{2p}\leq  \E\Vert y_0\Vert_W^{2p}+K(\delta)\E\int_0^{t\wedge\tau_M^1\wedge\tau_M^2}(1+\Vert y_1\Vert_{W^{2,4}}^2+\Vert y_2\Vert_{W^{2,4}}^2)^p\Vert y\Vert_{W}^{2p}ds\\
	&\quad+\delta\E\sup_{s\in [0,t\wedge \tau_M^1\wedge\tau_M^2]}\Vert y\Vert_W^{2p}+K(\delta)\E\int_0^{\tau_M^1\wedge\tau_M^2}\Vert U\Vert_{2}^{2p} ds+K\E\int_0^{t\wedge\tau_M^1\wedge\tau_M^2}(\Vert y\Vert_\infty \Vert y_1\Vert_{H^3}\Vert y\Vert_{H^2})^p ds.\nonumber
	\end{align*}
		Since $\Vert y_{i}(t)\Vert_{W^{2,4}}^{i=1,2} \leq M,\forall	t\in[0, \tau_M^1\wedge\tau_M^2]$. An  appropriate choice of $\delta$  and  Lemma \ref{interpolation-estimate-lem}	ensure
	\begin{align*}
&\E\sup_{s\in [0, t\wedge\tau_M^1\wedge\tau_M^2]}\Vert y(s)\Vert_{W}^{2p}\leq  \E\Vert y_0\Vert_W^{2p}+K(\delta,M,\epsilon)\E\int_0^{t\wedge\tau_M^1\wedge\tau_M^2}\Vert y\Vert_{W}^{2p}ds\\
&+K(\delta)\E\int_0^{\tau_M^1\wedge\tau_M^2}\Vert U\Vert_{2}^{2p} ds+ K(\epsilon)\Vert y_1\Vert_{L^{2p(d+1+\epsilon)}(\Omega\times(0,\tau_M^1\wedge\tau_M^2);(H^3)^d)}^{2p(d+\epsilon)}\Vert y\Vert_{L^{2p(d+1+\epsilon)}(\Omega\times(0,\tau_M^1\wedge\tau_M^2);V)}^{2p}.
\end{align*}
	Therefore, Gr\"onwall's lemma yields	\eqref{stability-estimate}.
\end{proof}
As a consequence of Proposition \ref{local-stability-thm} and Lemma \ref{Lemma-V-stability},  we  establish the following corollary:
\begin{cor}\label{cor-rest-0}
 	 Consider $U$ and $y_0$ satisfying \eqref{data-assumptions} and $\psi\in L^p((\Omega_T, \mathcal{P}_T);(H^1(D)) ^d)$. Defining
	$ U_\rho=U+\rho\psi, \quad \rho \in (0,1),$
	let $(y,\tau_M)$ and $(y_\rho,\tau_M^\rho)$ be  the solutions of \eqref{I} associated with $(U,y_0)$ and $(U_\rho,y_0)$, respectively. Then,  there exist $C,K >0$ such that 
	\begin{align*}
	\E\sup_{s\in [0, \tau_M\wedge\tau_M^\rho]}\Vert y_\rho(s)-y(s)\Vert_V^p
	&\leq C(M,L,T,p)\rho^p\E\int_0^{\tau_M\wedge\tau_M^\rho}\Vert \psi(s)\Vert_{2}^pds\leq C_\psi(M)\rho^p  , \quad \text{  for } p\geq 2,\\
	\E\sup_{s\in [0, \tau_M\wedge\tau_M^\rho]}\Vert (y_\rho-y)(s)\Vert_{W}^{q}&\leq K(M,\epsilon) \big\{\rho^{q}\E\int_0^{\tau_M\wedge\tau_M^\rho}\Vert \psi\Vert_{2}^{q} ds+ \Vert y\Vert_{L^{q_d}(\Omega_\tau;(H^3)^d)}^{q(d+\epsilon)}\Vert y_\rho-y\Vert_{L^{q_d}(\Omega_\tau;V)}^{q}\big\},
	\end{align*}
	for any $q\in [2,\infty[$ and $\epsilon\in ]0,1]$. In particular, for $q=2+\epsilon,$ we have
	\begin{align*}		
		\E\sup_{s\in [0, \tau_M\wedge\tau_M^\rho]}\Vert (y_\rho-y)(s)\Vert_{W}^{2+\epsilon}&\leq K(M,\epsilon) \big\{\rho^{2+\epsilon}\E\int_0^{\tau_M\wedge\tau_M^\rho}\Vert \psi\Vert_{2}^{2+\epsilon} ds+ C\Vert y_\rho-y\Vert_{L^{(2+\epsilon)(d+1+\epsilon)}(\Omega_\tau;V)}^{2+\epsilon}\big\}\\
			&\leq K(M,\epsilon)\rho^{2+\epsilon}	\text{	where	}	\Omega_\tau=\Omega\times(0,\tau_M\wedge\tau_M^\rho),\;	q_d=q(d+1+\epsilon).
	\end{align*}
\end{cor}
\section{The optimal control problem  and main results}\label{Sec4}
\subsection{Cost functional}
Our aim is to control the solution of the system \eqref{I} through a stochastic distributed force $U$, which belongs to the  admissible set $\mathcal{U}_{ad}^p$  defined as a  nonempty, bounded, closed and convex subset of $L^p((\Omega_T,\mathcal{P}_T);(H^1(D))^d)$, such that $0\in \mathcal{U}_{ad}^p$\footnote{For example, $\mathcal{U}_{ad}^p=\{ u \in L^p((\Omega_T,\mathcal{P}_T);(H^1(D))^d):=Y;\quad \Vert u\Vert_{Y} \leq K \}, \quad 0<K<\infty$. }. 

Let $(y(t,U))_{t\in [0,\tau^U_M)}$  be the  local pathwise solution of \eqref{I}. We recall that $ y:=y(t,U)=y_M(t,U)$,  $t\in [0,\tau_M^U]$, and   the stopping times $(\tau_M^U)_{M\in \mathbb{N}}$ depend on the control $U \in \mathcal{U}_{ad}^p$. Therefore, the solution $y(t,U)$ is defined up to a certain stopping time. Thus,  we introduce the cost functional $J_M:\mathcal{U}_{ad}^p\to \mathbb{R}$ defined by
\begin{align}\label{cost-exam}
J_M(U,y)=\dfrac{1}{2}\E\int_0^{\tau^U_M}\Vert y-y_d\Vert_2^2dt+ \dfrac{\lambda}{p}\E\int_0^T\Vert U\Vert_{(H^1(D))^d}^pdt,	\quad\lambda>0.
\end{align}
with a desired target field $y_d \in L^2(0,T;H)$\footnote{In	general,	$y_d$	is	smoother	with	respect	to	the	space	variables	and	satisfies	the	boundary	conditions.	Here,	we		assume	what	is	necessary	to	perform	the	analysis.}.
Our goal is  to solve the following problem
\begin{align}\label{problem}
\mathcal{P}\begin{cases}
\displaystyle\min_{U}
\{ J_M(U,y): \; U \in \mathcal{U}_{ad}^p \text{ and } y \text{ is the solution of \eqref{I} for the minimizing }  U  \}.
\end{cases}
\end{align}
In other words,  we intend to find $U_M^* \in \mathcal{U}_{ad}^p $ such that
 $J_M(U_M^*,y(U_M^*))= \displaystyle\min_{U\in \mathcal{U}_{ad}^p} J_M(U,y).$
\begin{remark}
		We wish to draw the reader’s attention to the fact that with the	inclusion	of		$\tau^U_M$	in	the	definition	of	the	cost	functional,	we	are	able	to	consider	2D	and	3D	domains.	On	the	other	hand,	even		if		$d=2$	and	$\beta\neq	0$,	we	do not 	know	if	$\tau^U_M=	T$.	In	the		case		$\beta=0$	and	$d=2$,	we	get	$\tau^U_M=	T$	and	we	can	recover	the	stochastic	optimal	control	problem	studied	in	\cite{Chem-Cip-2}.	Moreover,	in the two-dimensional deterministic case	\textit{i.e.}		$G\equiv	0$,	one	recovers	the	deterministic	optimal	control	problem	studied	in	\cite{yas-fer}	with	$\tau^U_M=	T.$
\end{remark}
\subsection{Main results}
Here, we present the main results of our article, which show the existence of a solution to the control problem and establish the first order optimality conditions. 
\begin{theorem}\label{main-thm} 
	Assume $\mathcal{H}_1$. Then the control problem \eqref{problem}  admits a unique   optimal solution 
	$(\widetilde{U_M},\tilde{y}) \in \mathcal{U}_{ad}^p \times L^p(\Omega;L^p(0,T;\widetilde{W})),$
	where $\tilde{y}:=y(\widetilde{U_M})$ is the unique solution of \eqref{I} with $U=\widetilde{U_M}$. Moreover, under the assumption $\mathcal{H}_0$: 	
	\begin{itemize} 
		\item  there exists a unique solution   $\tilde{z}$  of the linearized equations  \eqref{Linearized}, in 2D and 3D with $y=\tilde{y}$ and $\psi=\psi-\widetilde{U_M}$;
		\item if $d=2$,  there exists a unique solution  $(\tilde{\mathcalb{p}},\tilde{\mathcalb{q}})$ of the stochastic backward adjoint equation \eqref{adjoint}, with forces $g=\tilde{y}-y_d$;
		\item if $d=3$, there exists,  at least, one  solution   $(\tilde{\mathbf{p}},\tilde{\mathbf{q}})$ of the stochastic backward adjoint equation \eqref{adjoint3D}, with forces $g=\tilde{y}-y_d$.
	\end{itemize} In addition,    the duality property 
\begin{align}\label{duality-thm}
\E\int_0^{T}\mathbb{I}_{[0,\tau_M[}(\psi-\widetilde{U_M},\tilde{\mathcalb{p}})ds=\E
\int_0^T\hspace*{-0.25cm}\mathbb{I}_{[0,\tau_M[}(s)(y(\widetilde{U_M})-y_d,z)ds=	\E \int_0^T\hspace*{-0.25cm}\mathbb{I}_{[0,\tau_M[}(s)(\psi-\widetilde{U_M},\tilde{\mathbf{p}})dt
\end{align}
is valid for any  $\psi \in \mathcal{U}_{ad}^p$,  and the following optimality condition 	holds,  for any   $\psi \in \mathcal{U}_{ad}^p$
\begin{align}\label{oocc}
	&\E\int_0^T 	\big(\lambda\Vert \widetilde{U_M}\Vert_{(H^1(D))^2}^{p-2} (\widetilde{U_M},\psi-\widetilde{U_M})_{(H^1(D))^2}+ \mathbb{I}_{[0,\tau_M[}(\psi-\widetilde{U_M},\tilde{\mathcalb{p}})\big)ds\geq 0;\nonumber\\
	&\E\int_0^T 	\big(\lambda\Vert \widetilde{U_M}\Vert_{(H^1(D))^3}^{p-2} (\widetilde{U_M},\psi-\widetilde{U_M})_{(H^1(D))^3}+ \mathbb{I}_{[0,\tau_M[}(\psi-\widetilde{U_M},\tilde{\mathbf{p}})\big)ds\geq 0.
	\end{align}
\end{theorem}
\begin{proof}
The	proof	of	Theorem		\ref{main-thm}		results	from	the	combination	of	Theorem	\ref{thm-control},	Theorem	\ref{exis-z+estim},	Theorem	\ref{exis-adjoint+estim},	Theorem	\ref{exis-adjoint+estim3D}	and	Section	\ref{Section-Optimality-condition}.
\end{proof}
\begin{remark} In Theorem \ref{main-thm},
we did not distinguish between the notations in 2D and 3D  for  $z$, $\tilde{y}-y_d$, the control $U$ and $\psi$, since it is clear from the context.
\end{remark}

\subsection{Existence and uniqueness of optimal control}
Let us notice that  the velocity field and the  stopping times are non-convex with respect to the control $U$,  then \eqref{problem} is a  non-convex optimization problem. Our strategy 
to show the  existence and uniqueness of the optimal solution
relies on the lower semi-continuous 
of the cost functional and on the application of a result from \cite{Goebel}  (see Theorem \ref{thm-control}). 
First, let us prove the following results, which will play an important  role in the proof of existence of  optimal control associated with  the cost funtional \eqref{cost-exam}.

\begin{lemma}\label{Lemma-equality-stopping}
	For fixed $M \in \mathbb{N}$, let $(y_M(t;U))_{t\in [0,T]}$ be the solution to \eqref{I} in the sense of Definition \ref{Def-strong-sol-main} corresponding to the control $U\in L^p(\Omega_T,(H^1(D))^d)$ and the stopping time $\tau_M^U$. Then we have 
	$$\lim_{U_1 \to U_2} P(\tau_M^{U_1}\neq \tau_M^{U_2})=0,	\text{	where	} U_1\to U_2 \text{	means }\displaystyle\Vert U_1-U_2\Vert_{L^p(\Omega_T,(H^1(D))^d)}\to 0$$
 	\begin{align}\label{stop-time}
\text{	and	}\qquad	\tau_M^{U_i}=\tau_M^i=\inf\{ t\geq 0: \Vert  y_ i(t)\Vert_{W^{2,4}} \geq M\} \wedge T.
	\end{align}
\end{lemma}

\begin{proof}
	Let us consider $\epsilon,\delta,\gamma >0$ and two solutions  
	$(y_M(t,U_i))_{t\in [0,T]}^{i=1,2}$ of equation \eqref{I}, corresponding to $(U_i)^{i=1,2}$,  respectively, with the same initial data $y_0$. 
	Standard computations give
	\begin{align*}
	P\big(\sup_{s\in [0, \tau_M^1\wedge\tau_M^2]} \Vert y_M(s,U_1)-y_M(s,U_2)\Vert_{W^{2,4}} \geq \delta\big) \leq \dfrac{1}{\delta^{4(1+\gamma)}}\E \sup_{s\in [0, \tau_M^1\wedge\tau_M^2]} \Vert y_M(s,U_1)-y_M(s,U_2)\Vert_{W^{2,4}}^{4(1+\gamma)}.
	\end{align*}
	On the one hand,  due to the H\"older inequality there exists $C>0$ such that
	\begin{align*}
	&\sup_{s\in [0, \tau_M^1\wedge\tau_M^2]}	\Vert y_M(s,U_1)-y_M(s,U_2)\Vert_{W^{2,4}}^{4(1+\gamma)} \\&\qquad\leq C \sup_{s\in [0, \tau_M^1\wedge\tau_M^2]}\Vert y_M(s,U_1)-y_M(s,U_2)\Vert_{W}^{1+\gamma}\cdot \sup_{s\in [0, \tau_M^1\wedge\tau_M^2]}\Vert y_M(s,U_1)-y_M(s,U_2)\Vert_{W^{2,6}}^{3(1+\gamma)}.
	\end{align*}
	We recall that $p>2(d+1)$. Therefore the Cauchy Schwarz inequality yields
	\begin{align*}
	&\E\sup_{s\in [0, \tau_M^1\wedge\tau_M^2]}\Vert y_M(s,U_1)-y_M(s,U_2)\Vert_{W^{2,4}}^{4(1+\gamma)} \\&\leq \big(\E\sup_{s\in [0, \tau_M^1\wedge\tau_M^2]}\Vert y_M(s,U_1)-y_M(s,U_2)\Vert_{W}^{2(1+\gamma)}\big)^{1/2}\cdot\big(\E\sup_{s\in [0, \tau_M^1\wedge\tau_M^2]}\Vert y_M(s,U_1)-y_M(s,U_2)\Vert_{W^{2,6}}^{6(1+\gamma)}\big)^{1/2}\\
	&\leq C(M,\epsilon)\sqrt{  \E\int_0^{\tau_M^1\wedge\tau_M^2}\Vert U_1-U_2\Vert_{2}^{2(1+\gamma)} ds+ C\Vert U_1-U_2\Vert_{L^{2(1+\gamma)(d+1+\epsilon)}(\Omega\times(0,\tau_M^1\wedge\tau_M^2);L^2)}^{2(1+\gamma)}}:=A_M(\epsilon,\gamma),
	\end{align*}
	where we used  Proposition \ref{local-stability-thm} and Proposition \ref{Estimate-y-proposition} to deduce the last inequality. Therefore
	\begin{align}\label{contradiction-result}
	P\big(\sup_{s\in [0, \tau_M^1\wedge\tau_M^2]} \Vert y_M(s,U_1)-y_M(s,U_2)\Vert_{W^{2,4}} \geq \delta\big) \leq \dfrac{A_M(\epsilon,\gamma)}{\delta^{4(1+\gamma)}}\overset{U_1\to U_2}{\longrightarrow}0.
	\end{align}
	On the other hand, assume that $\displaystyle\lim_{U_1 \to U_2}P(\tau_M^{U_1} >\tau_M^{U_2}) >0$. From the definition of the stopping times, we get
	$\displaystyle\lim_{U_1 \to U_2}P(\{ \Vert y_M(\tau_M^{U_2},U_2)\Vert_{W^{2,4}} \geq M\}\cap \{\Vert y_M(\tau_M^{U_2},U_1)\Vert_{W^{2,4}} < M\} ) >0.$
	Therefore, there exists $\delta_0>0$ such that
	$\displaystyle\lim_{U_1 \to U_2}P(\Vert y_M(\tau_M^{U_2},U_2)\Vert_{W^{2,4}}-\Vert y_M(\tau_M^{U_2},U_1)\Vert_{W^{2,4}} \geq \delta_0) >0,$
	which gives 
	$\displaystyle\lim_{U_1 \to U_2}P\big(\Vert y_M(\tau_M^{U_2},U_2)-y_M(\tau_M^{U_2},U_1)\Vert_{W^{2,4}} \geq \delta_0\big) >0,$
	which contradicts  \eqref{contradiction-result}. Hence $\displaystyle\lim_{U_1 \to U_2}P(\tau_M^{U_1} >\tau_M^{U_2}) =0$. A similar reasoning yields $\displaystyle\lim_{U_1 \to U_2}P(\tau_M^{U_2} >\tau_M^{U_1}) =0$.
\end{proof}

Next, we establish a result which will be useful in the analysis of the variation of the cost functional \eqref{cost-exam}.
\begin{cor}\label{cor-delta0} 
	Let us consider $U$ and $y_0$ satisfying \eqref{data-assumptions} and $\psi\in L^p((\Omega_T, \mathcal{P}_T);(H^1(D))^d)$. Defining
	$ U_\rho=U+\rho\psi, \; \rho \in (0,1),$
	let $(y,\tau_M)$ and $(y_\rho,\tau_M^\rho)$ be  the solutions of \eqref{I} associated with $(U,y_0)$ and $(U_\rho,y_0)$, respectively. Then we have
	$\lim_{\rho \to 0} \dfrac{P(\tau_M^{U}\neq \tau_M^{U_\rho})}{\rho} =0. $
\end{cor}
\begin{proof}
	Let $\gamma>0$. Using  similar arguments as in  the proof of Lemma \ref{Lemma-equality-stopping}, we can show that $
	P\big(\sup_{s\in [0, \tau_M\wedge\tau_M^\rho]} \Vert y_\rho-y\Vert_{W^{2,4}} \geq \delta\big) \leq \dfrac{C(M,\epsilon)}{\delta^{4(1+\gamma)}} \rho^{1+\gamma},
	$
	which implies the desired result.
\end{proof}

In order to simplify the notation, we set 
$X:=L^p((\Omega_T,\mathcal{P}_T);(H^1(D))^d))$ and denote by $\|\cdot\|_X$ the natural norm on	$X$. The subset  $\mathcal{U}_{ad}^p\subset X$  is endowed with
the induced topology.

\begin{lemma}\label{lemma-continuous}
	For fixed $M \in \mathbb{N}$. Let $y_d\in L^2(0,T;H)$  fixed, and $y(U)=(y_M(t;U), \tau_M^U), \,t\in [0,T],$ be the local pathwise solution of \eqref{I} associated with the control $U\in \mathcal{U}_{ad}^p$ and  stopping times $(\tau_M^U)_{M\in \mathbb{N}}$.
		  Then the mapping $f_M:\mathcal{U}_{ad}^p\to \mathbb{R}$ given by
	 $	\displaystyle f_M(U,y(U))=\E\int_0^{\tau^U_M}\Vert y(U)-y_d\Vert_2^2dt	
	 $
	 is continuous.
\end{lemma}
\begin{proof}
	Let $y_M(t;U_i)),\,t\in [0,T]$,  be the  strong solution  of \eqref{I}, in the sense of Definition \ref{Def-strong-sol-main}, corresponding to the control $U_i\in \mathcal{U}_{ad}^p$,
	$i=1,2$, respectively.
		 We introduce  the stopping times $\underline{\tau_M}=\tau_M^{U_1}\wedge\tau_M^{U_2}$, $ \overline{\tau_M}=\tau_M^{U_1}\vee \tau_M^{U_2}$\footnote{We	recall	that	$a\wedge	b=\min(a,b)$	and	$a\vee	b=\max(a,b)$	for	$a,b\in\mathbb{R}$.} and the control $\bar u \in \mathcal{U}_{ad}^p$ defined by
$\bar u=	\begin{cases} U_1 &\text{ if }  \overline{\tau_M}=\tau_M^{U_1},\\
U_2 &\text{ if } \overline{\tau_M}=\tau_M^{U_2}.
	\end{cases}$

	The  triangular inequality and  Lemma \ref{Lemma-V-stability} yield
	\begin{align*}
		\vert f_M(U_1,y(U_1))&-f_M(U_2,y(U_2))\vert
\leq C\E \int_0^{\underline{\tau_M}}\Vert y(U_1)- y(U_2)\Vert_2^2dt+\E\int_{\underline{\tau_M}}^{\overline{\tau_M}}\Vert y(\bar u)-y_d\Vert_2^2dt\\
		&\leq C(M,L)\E\int_0^{\underline{\tau_M}}\Vert U_1(s)-U_2(s)\Vert_{2}^2ds+\int_0^TP(\underline{\tau_M}< t\leq \overline{\tau_M})(C^2M^2+\Vert y_d\Vert_{2}^2)dt.
		\end{align*}
Taking into account Lemma  \ref{Lemma-equality-stopping}, we deduce
	$ \displaystyle\lim_{\Vert U_1 - U_2\Vert_X\to 0}\vert f_M(U_1,y(U_1))-f_M(U_2,y(U_2))\vert=0,$
	which gives the continuity of $f_M$.
 \end{proof}
Next, we establish a main result on the existence and uniqueness of optimal control for \eqref{cost-exam}.
\begin{theorem}\label{thm-control}
	Let us consider the cost functional $J_M: \mathcal{U}_{ad}^p \to \mathbb{R}$ given by  \eqref{cost-exam} for fixed $M \in \mathbb{N}$. Then there exists a unique optimal control $\widetilde{U_M}  \in \mathcal{U}_{ad}^p.$
\end{theorem}
\begin{proof}
	Notice that $X$ is a uniformly convex Banach space.  The set $\mathcal{U}_{ad}^p$ is bounded and closed convex  subset of $X$. Thanks to Lemma \ref{lemma-continuous}, for	$\lambda>0$	we know that $\dfrac{p}{2\lambda}f_M:\mathcal{U}_{ad}^p\to \mathbb{R}$ is continuous  and  $f_M(U) \geq 0$ for every $U\in \mathcal{U}_{ad}^p$. 
	 By applying  \cite[Thm. 4]{Goebel}, we infer  the existence of a dense subset	 $M_0 \subset \mathcal{U}_{ad}^p$ such that for any $v\in M_0$, the functional $ \dfrac{p}{2\lambda}f_M(U)+\Vert U-v\Vert_X^p$
	attains its minimum over $\mathcal{U}_{ad}^p$. In other words, there exists a unique $U_M(v)\in \mathcal{U}_{ad}^p$ (by taking $\alpha=p$ in   \cite[Thm. 4]{Goebel}) such that
	$(\dfrac{p}{2\lambda}f_M(U_M(v))+\Vert U_M(v)-v\Vert_X^p)=\displaystyle\inf_{U\in \mathcal{U}_{ad}^p}(\dfrac{p}{2\lambda}f_M(U)+\Vert U-v\Vert_X^p). $
	Additionally, $v\mapsto U_M(v)$ is continuous,		see	\cite[Prop.	5.6]{Benner-Trautwein2019}. Since $0\in \mathcal{U}_{ad}^p$, there exists a sequence $(v_k)_{k\in \mathbb{N}} \subset M_0$ such that $\displaystyle\lim_{k\to \infty}\Vert v_k\Vert_{X}=0$. Now, set $\widetilde{U_M}=\displaystyle\lim_{k\to \infty}U_M(v_k)$ and notice that
	\begin{align*}
	\dfrac{p}{\lambda}	J_M(\widetilde{U_M})&=\lim_{k\to \infty}(\dfrac{p}{2\lambda}f_M(U_M(v_k))+\Vert U_M(v_k)-v_k\Vert_X^p)=\lim_{k\to \infty}\inf_{U\in \mathcal{U}_{ad}^p}(\dfrac{p}{2\lambda}f_M(U)+\Vert U-v_k\Vert_X^p)\\
		&=\inf_{U\in \mathcal{U}_{ad}^p}(\dfrac{p}{2\lambda}f_M(U)+\Vert U\Vert_X^p)=\inf_{U\in \mathcal{U}_{ad}^p}\dfrac{p}{\lambda}J_M(U),
	\end{align*}
	where the last inequalities hold thanks to the continuity properties  of  $v\mapsto\displaystyle\inf_{U\in \mathcal{U}_{ad}}\dfrac{p}{2\lambda}f_M(U)+\Vert U-v\Vert_X^p$. The uniqueness of $\widetilde{U_M}$ follows  from the uniqueness of the limit.
\end{proof}
\section{Linearized state equation}\label{Sec5}
This section is devoted to the study of the linearized state equation.
The existence of the solution is based on the  Faedo-Galerkin's approximation, which relies on a special basis, in order to derive the uniform estimates  in $V$ and  $W$.
\\
$\quad$	Let us consider $\psi:D\times ]0,T[\times \Omega \to \mathbb{R}^d, d=2,3$ such that
\begin{align}\label{psi}
\psi \in L^{p}((\Omega_T,\mathcal{P}_T),(H^1(D))^d).
\end{align}
Taking into account the assumptions \eqref{noiseV} and $\mathcal{H}_0$, let us set
\begin{align}
\label{14_1}
\nabla_yG(t,y): (L^2(D))^d &\to L_2(\mathbb{H},(L^2(D))^d)\nonumber\\
	v&\mapsto \nabla_yG(t,y)v \,=\,\{e_\k\to \nabla_yG(t,y)ve_\k:=\nabla_y\sigma_\k(t,y) v\}.
\end{align}
Our goal is to show the existence of the solution $z$ to the following problem:
\begin{align}\label{Linearized}
\begin{cases}
dv(z)+\mathbb{I}_{[0,\tau_M[}\big\{-\nu \Delta z+(y\cdot \nabla )v(z)+(z\cdot \nabla)v(y)+\sum_{j}v(z)^j\nabla y^j+\sum_{j}v(y)^j\nabla z^j&\\[0.1cm]\quad-(\alpha_1+\alpha_2)\text{div}[A(y)A(z)+A(z)A(y)]-\beta\text{ div} \big[ |A(y)|^2A(z)\big]
-2\beta\text{ div}\left[ \left( A(z):A(y)\right)A(y)\right]\big\} dt&\\\quad=\{\psi-\nabla \pi\}\mathbb{I}_{[0,\tau_M[}dt+\mathbb{I}_{[0,\tau_M[}\nabla_yG(\cdot,y)zd\mathcal{W} \quad &\hspace*{-7.4cm}\text{in } D\times\Omega\times (0,T), \\
\text{div}(z)=0 \quad &\hspace*{-7.4cm}\text{in } D\times\Omega\times (0,T),\\
z\cdot \eta=0 \quad [\eta \cdot \mathbb{D}(z)]\cdot \tau=0  \quad &\hspace*{-7.4cm}\text{on } \partial D\times\Omega\times (0,T),\\
z(0)=0 \quad &\hspace*{-7.4cm} \text{in } D\times\Omega,
\end{cases}
\end{align}
where $(\tau_M)_M$ is the sequence of stopping times introduced in \eqref{stopping-time}.
Since the solution of \eqref{I} is defined up to the stopping times $(\tau_M)_{M\in \mathbb{N}}$.Then we define the solution of \eqref{Linearized} accordingly.
\begin{definition}\label{Def-lin}
	A stochastic process $z$ is a solution of \eqref{Linearized} if and only if:	
	\begin{itemize}
		\item $z$ is predictable  with values in $W$ and $v(z) \in \mathcal{C}_w([0,T];(L^2(D))^d)$\footnote{For a Banach space $X$, $\mathcal{C}_w([0,T];X)$ denotes the space of weakly continuous functions  on $[0,T]$ with values in $X$.} P-a.s.
		\item $z\in L^{p}((\Omega_T,\mathcal{P}_T),V)\cap L^2_w(\Omega;L^\infty(0,T;W)).$  
	\item For any $t\in [0,T]$ and P-a.s. $\omega \in \Omega$, the following equality holds
	\begin{align}
	\label{15_1}
	&(v(z(t)),\phi)+\int_0^{t}\mathbb{I}_{[0,\tau_M[}\{2\nu(\mathbb{D} z,\mathbb{D}\phi)+b(y,v(z),\phi)+b(z,v(y),\phi)+b(\phi,y,v(z))+b(\phi,z,v(y))\}ds\notag\\
	&\quad+\int_0^{t}\mathbb{I}_{[0,\tau_M[}\{(\alpha_1+\alpha_2)\big(A(y)A(z)+A(z)A(y),\nabla \phi\big) 
	+\beta\big( |A(y)|^2A(z), \nabla\phi\big)\}ds\notag\\&\qquad+\int_0^{t}2\beta \mathbb{I}_{[0,\tau_M[}\big((A(z):
	A(y))A(y),\nabla\phi\big)ds\notag\\
	&\quad\qquad= \int_0^{t}\mathbb{I}_{[0,\tau_M[}(\psi,\phi)ds+\int_0^{t}\mathbb{I}_{[0,\tau_M[}(\nabla_yG(\cdot,y)z,\phi)d\mathcal{W}(s),\quad
	 \forall \phi \in V.
	\end{align}
\end{itemize} 
\end{definition}
\begin{remark} Due to the  presence of  $\tau_M$ in the system \eqref{Linearized},  the solution  $z$ 
 depends on $M$. 
\end{remark}
\begin{theorem}\label{exis-z+estim}
	Assume that $\psi$ satisfy \eqref{psi}. Then there exists a unique solution $z$ to \eqref{Linearized}, in the sense of Definition \ref{Def-lin},
	satisfying the following estimates
	\begin{align}\label{estimate-z}
	\E\sup_{s\in [0,T]}\Vert z(s)\Vert_{V}^{p} \leq C(M)\int_0^{\tau_M}\Vert \psi\Vert_{2}^{p}ds\text{	and	}
	\E\sup_{s\in [0, T]}\Vert z(s)\Vert_{W}^{2} \leq 
	C(\alpha_1,\alpha_2,\beta,\nu,T,M),
	\end{align}
	where  $ C(M)>0$ and $ C(\alpha_1,\alpha_2,\beta,\nu,T,M)>0$.
\end{theorem}
\subsection{Approximation}
\label{S1_zn}
Following the same strategy as in \cite{Chem-Cip-2},  let us consider 
an orthonormal basis $\{h_i\}_{i\in \mathbb{N}} \subset (H^4(D)^d) \cap W$  in $V$, which  satisfies
\begin{align}\label{basis1}
(v,h_i)_W=\mu_i(v,h_i)_V, \;\forall v \in W, \;  \mu_i >0, \forall i\in \mathbb{N},\text{	and	}\;\mu_i \to \infty,\; \text{as}\;i \to \infty.
\end{align}
 As a consequence of \eqref{basis1}, the sequence $\{\tilde{h}_i=\frac{1}{\sqrt{\mu_i}}h_i\}$ is an orthonormal basis in $W$. Let us introduce the Galerkin approximations of \eqref{Linearized}. Consider $W_n=\text{span}\{h_1,\cdots,h_n\}$ and define
$ z_n(t)=\sum_{i=1}^n c_i(t)h_i$   for each  $t\in [0,T].$
The approximated problem for \eqref{Linearized} reads,	 for	$ \phi \in W_n,$
	\begin{align}\label{approximation}
&(v(z_n(t)),\phi)+\int_0^{t}\hspace*{-0.15cm}\mathbb{I}_{[0,\tau_M[}\{2\nu(\mathbb{D} z_n,\mathbb{D}\phi)+b(y,v(z_n),\phi)+b(z_n,v(y),\phi)+b(\phi,y,v(z_n))+b(\phi,z_n,v(y))\}ds\nonumber\\
&\quad+\int_0^{t}\mathbb{I}_{[0,\tau_M[}\{(\alpha_1+\alpha_2)\big(A(y)A(z_n)+A(z_n)A(y),\nabla \phi\big) 
+\beta\big( |A(y)|^2A(z_n), \nabla\phi\big)\}ds\\&\quad+\int_0^{t}2\beta \mathbb{I}_{[0,\tau_M[}\big((A(z_n):
A(y))A(y),\nabla\phi\big)ds= \int_0^{t}\mathbb{I}_{[0,\tau_M[}(\psi,\phi)ds+\int_0^{t}\mathbb{I}_{[0,\tau_M[}(\nabla_yG(\cdot,y)z_n,\phi)d\mathcal{W}. \nonumber
\end{align}
 \eqref{approximation} defines  a system of $n$-stochastic  linear ODE, 
 by using \textit{e.g.}	a  \textit{"Banach fixed point theorem"}, it follows that  \eqref{approximation}  has a unique  solution $z_n$ such that $z_n$ is a predictable process and satisfies
\begin{align}\label{local-time-lin}
z_n\in L^2(\Omega;\mathcal{C}([0,T],W_n)).
\end{align}
\subsection{Uniform estimates}
\subsubsection{Estimate in the space $V$ for $z_n$}\label{subsectzn-V}
For any $t\in [0,T]$, setting $\phi=h_i$ in \eqref{approximation}, and applying It\^o's formula for the function $x\mapsto x^2$, we infer
\begin{align*}
&(z_n(t),h_i)_V^2+\int_0^{t}2\mathbb{I}_{[0,\tau_M[}(z_n,h_i)_V\{(\alpha_1+\alpha_2)\big(A(y)A(z_n)+A(z_n)A(y),\nabla h_i\big) 
+\beta\big( |A(y)|^2A(z_n), \nabla h_i\big)\}ds\\
&+\int_0^{t}2\mathbb{I}_{[0,\tau_M[}(z_n,h_i)_V\{2\nu(\mathbb{D} z_n,\mathbb{D}h_i)+b(y,v(z_n),h_i)+b(z_n,v(y),h_i)+b(h_i,y,v(z_n))+b(h_i,z_n,v(y))\}ds\nonumber\\
&\qquad+\int_0^{t}4\beta\mathbb{I}_{[0,\tau_M[} (z_n,h_i)_V\big((A(z_n):
A(y))A(y),\nabla h_i\big)ds= \int_0^{t}2\mathbb{I}_{[0,\tau_M[}(z_n,h_i)_V(\psi,h_i)ds\\
&\quad\qquad+\int_0^{t}2\mathbb{I}_{[0,\tau_M[}(z_n,h_i)_V(\nabla_yG(\cdot,y)z_n,h_i)d\mathcal{W}+\sum_{ \k\geq 1}\int_0^{t}\mathbb{I}_{[0,\tau_M[}(\nabla_y\sigma_\k(\cdot,y)z_n,h_i)^2 ds.
\end{align*}
By summing the last equalities from $i=1$ to $n$, we obtain
\begin{align*}
&\Vert z_n(t)\Vert_V^2+\int_0^{t}2\mathbb{I}_{[0,\tau_M[}\{(\alpha_1+\alpha_2)\big(A(y)A(z_n)+A(z_n)A(y),\nabla z_n\big) 
+\beta\big( |A(y)|^2A(z_n), \nabla z_n\big)\}ds\\&+\int_0^{t}2\mathbb{I}_{[0,\tau_M[}\{2\nu\|\mathbb{D}z_n\|^2_2
+b(y,v(z_n),z_n)+b(z_n,v(y),z_n)+b(z_n,y,v(z_n))+b(z_n,z_n,v(y))\}ds\nonumber\\[-0.5cm]
\\&\qquad+\int_0^{t}4\beta \mathbb{I}_{[0,\tau_M[}\big((A(z_n):
A(y))A(y),\nabla z_n\big)ds= \int_0^{t}2\mathbb{I}_{[0,\tau_M[}(\psi,z_n)ds\\
&\quad\qquad+2\int_0^{t}\mathbb{I}_{[0,\tau_M[}(\nabla_yG(\cdot,y)z_n,z_n)d\mathcal{W}+\sum_{i=1}^n\sum_{ \k\geq 1}\int_0^{t}\mathbb{I}_{[0,\tau_M[}(\nabla_y\sigma_\k(\cdot,y)z_n,h_i)^2 ds.
\end{align*}
Thanks to Lemma \ref{tecnical-lemma} (see Section \ref{Technical-results}), we have 
\begin{align}
\label{12_1}
&\int_0^{t}\mathbb{I}_{[0,\tau_M[}\{b(y,v(z_n),z_n)+b(z_n,v(y),z_n)+b(z_n,y,v(z_n))+b(z_n,z_n,v(y))\}ds\notag
\\&=\int_0^{t}\mathbb{I}_{[0,\tau_M[}(-b(y,z_n,v(z_n))+b(z_n,y,v(z_n)))ds\leq  C\int_0^{t}\mathbb{I}_{[0,\tau_M[}\Vert y\Vert_{W^{2,4}}\Vert z_n\Vert_{V}^2ds\notag\\
&\leq  CM\int_0^{t}\mathbb{I}_{[0,\tau_M[}\Vert z_n\Vert_{V}^2ds.
\end{align}
Now, similarly to \cite[Section 4.2]{yas-fer}, we derive the following estimates:
\begin{align}
\label{12_2}
&\left\vert \int_0^{t}2(\alpha_1+\alpha_2)\mathbb{I}_{[0,\tau_M[}\int_D [A(y)A(z_n)+A(z_n)A(y)]: \nabla z_n dx ds\right\vert \leq  C_2  \int_0^{t}\mathbb{I}_{[0,\tau_M[}\Vert z_n\Vert_{H^1}^2\Vert  y\Vert_{W^{1,\infty}}ds\notag\\
&\quad\leq C_2\int_0^{t}\mathbb{I}_{[0,\tau_M[}  \Vert z_n\Vert_{V}^2\Vert  y\Vert_{W^{2,4}}ds\leq C_2M\int_0^{t} \mathbb{I}_{[0,\tau_M[} \Vert z_n\Vert_{V}^2ds;\notag\\
&\left\vert \int_0^{t}4\beta\mathbb{I}_{[0,\tau_M[}\int_D[\left(A(z_n)\cdot A(y)\right)A(y)\big]:\nabla z_ndxds\right\vert\leq C_3  \int_0^{t}\mathbb{I}_{[0,\tau_M[}\Vert z_n\Vert_{H^1}^2\Vert  y\Vert_{W^{1,\infty}}^2ds\notag\\&\quad \leq C_3  \int_0^{t}\mathbb{I}_{[0,\tau_M[}\Vert z_n\Vert_{V}^2\Vert  y\Vert_{W^{2,4}}^2 ds
\leq C_3 M^2 \int_0^{t}\mathbb{I}_{[0,\tau_M[}\Vert z_n\Vert_{V}^2 ds.
\end{align}
Since  $A(z_n)$ is symmetric, we have the relation
\begin{align}
\label{12_3}
&\int_0^{t}2\beta\mathbb{I}_{[0,\tau_M[}\big( |A(y)|^2A(z_n), \nabla z_n\big)ds=\beta\int_0^{t}\mathbb{I}_{[0,\tau_M[}\int_D|A(y)|^2|A(z_n)|^2dxds \geq 0.
\end{align}
Let $\k \in \mathbb{N}^*$, denote by $\bar{\sigma_\k}$ the solution of \eqref{Stokes} with $f=\nabla_y\sigma_\k(\cdot,y)z_n \in (L^2(D))^d$. Therefore
\begin{align*}
&\sum_{i=1}^n\sum_{ \k\geq 1}\int_0^{t}\mathbb{I}_{[0,\tau_M[}(\nabla_y\sigma_\k(\cdot,y)z_n,h_i)^2 ds=
\sum_{i=1}^n\sum_{ \k\geq 1}\int_0^{t}\mathbb{I}_{[0,\tau_M[}(\bar{\sigma_\k},h_i)_V^2 ds =
\sum_{ \k\geq 1}\int_0^{t}\mathbb{I}_{[0,\tau_M[}\Vert\bar{\sigma_\k}\Vert_V^2 ds\\&\leq  	C\sum_{ \k\geq 1}\int_0^{t}\mathbb{I}_{[0,\tau_M[}\Vert\nabla_y\sigma_\k(\cdot,y)z_n\Vert_2^2 ds \leq  	C\sum_{ \k\geq 1}\int_0^{t}\mathbb{I}_{[0,\tau_M[}a_k^2\Vert z_n\Vert_2^2 ds\leq C_{*}\int_0^{t}\mathbb{I}_{[0,\tau_M[}\Vert z_n\Vert_2^2 ds.
\end{align*}
Gathering the previous estimates, there exists $C>0$ independent of $n$ such that the following inequality holds
\begin{align*}
\Vert z_n(t)\Vert_V^2+4\nu\int_0^{t} \mathbb{I}_{[0,\tau_M[}\Vert  \mathbb{D} z_n\Vert_{2}^2ds &\leq C(1+M^2)\int_0^{t}\mathbb{I}_{[0,\tau_M[}\Vert z_n\Vert_{V}^2ds+\int_0^{\tau_M}\Vert \psi\Vert_2^2ds\\
&\qquad+2\int_0^{t}\mathbb{I}_{[0,\tau_M[}(\nabla_yG(\cdot,y)z_n,z_n)d\mathcal{W}.
\end{align*}
Letting $p\geq 1$ and  taking the $p^{th}$ power of each term of this inequality, we obtain
\begin{align}
\label{09:01:23_1}
\Vert z_n(t)\Vert_V^{2p}&\leq C(p)T^{p-1}[(1+M^2)^p\int_0^{t}\mathbb{I}_{[0,\tau_M[}\Vert z_n\Vert_{V}^{2p}ds+\int_0^{t}\mathbb{I}_{[0,\tau_M[}\Vert \psi\Vert_2^{2p}ds]\notag\\&\qquad+C(p)\vert \int_0^{t}\mathbb{I}_{[0,\tau_M[}(\nabla_yG(\cdot,y)z_n,z_n)d\mathcal{W}\vert^p.
\end{align}
For fixed $N\in \mathbb{N}$, let us introduce  the  stopping time 
$\mathbf{t}_N^n=\inf\{ t\in [0,T]: \Vert z_n(t)\Vert_V \geq N\}$. 
For $t\in[0,T]$, and any $\delta >0$, the Burkholder-Davis-Gundy inquality gives 
\begin{align*}
&\E\sup_{s\in [0, t\wedge\mathbf{t}_N^n]}\vert \int_0^{s}\mathbb{I}_{[0,\tau_M[}(\nabla_yG(\cdot,y)z_n,z_n)d\mathcal{W}\vert^p
\leq \delta\E\sup_{s\in [0, t\wedge\mathbf{t}_N^n]}\Vert z_n(s)\Vert_{V}^{2p}+C_\delta \E \int_0^{t\wedge\mathbf{t}_N^n}
\mathbb{I}_{[0,\tau_M[}
\Vert z_n\Vert_2^{2p}ds.
\end{align*}
Considering  the supremum in the inequality \eqref{09:01:23_1},
for $s\in [0, t\wedge\mathbf{t}_N^n]$, and inserting the estimate of the stochastic term, after an appropriate choice of $\delta$, we deduce
\begin{align*}
&\E\sup_{s\in [0, t\wedge\mathbf{t}_N^n]}\Vert z_n(s)\Vert_{V}^{2p}\leq C(p)T^{p-1}[(1+M^2)^p\E
\int_0^{t}\mathbb{I}_{[0,\tau_M[}
\sup_{r\in [0,s\wedge\mathbf{t}_N^n]}
\Vert z_n(r)\Vert_{V}^{2p}ds+\E\int_0^{\tau_M}\Vert \psi\Vert_2^{2p}ds].
\end{align*}
Using the Gronwall's inequality, we obtain
$\E\sup_{s\in [0, \mathbf{t}_N^n]}\Vert z_n(s)\Vert_{V}^{2p}\leq  C(M)\E\int_0^{\tau_M}\Vert \psi\Vert_{2}^{2p}ds,$
where $ C(M)=C(p)T^{p-1}[(1+M^2)^pe^{C(p)T^{p}[(1+M^2)^p}.$ Note	that  $(\mathbf{t}_N^n)_N$ is an increasing positive sequence of stopping times and $\mathbf{t}_N^n \to T$ in probability.	Thus, the monotone convergence theorem  ensures
\begin{align*}
&\E\sup_{s\in [0, T]}\Vert z_n(s)\Vert_{V}^{2p} \leq C(M)\E\int_0^{\tau_M}\Vert \psi\Vert_{2}^{2p}ds, \quad \forall t\in [0,T], \quad \forall p \geq 1.
\end{align*}

\smallskip
Next, we will establish a $W$-uniform estimate for 
$z_n$.  Before entering in details, we recall a useful  equality for  divergence free 
vector fields  tangent to the boundary. 
\begin{lemma}\label{Lemma 3.1} We have:
	$ (curl v(y)\times u,\phi)=b(\phi,u,v(y))-b(u,\phi,v(y)),$
	$\forall	u,y \in \widetilde{W}$ and $\phi \in V${\color{blue}.}
\end{lemma}

\subsubsection{Estimate in the space $W$ for $z_n$}

As a consequence of  Lemma \ref{Lemma 3.1}, we	get
\begin{align*}
&((y\cdot \nabla )v(z)+(z\cdot \nabla)v(y)+\sum_{j}v(z)^j\nabla y^j+\sum_{j}v(y)^j\nabla z^j,\phi)
=(curl v(y)\times z,\phi)+(curl v(z)\times y,\phi),
\end{align*}
for any $y\in \widetilde{W}$ and  $z,\phi \in W$.
Setting $\phi=h_i$ in \eqref{approximation} we write 
	\begin{align}\label{approx-Ito}
&(z_n(t),h_i)_V+\int_0^{t}\mathbb{I}_{[0,\tau_M[}(f_n,h_i)ds=\sum_{ \k\geq 1}\int_0^{t}\mathbb{I}_{[0,\tau_M[}(\nabla_y\sigma_\k(\cdot,y)z_n,h_i)d\beta_\k, \quad \forall t\in[0,T],\\
&\text{		where 	}	f_n= -\nu \Delta z_n+ curl v(y)\times z_n+curl v(z_n)\times y-(\alpha_1+\alpha_2)\text{div}[A(y)A(z_n)+A(z_n)A(y)]\notag\\
		&\qquad\qquad\qquad-\beta\text{ div} \big[ |A(y)|^2A(z_n)\big]
		-2\beta\text{ div}\left[ \left( A(z_n):A(y)\right)A(y)\right]-\psi.\notag
	\end{align}
Let $\widetilde{f}_n$ and $\tilde{\sigma}_\k^n$ be the solutions of \eqref{Stokes} for $f=f_n$ and $f=\nabla_y\sigma_\k(\cdot,y)z_n, \forall \k\in \mathbb{N}$.  We have
\begin{equation}\label{V-W-relation}
(f_n,h_i)=(\widetilde{f}_n,h_i)_V, \; (\nabla_y\sigma_\k(\cdot,y)z_n,h_i)=(\tilde{\sigma}_k^n,h_i)_V 
\end{equation}
By multiplying \eqref{approx-Ito} by   $\mu_i$ and using \eqref{basis1}, we derive 
$$ (z_n(t),h_i)_W+\int_0^{t}\mathbb{I}_{[0,\tau_M[}(\widetilde{f}_n,h_i)_Wds=\sum_{ \k\geq 1}\int_0^{t}\mathbb{I}_{[0,\tau_M[}(\tilde{\sigma}_k^n,h_i)_Wd\beta_\k.$$
Using Ito's formula, next multiplying  by $\frac{1}{\mu_i}$ and summing over $i=1,\cdots,n$  we are able to infer
$$\Vert z_n(t)\Vert_{W}^2+2\int_0^{t}\mathbb{I}_{[0,\tau_M[}(\widetilde{f}_n,z_n)_Wds=2\sum_{ \k\geq 1}\int_0^{t}\mathbb{I}_{[0,\tau_M[}(\tilde{\sigma}_k^n,z_n)_Wd\beta_\k+\sum_{\k\geq 1}\int_0^{t}\mathbb{I}_{[0,\tau_M[} \Vert\tilde{\sigma}_k^n\Vert^2_Wds.$$
By using the definition of  inner product in the space $W$ and using the properties of $\mathbb{P}$, we deduce
\begin{align}\label{eqn-key-zn}
&\Vert z_n(t)\Vert_{W}^2+2\int_0^{t}\mathbb{I}_{[0,\tau_M[}(f_n,z_n)ds+ 2\int_0^{t}\mathbb{I}_{[0,\tau_M[}(f_n,\mathbb{P}v(z_n))ds-\sum_{\k\geq 1} \int_0^{t}\mathbb{I}_{[0,\tau_M[}\Vert\tilde{\sigma}_k^n\Vert^2_Wds\notag \\
&=2\sum_{ \k\geq 1}\int_0^{t}\mathbb{I}_{[0,\tau_M[}(\nabla_y\sigma_\k(\cdot,y)z_n,z_n)d\beta_\k+2\sum_{ \k\geq 1}\int_0^{t}\mathbb{I}_{[0,\tau_M[}(\nabla_y\sigma_\k(\cdot,y)z_n,\mathbb{P}v(z_n))d\beta_\k.
\end{align}


Arguments already detailed in \eqref{12_1}-\eqref{12_3}
 yield
 \begin{align*}
&2\bigl|\int_0^{t}\mathbb{I}_{[0,\tau_M[}(f_n,z_n)ds\bigr|
\leq 4\nu\int_0^{t}\mathbb{I}_{[0,\tau_M[} \Vert  \mathbb{D} z_n\Vert_{2}^2ds+C(1+M^2)\int_0^{t}\mathbb{I}_{[0,\tau_M[}\Vert z_n\Vert_{V}^2ds+\int_0^{\tau_M}\Vert \psi\Vert_{2}^2ds.
\end{align*}

Concerning the  third term of left hand side of \eqref{eqn-key-zn}, we have
\begin{align*}
&2\int_0^{t}\mathbb{I}_{[0,\tau_M[}(f_n,\mathbb{P}v(z_n)) ds= 2\nu\int_0^{t}\mathbb{I}_{[0,\tau_M[} (-\Delta z_n ,\mathbb{P}v(z_n))ds\\
&+ 2\int_0^{t}\mathbb{I}_{[0,\tau_M[}((y\cdot \nabla )v(z_n)+(z_n\cdot \nabla)v(y)+\sum_{j}v(z_n)^j\nabla y^j+\sum_{j}v(y)^j\nabla z_n^j,\mathbb{P}v(z_n))ds\\&-2\int_0^{t}\mathbb{I}_{[0,\tau_M[}(\psi,\mathbb{P}v(z_n))ds-4\beta\int_0^{t}\mathbb{I}_{[0,\tau_M[}\text{ div}\left[ \left( A(z_n):A(y)\right)A(y)\right],\mathbb{P}v(z_n))ds
\\&+2\int_0^{t}\mathbb{I}_{[0,\tau_M[}(-(\alpha_1+\alpha_2)\text{div}[A(y)A(z_n)+A(z_n)A(y)]-\beta\text{ div} \big[ |A(y)|^2A(z_n)\big],\mathbb{P}v(z_n))ds
.\end{align*}
Using the properties of the trilinear form $b$,
\cite[Lemma 5]{Bus-Ift-2} and $W^{2,4}(D)\hookrightarrow W^{1,\infty}(D),$ we deduce 
\begin{align*}
	& 2((y\cdot \nabla )v(z_n)+(z_n\cdot \nabla)v(y)+\sum_{j}v(z_n)^j\nabla y^j+\sum_{j}v(y)^j\nabla z_n^j,\mathbb{P}v(z_n))ds\\
	&\leq 2\big(\Vert y\Vert_{\infty}\Vert v(z_n)-\mathbb{P}v(z_n)\Vert_{H^1}\Vert z_n\Vert_W+\Vert z_n\Vert_{\infty}\Vert y\Vert_{H^3}\Vert z_n\Vert_W+\Vert y\Vert_{W^{1,\infty}}\Vert z_n\Vert_{W}^2+\Vert y\Vert_{W^{2,4}}\Vert z_n\Vert_{W^{1,4}}\Vert z_n\Vert_{W}\big)\\
	&\leq C \Vert y\Vert_{W^{2,4}}\Vert z_n\Vert_{W}^2+2\Vert z_n\Vert_{\infty}\Vert y\Vert_{H^3}\Vert z_n\Vert_W.
\end{align*}
 In addition, we notice that
$	2\nu(-\Delta z_n ,\mathbb{P}v(z_n))-2(\psi,\mathbb{P}v(z_n))\leq (2\nu+1)\Vert z_n\Vert_{W}^2+\Vert \psi\Vert_{2}^2.$
On the other hand, standard computations ensure
\begin{align*}
	&2(-(\alpha_1+\alpha_2)\text{div}[A(y)A(z_n)+A(z_n)A(y)]-\beta\text{ div} \big[ |A(y)|^2A(z_n)\big]
	-2\beta\text{ div}\left[ \left( A(z_n):A(y)\right)A(y)\right],\mathbb{P}v(z_n))\\
	&\leq C \Vert y\Vert_{W^{2,4}}\Vert z_n\Vert_{W}^2+ C \Vert y\Vert_{W^{2,4}}^2\Vert z_n\Vert_{W}^2\leq  C(\alpha_1,\alpha_2,\beta) (\Vert y\Vert_{W^{2,4}}^2+1)\Vert z_n\Vert_{W}^2.
\end{align*}
Therefore, there exists $\mathbf{C} >0$ such that
\begin{align}\label{est-fn-2}
\vert 2(f_n,\mathbb{P}v(z_n))\vert &\leq C(\alpha_1,\alpha_2,\beta,\nu) (\Vert y\Vert_{W^{2,4}}^2+1)\Vert z_n\Vert_{W}^2+\Vert \psi\Vert_{2}^2+2\Vert z_n\Vert_{\infty}\Vert y\Vert_{H^3}\Vert z_n\Vert_W.
\end{align}
Taking into account the properties of the solution of \eqref{Stokes}, we	infer	that
\begin{align*}
\sum_{\k\geq 1} \int_0^{t}\mathbb{I}_{[0,\tau_M[}\Vert\tilde{\sigma}_k^n\Vert^2_Wds \leq C\sum_{\k\geq 1} \int_0^{t}\mathbb{I}_{[0,\tau_M[}\Vert \nabla_y\sigma_\k(\cdot,y)z_n\Vert^2_2ds\leq  C_M\int_0^{t}\mathbb{I}_{[0,\tau_M[}\Vert z_n\Vert^2_2ds.
\end{align*}
\smallskip
Concerning the stochastic integrals, let  us consider $q\geq 1$
and define the sequence of  stopping times 
 $\mathbf{t}_N^n=\inf\{ t\in [0,T]: \Vert z_n(t)\Vert_W \geq N\}$, $N\in \mathbb{N}$. 
 For any $t\in[0,  T]$ and $\delta >0$, the Burkholder-Davis-Gundy inquality gives 
\begin{align*}
&\E\sup_{s\in [0, t\wedge\mathbf{t}_N^n]}\vert \int_0^{s}\mathbb{I}_{[0,\tau_M[}(\nabla_yG(\cdot,y)z_n,z_n)d\mathcal{W}\vert^q+\E\sup_{s\in [0, t\wedge\mathbf{t}_N^n]}\big\vert \int_0^{s}\mathbb{I}_{[0,\tau_M[}\sum_{ \k\geq 1}(\nabla_y\sigma_\k(\cdot,y)z_n,\mathbb{P}v(z_n))d\beta_\k\big\vert^q
\\
&\quad\leq	\delta\E\sup_{s\in [0, t\wedge\mathbf{t}_N^n]}\Vert \mathbb{P}v(z_n)(s)\Vert_{2}^{2q}+
 \delta\E\sup_{s\in [0, t\wedge\mathbf{t}_N^n]}\Vert z_n(s)\Vert_{V}^{2q}+C_\delta \E \int_0^{t\wedge\mathbf{t}_N^n}\mathbb{I}_{[0,\tau_M[}\Vert z_n\Vert_2^{2q}ds.
\end{align*}
Now, taking the $q^{th}$power of relation \eqref{eqn-key-zn}, computing the supremum over the interval
 $[0,  t\wedge\mathbf{t}_N^n]$, taking	the	expectation	 and gathering the previous estimates, we infer that
\begin{align}\label{zn-W-estimate-p}
(1-2\delta)\E\sup_{s\in [0, t\wedge\mathbf{t}_N^n]}\Vert z_n(s)\Vert_{W}^{2q} &\leq 
C(\alpha_1,\alpha_2,\beta,\nu,T,M,q,\delta)
\int_0^{t} \E\sup_{r\in [0, s\wedge\mathbf{t}_N^n]}
\Vert z_n(r)\Vert_{W}^{2q}ds\\
&+C(T)\E\int_0^{\tau_M}\Vert \psi\Vert_{2}^{2q}ds+C(q,T)\E\int_0^{t\wedge\mathbf{t}_N^n}(\Vert z_n\Vert_{\infty}\Vert y\Vert_{H^3}\Vert  z_n\Vert_W)^q ds.\nonumber
\end{align} 
Thanks to Lemma \ref{interpolation-estimate-lem} ( see Section \ref{Technical-results}),
for any $q\geq 1$ and $\epsilon\in ]0,1]$ we have
\begin{align*}
	\E\int_0^{t\wedge\mathbf{t}_N^n}&(\Vert z_n\Vert_{\infty}\Vert y\Vert_{H^3}\Vert  z_n\Vert_W)^q ds \\
	&\leq C_\epsilon(D) [\int_0^{t}
	\E\sup_{r\in [0, s\wedge\mathbf{t}_N^n]}
	\Vert z_n(r)\Vert_{W}^{2q}ds
	+\Vert y\Vert_{L^{2q(d+1+\epsilon)}(\Omega\times(0,\tau_M);H^3)}^{2q(d+\epsilon)}\Vert z_n\Vert_{L^{2q(d+1+\epsilon)}(\Omega\times(0,T);V)}^{2q}],
\end{align*}
 Using this inequality 
and choosing   $\delta$ small enough in \eqref{zn-W-estimate-p}, we are able  to apply Gr\"onwall's inequality to show the existence of a constant $C:=C(\alpha_1,\alpha_2,\beta,\nu,T,M,q,\delta,\epsilon)>0$ such that 
\begin{align*}
\E\sup_{s\in [0, t\wedge\mathbf{t}_N^n]}\Vert z_n(s)\Vert_{W}^{2q} &\leq 
C\E\int_0^{\tau_M}\Vert \psi\Vert_{2}^{2q}ds+ C_\epsilon(D)\Vert y\Vert_{L^{2q(d+1+\epsilon)}(\Omega\times(0,\tau_M);H^3)}^{2q(d+\epsilon)}\Vert z_n\Vert_{L^{2q(d+1+\epsilon)}(\Omega\times(0,T);V)}^{2q}.
\end{align*} 
In particular, for $q=1$ and $t=T$ we obtain 
\begin{align*}
\E\sup_{s\in [0, \mathbf{t}_N^n]}\Vert z_n(s)\Vert_{W}^{2} &\leq 
C\E\int_0^{\tau_M}\Vert \psi\Vert_{2}^{2}ds+ C_\epsilon(D)\Vert y\Vert_{L^{2(d+1+\epsilon)}(\Omega\times(0,\tau_M);H^3)}^{2(d+\epsilon)}\Vert z_n\Vert_{L^{2(d+1+\epsilon)}(\Omega\times(0,T);V)}^{2}.
\end{align*} 
 From  Subsection \ref{subsectzn-V}, we know that
 $
 \displaystyle\E\sup_{s\in [0,T]}\Vert z_n(s)\Vert_{V}^{p} \leq C(M)\E\int_0^{\tau_M}\Vert \psi\Vert_{2}^{p}ds, \quad \forall p \geq 2.$
 
Since	$p>2(d+1)$, for $\epsilon=\dfrac{p-2(d+1)}{2}$, there exists $C(M)>0$ such that
$$C_\epsilon(D)\Vert y\Vert_{L^{2(d+1+\epsilon)}(\Omega\times(0,\tau_M);H^3)}^{2(d+\epsilon)}\Vert z_n\Vert_{L^{2(d+1+\epsilon)}(\Omega\times(0,T);V)}^{2} \leq C(M).$$
 Since  $(\mathbf{t}_N^n)_N$ is  non-decreasing positive sequence of stopping time and $\mathbf{t}_N^n \to T$ in probability, the monotone convergence theorem  ensures 
 $\displaystyle\E\sup_{s\in [0, T]}\Vert z_n(s)\Vert_{W}^{2} \leq 
 C(\alpha_1,\alpha_2,\beta,\nu,T,M).$	Thus

\begin{prop}\label{prop-estimate-zn} Let $z_n$ be the solution to  \eqref{approximation}. Then, $z_n\in \mathcal{C}([0,T],W_n)$ and there exist $ C(M)>0$ and $ C(\alpha_1,\alpha_2,\beta,\nu,T,M)>0$, independents of $n$ such that
	\begin{align*}
	\E\sup_{s\in [0,T]}\Vert z_n(s)\Vert_{V}^{p} \leq C(M)\E\int_0^{\tau_M}\Vert \psi\Vert_{2}^{p}ds	\text{	and	}
	\E\sup_{s\in [0, T]}\Vert z_n(s)\Vert_{W}^{2} \leq 
	C(\alpha_1,\alpha_2,\beta,\nu,T,M).
	\end{align*}
\end{prop}
\subsection{Proof of Theorem \ref{exis-z+estim}}
\subsubsection{$1^{st} step$}
Using  Proposition \ref{prop-estimate-zn}, 
	we deduce the existence of a subsequence of  $z_n$, still denoted 
	by  $z_n$,  and a predictable stochastic process $z \in  L^2(\Omega, L^q(0,T;W))$ such that 
	\begin{align}\label{cv-zn-weak-1}
	z_n &\rightharpoonup z \text{  weakly  in } L^2(\Omega, L^q(0,T;W)),\quad \text{  for any }2\leq q<\infty.
	\end{align} 
	On the other hand, the sequence  $(z_n)$ is bounded in $L^2(\Omega, L^\infty(0,T;W))$, thus in $ L_w^2(\Omega, L^\infty(0,T;W))\simeq( L^2(\Omega, L^1(0,T;W^\prime)))^\prime$, where $w$ stands for the weak-* measurability (see  e.g. \cite[Thm. 8.20.3]{Edwards}, \cite[Rmq. 2.1]{Vallet-Zimmermann}). Hence,  Banach–Alaoglu theorem's ensures  $z\in  L_w^2(\Omega, L^\infty(0,T;W) $.

		Recall	that	the	stochastic	It\^o	integral	is linear and continuous. 
 Thanks to \cite[Prop. 21.27 p. 261]{Zeidler} and \eqref{cv-zn-weak-1}, we infer that 
 \begin{align}\label{stochastic-cv-zn}
 \int_0^{t}\mathbb{I}_{[0,\tau_M[}(\nabla_yG(\cdot,y)z_n,\phi)d\mathcal{W}  &\rightharpoonup \int_0^{t}\mathbb{I}_{[0,\tau_M[}(\nabla_yG(\cdot,y)z,\phi)d\mathcal{W} \text{ weakly  in } L^2(\Omega\times [0,T]).
 \end{align}
\subsubsection{$2^{nd} step$} 
For  $t\in [0,T]$, let us  set $$B_n(t) :=v(z_n(t))-\displaystyle\int_0^{t}\mathbb{I}_{[0,\tau_M[}\nabla_yG(\cdot,y)z_nd\mathcal{W},\quad B(t) :=v(z(t))-\displaystyle\int_0^{t}\mathbb{I}_{[0,\tau_M[}\nabla_yG(\cdot,y)zd\mathcal{W}.$$ Due to 
  \eqref{approximation} we have 
\begin{align}\label{zn-eqn-1}
&\dfrac{d}{dt}(B_n,\phi)=-\mathbb{I}_{[0,\tau_M[}\{2\nu(\mathbb{D} z_n,\mathbb{D}\phi)+b(y,v(z_n),\phi)+b(z_n,v(y),\phi)+b(\phi,y,v(z_n))+b(\phi,z_n,v(y))\}\nonumber\\
&\quad-\mathbb{I}_{[0,\tau_M[}\{(\alpha_1+\alpha_2)\big(A(y)A(z_n)+A(z_n)A(y),\nabla \phi\big) 
+\beta\big( |A(y)|^2A(z_n), \nabla\phi\big)\}\\&\qquad-2\beta \mathbb{I}_{[0,\tau_M[}\big((A(z_n):
A(y))A(y),\nabla\phi\big)+\mathbb{I}_{[0,\tau_M[}(\psi,\phi),
\quad \forall \phi \in W_n.\nonumber
\end{align}
Let $A\in \mathcal{F}$ and $\xi\in  \mathcal{D}(0,T)$\footnote{$\mathcal{D}(0,T)$ denotes the space of $\mathcal{C}^\infty$-functions with compact suport in $]0,T[$.}, by multiplying \eqref{zn-eqn-1} by $\mathbb{I}_A\xi$ and integrating  over $\Omega_T $ we derive 
\begin{align}\label{passlimitzn}
&-\int_A\int_0^T\bigl[(B_n,\phi)\dfrac{d\xi}{ds}\bigr]dsdP\\
&=-\int_A\int_0^T\mathbb{I}_{[0,\tau_M[}\{2\nu(\mathbb{D} z_n,\mathbb{D}\phi)+b(y,v(z_n),\phi)+b(z_n,v(y),\phi)+b(\phi,y,v(z_n))+b(\phi,z_n,v(y))\}\xi dsdP\nonumber\\
&\quad-\int_A\int_0^T\mathbb{I}_{[0,\tau_M[}\{(\alpha_1+\alpha_2)\big(A(y)A(z_n)+A(z_n)A(y),\nabla \phi\big) 
+\beta\big( |A(y)|^2A(z_n), \nabla\phi\big)\}\xi dsdP\nonumber\\&\qquad-2\beta \int_A\int_0^T\mathbb{I}_{[0,\tau_M[}\big((A(z_n):
A(y))A(y),\nabla\phi\big)\xi dsdP+\int_A\int_0^T\mathbb{I}_{[0,\tau_M[}(\psi,\phi)\xi dsdP,
\quad \forall \phi \in W_n.\nonumber
\end{align}
The convergences  \eqref{cv-zn-weak-1}-\eqref{stochastic-cv-zn} allow
to pass to the limit, as $n\to \infty$, in \eqref{passlimitzn} and deduce
\begin{align}\label{trace-time-0zn}
&-\int_A\int_0^T\bigl[(B,\phi)\dfrac{d\xi}{ds}\bigr]dsdP\\
&=-\int_A\int_0^T\mathbb{I}_{[0,\tau_M[}\{2\nu(\mathbb{D} z,\mathbb{D}\phi)+b(y,v(z),\phi)+b(z,v(y),\phi)+b(\phi,y,v(z))+b(\phi,z,v(y))\}\xi dsdP\nonumber\\
&\quad-\int_A\int_0^T\mathbb{I}_{[0,\tau_M[}\{(\alpha_1+\alpha_2)\big(A(y)A(z)+A(z)A(y),\nabla \phi\big) 
+\beta\big( |A(y)|^2A(z), \nabla\phi\big)\}\xi dsdP\\&\qquad-2\beta \int_A\int_0^T\mathbb{I}_{[0,\tau_M[}\big((A(z):
A(y))A(y),\nabla\phi\big)\xi dsdP+\int_A\int_0^T\mathbb{I}_{[0,\tau_M[}(\psi,\phi)\xi dsdP,
\quad \forall \phi \in V.\nonumber
\end{align}
Then, taking into account the result of Proposition \ref{prop-estimate-zn}, we infer that the distributional derivative $\dfrac{dB}{dt}$ belongs to the space $L^2(\Omega;L^2(0,T;V^\prime).$ Recalling that $B\in L^2(\Omega;L^2(0,T;L^2(D))$, we conclude that 
(see \cite{Simon})  
 $B(\cdot)=v(z(\cdot))-\displaystyle\int_0^{\cdot}\mathbb{I}_{[0,\tau_M[}\nabla_yG(s,y)zd\mathcal{W}(s)\in L^2(\Omega;\mathcal{C}([0,T];V^\prime).$

Considering the properties of the It\^o's integral, we conclude that 
 $ v(z)\in L^2(\Omega;\mathcal{C}([0,T];V^\prime).$ Thus $v(z)\in L^2(\Omega;\mathcal{C}_w([0,T];(L^2(D))^d),$ thanks to \cite[Lemma. 1.4 p. 263]{Temam77}.  
In order to identify the initial datum, let $\phi \in  V$ and $\xi \in \mathcal{C}^\infty([0,t])$ for $t\in ]0,T]$ and note that the	following	integration	by	parts	formula	holds
\begin{align}\label{IPPtimez}
\int_0^t\langle	\dfrac{dB(s)}{ds},\phi\xi\rangle_{V^*,V} ds&=-\int_0^t[(B(s),\phi)\dfrac{d\xi}{ds}]ds+(B(t),\phi)\xi(t)-(z(0),\phi)_V\xi(0)
\end{align}
	First,	we	multiply \eqref{zn-eqn-1} by $\mathbb{I}_A\xi$	and integrate over $\Omega_T$.	Then,	we	pass	to	the	limit	as	$n\to\infty$	and	we	use	\eqref{IPPtimez} to obtain 
	 $\text{for all }t\in [0,T],\quad 
	 v(z_n(t)) \rightharpoonup v(z(t)) \text{ in } L^2(\Omega,(L^2(D))^d), \text{ as } n\to \infty,$
(see	e.g.	\cite[Prop.	3]{Val-Zim19} for similar arguments).
Hence the  proof of Theorem \ref{exis-z+estim} is completed.

 \section{Gâteaux differentiability of the control-to-state mapping}\label{Sec6}

 
 In	this	section,	we will prove that
the Gâteaux derivative of the control-to-state mapping is provided by the solution of the linearized equations, given by \eqref{Linearized}.
\begin{prop}\label{prop-gateux-diff}
	Let us consider $U$ and $y_0$ satisfying \eqref{data-assumptions} and $\psi\in L^p((\Omega_T, \mathcal{P}_T);(H^1(D))^d)$. Defining
	$ U_\rho=U+\rho\psi, \quad \rho \in (0,1),$
	let $(y,\tau_M)$ and $(y_\rho,\tau_M^\rho)$ be  the solutions of \eqref{I} associated with $(U,y_0)$ and $(U_\rho,y_0)$, respectively, then the following representaion holds
	\begin{align}\label{representationz}
	y_\rho=y+\rho z+\rho \delta_\rho \;\text{ with } \;
	\displaystyle\lim_{\rho\to 0}	\E\sup_{s\in [0, \tau_M\wedge \tau_M^\rho]}\Vert \delta_\rho(s)\Vert_{V}^2=0,
	\end{align}
	where $z$	is the solution of  \eqref{Linearized}, in the sense of Definition \ref{Def-lin} and  satisfying the estimate \eqref{estimate-z}.
\end{prop}

\begin{proof}

 Let $t\in [0,\tau_M\wedge \tau_M^\rho]$,  by using \eqref{I}, we 
derive
\begin{align*}
&d(v(y_\rho-y))=\{-\nabla ( \mathbf{ P}_\rho-\mathbf{ P})+\nu \Delta (y_\rho-y)-((y_\rho\cdot \nabla)v_\rho-(y\cdot \nabla)v)-\sum_{j}(v_\rho^j\nabla y_\rho^j-v^j\nabla y^j)
\\&+(\alpha_1+\alpha_2)\text{div}(A_\rho^2-A^2)+\beta\text{ div} \left ( |A(y_\rho)|^2A(y_\rho)-|A(y)|^2A(y)\right)+\rho\psi\}dt+ (G(\cdot,y_\rho)-G(\cdot,y))d\mathcal{W},
\end{align*}
	where
	$	v_\rho=v(y_\rho),\quad v=v(y),\quad A_\rho=A(y_\rho),\quad A=A(y).$\\
	Setting $z_\rho=\dfrac{y_\rho-y}{\rho}$,
	$\pi_\rho=\dfrac{ \mathbf{ P}_\rho- \mathbf{ P}}{\rho}$,
	we  notice that  $z_\rho$ is the unique solution to	
	$$
	\begin{array}{ll}
	&d(v(z_\rho))=\big\{\psi-\nabla \pi_\rho+\nu \Delta z_\rho-\left[(z_\rho\cdot\nabla) v(y_\rho)+(y\cdot\nabla) v(z_\rho)\right]\vspace{2mm}\\
	&-\sum_{j}[v^j(z_\rho)\nabla y_\rho^j+v^j(y)\nabla z_\rho^j]+\beta\text{div} \bigl( |A(y_\rho)|^2A(z_\rho)
	+[A(z_\rho):A(y_\rho)\vspace{2mm}\\
	&+A(y):A(z_\rho)]A(y)\bigr)\vspace{2mm}+(\alpha_1+\alpha_2)\text{div}
	\left[A(z_\rho)A(y_\rho)+A(y)A(z_\rho)\right]\big\}dt+\dfrac{1}{\rho}(G(\cdot,y_\rho)-G(\cdot,y))d\mathcal{W}.
	\end{array}
	$$
	Defining   $\delta_\rho=z_\rho-z$, the following equation  hold
	\begin{equation*}
	\begin{array}{ll}
	&	d(v(\delta_\rho))=\big\{-\nabla (\pi_\rho-\pi)+\nu \Delta \delta_\rho-
	\left[(y\cdot\nabla) v(\delta_\rho)+(\delta_\rho\cdot \nabla)v(y_\rho)\right]-(z\cdot\nabla) v(y_\rho-y)
	\vspace{2mm}\\
	&-\sum_{j}[v^j(y)\nabla\delta_\rho^j+v^j(\delta_\rho)\nabla y_\rho^j+v^j(z)\nabla (y_\rho-y)^j]+(\alpha_1+\alpha_2)\text{div}\bigl[A(y)A(\delta_\rho)+A(\delta_\rho)A(y_\rho)\vspace{2mm}\\
	&+A(z)A(y_\rho-y)\bigr]+\beta\text{ div} \left\{\left( A(y):A(\delta_\rho)\right)A(y)+ |A(y_\rho)|^2A(\delta_\rho)\right\}\vspace{2mm}\\
	&+\beta\text{ div} (\left[ A(y_\rho-y): A(y_\rho)+A(y): A(y_\rho-y)\right]A(z)])\vspace{2mm}\\
	&
	+\beta\text{ div} \left[ \left(A(\delta_\rho):A(y_\rho)+A(z): A(y_\rho-y)\right)A(y)\right]\big\}dt\\
	&+ \{\dfrac{1}{\rho}(G(\cdot,y_\rho)-G(\cdot,y))- \nabla_yG(\cdot,y)z\}d\mathcal{W}=:g(\delta_\rho)dt+Rd\mathcal{W}=g(\delta_\rho)dt+\sum_{ \k\geq 1} R_\k d\beta_\k.
	\end{array}
	\end{equation*}
	By applying  $(I-\alpha_1\mathbb{P}\Delta)^{-1}$ to the	last	equations and  using It\^o formula for $\Vert \delta_\rho\Vert_V^2$, we deduce
					\begin{align}
	\label{15_2}
d\Vert\delta_\rho\Vert_V^2&=2(g(\delta_\rho),\delta_\rho)dt+2(R,\delta_\rho)d\mathcal{W}+\sum_{ \k\geq 1}\Vert \tilde{\sigma}_k^\rho- \tilde{\sigma}_k^z\Vert_V^2dt,
	\end{align}
	where $\tilde{\sigma}_k^\rho$ and $\tilde{\sigma}_k^z$ are the solutions  of \eqref{Stokes} with $f$ replaced   by $\{\dfrac{1}{\rho}(\sigma_\k(\cdot,y_\rho)-\sigma_\k(\cdot,y))$ and $\nabla_y\sigma_\k(\cdot,y)z$, respectively. Thanks to (\cite[Theorem 3]{Busuioc}), there exists   $C>$ such that
	\begin{align*}
		\sum_{ \k\geq 1}\Vert \tilde{\sigma}_k^\rho- \tilde{\sigma}_k^z\Vert_V^2 \leq C\sum_{ \k\geq 1}\Vert \dfrac{1}{\rho}(\sigma_\k(\cdot,y_\rho)-\sigma_\k(\cdot,y))- \nabla_y\sigma_\k(\cdot,y)z\Vert_2^2
	\end{align*}
Using the assumption $\mathcal{H}_0$ (see \eqref{Gateaux-derivative-noise}), we have  
	\begin{align*}
	\dfrac{1}{\rho}(\sigma_\k(\cdot,y_\rho)-\sigma_\k(\cdot,y))=
	\dfrac{1}{\rho}[\sigma_\k(\cdot,y+\rho z_\rho)-\sigma_\k(\cdot,y)]=\nabla_y\sigma_\k(\cdot,y)z_\rho+R_\k^\rho(\cdot, z_\rho), \quad \forall \k \geq 1.
	\end{align*}
	Therefore
	\begin{align*}
	\sum_{ \k\geq 1}\Vert \tilde{\sigma}_k^\rho- \tilde{\sigma}_k^z\Vert_V^2 &\leq C\sum_{ \k\geq 1}\Vert \nabla_y\sigma_\k(\cdot,y)\delta_\rho+R_\k^\rho(\cdot, z_\rho)\Vert_2^2\leq  C \Vert \delta_\rho\Vert_{2}^2+\sum_{ \k\geq 1}b_\k^2\Vert z_\rho\Vert_{2}^2\rho^{2\gamma}\\
	&\leq  C \Vert \delta_\rho\Vert_{2}^2+C\Vert z+\delta_\rho\Vert_{2}^2\rho^{2\gamma} \leq C \Vert \delta_\rho\Vert_{2}^2+C\Vert z\Vert_{2}^2\rho^{2\gamma},
	\end{align*}
	where we used  \eqref{noiseV} and \eqref{Gateaux-derivative-noise} to get the last estimate. In order to estimate the first term in the right hand side of \eqref{15_2}, let us write
	$	2(g(\delta_\rho),\delta_\rho)
	=-4\nu \Vert \mathbb{D}\delta_\rho\Vert_{2}^2+R_1+R_2+R_3,$ where
	\begin{equation*}
	\begin{array}{ll}
R_1
	&=2b(y,\delta_\rho,v(\delta_\rho))-2b(\delta_\rho,y_\rho,v(\delta_\rho))-[2b(\delta_\rho,v(y_\rho),\delta_\rho)+2b(z,v(y_\rho-y),\delta_\rho)]\vspace{2mm}\\&\quad-2b(\delta_\rho,\delta_\rho,v(y))-2b(\delta_\rho,(y_\rho-y),v(z)).
	\end{array}
	\end{equation*}
	Lemma \ref{tecnical-lemma} 	ensures
$		b(\delta_\rho,y_\rho,v(\delta_\rho)) \leq C \Vert y_\rho\Vert_{W^{2,4}}\Vert \delta_\rho\Vert_{V}^2 \text{  and } \vert  b(y,\delta_\rho,v(\delta_\rho))\vert \leq C  \Vert y\Vert_{W^{2,4}}\Vert \delta_\rho\Vert_{V}^2.$
 Consequently, with the help of H\"older inequality we obtain
	\begin{equation*}
	\begin{array}{ll}
	R_1\leq C(\Vert y\Vert_{W^{2,4}}+\Vert y_\rho\Vert_{W^{2,4}})\Vert \delta_\rho\Vert_V^2+C\Vert \delta_\rho\Vert_V^2+C\Vert z\Vert_{H^2}^2\Vert y_\rho-y\Vert_{H^2}^2.
	\end{array}
	\end{equation*}
	Using the Stokes theorem and the boundary conditions for  $\delta_\rho$, we deduce
	\begin{align*}
R_2	&=-2(\alpha_1+\alpha_2)\int_D[A(y)A(\delta_\rho)+A(\delta_\rho)A(y_\rho)+A(z)A(y_\rho-y)]:\nabla\delta_\rho dx\vspace{2mm}\\
	&\leq C\Vert y\Vert_{W^{1,\infty}}\Vert \delta_\rho\Vert_V^2+C\Vert y_\rho\Vert_{W^{1,\infty}}\Vert \delta_\rho\Vert_V^2+C\Vert z\Vert_{W^{1,4}}\Vert y_\rho-y\Vert_{W^{1,4}}\Vert\delta_\rho\Vert_V\vspace{2mm}\\
	&\leq C(\Vert y\Vert_{W^{2,4}}+\Vert y_\rho\Vert_{W^{2,4}})\Vert \delta_\rho\Vert_V^2+C\Vert \delta_\rho\Vert_V^2+C\Vert z\Vert_{H^2}^2\Vert y_\rho-y\Vert_{H^2}^2.\\
	&\hspace*{-2cm}\text{	Analogous  arguments give 	}R_3
	=-2\beta\int_D \left\{A(y):A(\delta_\rho)A(y)+ |A(y_\rho)|^2A(\delta_\rho)\right\}:\nabla\delta_\rho dx\\
	& -2\beta\int_D\left\{\left[ A(y_\rho-y):A(y_\rho)+A(y): A(y_\rho-y)\right]A(z)\right\}:\nabla\delta_\rho dx\\
	&\quad
	-2\beta\int_D \left[ \left(A(\delta_\rho):A(y_\rho)+A(z): A(y_\rho-y)\right)A(y)\right]:\nabla\delta_\rho dx\\
	&\leq C\left(\Vert y\Vert_{W^{2,4}}^2+\Vert y_\rho\Vert_{W^{2,4}}^2\right)\Vert \delta_\rho\Vert_V^2+C\Vert z\Vert_{H^2}^2\Vert y_\rho-y\Vert_{H^2}^2,
	\end{align*}
	where we used H\"older and Young inequalities to deduce the last estimate. Summing up, we obtain the existence of constant $C_* >0$ such that
	\begin{align}\label{rest-step1}
	d\Vert \delta_\rho\Vert_V^2+4\nu \Vert \mathbb{D}\delta_\rho\Vert_{2}^2dt&\leq 
	C_*\left(1+\Vert y\Vert_{W^{2,4}}^2+\Vert y_\rho\Vert_{W^{2,4}}^2\right)\Vert \delta_\rho\Vert_V^2dt\nonumber\\
	&\qquad+C_*\Vert z\Vert_{H^2}^2\Vert y_\rho-y\Vert_{H^2}^2dt+C_*\Vert z\Vert_{2}^2\rho^{2\gamma}dt+2(R,\delta_\rho)d\mathcal{W}.
	\end{align}
	\\
	Regarding  the last term in the right hand side of \eqref{15_2}, we write 
	\begin{align*}
	\E\sup_{t\in [0, \tau_M\wedge \tau_M^\rho]}\vert \int_0^{t}(R,\delta_\rho)d\mathcal{W}(s)\vert
	&=\E\sup_{t\in [0, \tau_M\wedge \tau_M^\rho]}\vert \int_0^{t}\sum_{ \k\geq 1}( \dfrac{1}{\rho}(\sigma_\k(\cdot,y_\rho)-\sigma_\k(\cdot,y))- \nabla_y\sigma_\k(\cdot,y)z,\delta_\rho)d\beta_\k\vert.
	\end{align*}
	By	using  $\mathcal{H}_0$ ( see \eqref{Gateaux-derivative-noise}), we have  $\dfrac{1}{\rho}(\sigma_\k(\cdot,y_\rho)-\sigma_\k(\cdot,y))=\nabla_y\sigma_\k(\cdot,y)z_\rho+R_\k^\rho(\cdot, z_\rho),  \forall \k \geq 1.
	$
	
	Hence, 	for any $t\in [0, T]$ and $\epsilon>0$, the  Burkholder–Davis–Gundy  and Young inequalities ensure
	\begin{align*}
	&\E\sup_{s\in [0, t \wedge\tau_M\wedge \tau_M^\rho]}\vert \int_0^{s}(R,\delta_\rho)d\mathcal{W}\vert=	\E\sup_{s\in [0, t \wedge\tau_M\wedge \tau_M^\rho]}\vert \int_0^{s}\sum_{ \k\geq 1}( \nabla_y\sigma_\k(\cdot,y)\delta_\rho+R_\k^\rho(\cdot, z_\rho),\delta_\rho)d\beta_\k\vert\\
	&\leq C\E \bigl[\sum_{ \k\geq 1}\int_0^{t \wedge\tau_M\wedge \tau_M^\rho}\Vert \nabla_y\sigma_\k(\cdot,y)\delta_\rho+R_\k^\rho(\cdot, z_\rho)\Vert_2^2\Vert \delta_\rho\Vert_2^2 ds\bigr]^{1/2}\\
	&\leq 	\epsilon\E\sup_{s\in [0, t \wedge\tau_M\wedge \tau_M^\rho]}\Vert \delta_\rho(s)\Vert_{V}^2+C_\epsilon \E \int_0^{t \wedge\tau_M\wedge \tau_M^\rho}\Vert z+\delta_\rho\Vert_{2}^2\rho^{2\gamma}  ds+	C_\epsilon \E \int_0^{t \wedge\tau_M\wedge \tau_M^\rho}\Vert \delta_\rho\Vert_2^2 ds\\
&\leq	\epsilon\E\sup_{s\in [0,t \wedge \tau_M\wedge \tau_M^\rho]}\Vert \delta_\rho(s)\Vert_{V}^2+C_\epsilon \E \int_0^{t \wedge\tau_M\wedge \tau_M^\rho}\Vert z\Vert_{2}^2\rho^{2\gamma}  ds+	C_\epsilon \E \int_0^{t \wedge\tau_M\wedge \tau_M^\rho}\Vert \delta_\rho\Vert_2^2 ds,
	\end{align*}
	where  \eqref{Gateaux-derivative-noise}	and	 (\cite[Theorem 3]{Busuioc})	ensure the last inequality. 
	Taking the supremum of  
\eqref{rest-step1} over the time interval $[0, t\wedge\tau_M\wedge \tau_M^\rho]$, next applying the expectation,  we obtain
		\begin{align*}
\hspace*{-0.5cm}(1-\epsilon)&\E\sup_{s\in [0, t\wedge\tau_M\wedge \tau_M^\rho]}\Vert \delta_\rho(s)\Vert_{V}^2\leq 
	C_*\E\int_0^{t\wedge\tau_M\wedge \tau_M^\rho}\Vert z\Vert_{H^2}^2\Vert y_\rho-y\Vert_{H^2}^2ds+C_*\E\int_0^{t\wedge\tau_M\wedge \tau_M^\rho}\Vert z\Vert_{2}^2\rho^{2\gamma}ds\\
	&\qquad\qquad+C_\epsilon(M) \E \int_0^{t}\sup_{r\in [0, s \wedge\tau_M\wedge \tau_M^\rho]}
	\Vert \delta_\rho(r)\Vert_V^2 ds, \quad \forall t\in[0,T].
	\end{align*}
Now, we choose $\epsilon=\dfrac{1}{2}$ and apply Gronwall's inequality to obtain
		\begin{align*}
\E\sup_{s\in [0, \tau_M\wedge \tau_M^\rho]}\Vert \delta_\rho(s)\Vert_{V}^2
&\leq 
C_*e^{C(M)\tau_M}\big\{\E\int_0^{\tau_M\wedge \tau_M^\rho}\Vert z\Vert_{H^2}^2\Vert y_\rho-y\Vert_{H^2}^2ds+\rho^{2\gamma}\E\int_0^{\tau_M\wedge \tau_M^\rho}\Vert z\Vert_{2}^2ds\big\}\\
&:=A_1+A_2.
\end{align*}
The right hand side of the last inequality converges to $0$ as $\rho \to 0$. Indeed, it is clear that $\displaystyle\lim_{\rho\to 0}A_2=0$. On the other hand, let us introduce the following stopping time  
$$ \tau_N^z= \inf\{ t\in [0,T]: \int_0^t\Vert z\Vert_{H^2}^2ds \geq N\}\wedge \tau_M.$$
 By using  Corollary \ref{cor-rest-0}, we get	
$\displaystyle\lim_{N\to \infty}\displaystyle\lim_{\rho\to 0}\E\int_0^{\tau_M\wedge \tau_M^\rho\wedge\tau_N^z}\Vert z\Vert_{H^2}^2\Vert y_\rho-y\Vert_{H^2}^2ds=0.$
Note that   $(\tau_N^z)_N$ is non-decreasiong positive sequence of stopping time and converges to $\tau_M$ in probability, thanks to Theorem \ref{exis-z+estim}. Hence, monotone convergence theorem ensures 
$$\displaystyle\lim_{N\to \infty}\E\int_0^{\tau_M\wedge \tau_M^\rho\wedge\tau_N^z}\Vert z\Vert_{H^2}^2\Vert y_\rho-y\Vert_{H^2}^2ds=\E\int_0^{\tau_M\wedge \tau_M^\rho}\Vert z\Vert_{H^2}^2\Vert y_\rho-y\Vert_{H^2}^2ds. $$
Now,	a	standard	argument	allows	to	conclude
that $\displaystyle\lim_{\rho\to 0}A_1=0$. 
\end{proof}

\subsection{Variation of the cost functional}\label{Subec6-2}

\begin{prop}\label{vari-cost}
Let $J_M$ be given by \eqref{cost-exam} and  consider $U,\,y_0, \,\psi$, $U_\rho=U+\rho\psi$ verifying 
	the hypothesis of the Proposition 
	\ref{prop-gateux-diff}. Then  
	\begin{equation*}
	J_M(U_\rho,y_\rho)=J_M(U,y)+\rho\lambda\E\int_0^T \Vert U\Vert_{(H^1(D))^d}^{p-2}(U,\psi)_{(H^1(D))^d}dt¨+\rho\E\int_0^{\tau_M^U}(y-y_d,z)dt+o(\rho),
	\end{equation*}
	where $y_\rho,y$ are the solutions of \eqref{I}, corresponding to $(U_\rho,y_0)$ and $(U,y_0)$, respectively, $\tau_M^U$ is the stopping time correponding to the solution $y$ and $z$ is the solution of \eqref{Linearized} associated	to $\psi$.
\end{prop}
\begin{proof}
Let us split the cost functional into two parts. Namely,
\begin{align*}
\begin{array}{llll}
S_1:\mathcal{U}_{ad}^p&\to \mathbb{R};\hspace{5.6cm}
 S_2:\mathcal{U}_{ad}^p&\to \mathbb{R}\\
\qquad \quad U&\to S_1(U)=\dfrac{1}{2}\E\int_0^{\tau^U_M}\Vert y(U)-y_d\Vert_2^2dt; \qquad \qquad U&\to S_2(U) = \dfrac{\lambda}{p}\E\int_0^T\Vert U\Vert_{(H^1(D))^d}^pdt.
\end{array}
\end{align*}
First, let us derive the G\^ateaux derivative of $S_1$ at $U\in L^p((\Omega_T,\mathcal{P}_T);(H^1(D))^d)$ in direction of $\psi\in L^p((\Omega_T,\mathcal{P}_T);(H^1(D))^d)$.  Let $\rho>0$ and $z$ be the solution of \eqref{Linearized}, we write
\begin{align*}
	\big\vert \dfrac{1}{2\rho}&[S_1(U+\rho\psi)-S_1(U)]-\E\int_0^{\tau_M^U}(y-y_d,z)dt\big\vert 
	\leq I_1(\rho)+I_2(\rho)+I_3(\rho)+I_4(\rho)+I_5(\rho),\\
\text{	where	}	I_1(\rho)&=\dfrac{1}{2\rho}\E\int_0^{\tau_M^U\wedge\tau^{U+\rho\psi}_M}\Vert y(U+\rho\psi)-y(U)\Vert_2^2dt;\\
	I_2(\rho)&=\vert \E \int_0^{\tau_M^U\wedge\tau^{U+\rho\psi}_M}(y(U)-y_d,\dfrac{1}{\rho}[y(U+\rho\psi)-y(U)]-z(\psi))dt\vert;\\
		I_3(\rho)&=\vert \E\int_{\tau_M^U\wedge\tau^{U+\rho\psi}_M}^{\tau_M^U}( y(U)-y_d,z(\psi))dt\vert;\quad
	I_4(\rho)=\vert \dfrac{1}{2\rho}\E\int_{\tau_M^U\wedge\tau^{U+\rho\psi}_M}^{\tau_M^U}\Vert y(U)-y_d\Vert_2^2dt\vert;\\
	I_5(\rho)&=\vert \dfrac{1}{2\rho}\E\int_{\tau_M^U\wedge\tau^{U+\rho\psi}_M}^{\tau_M^{U+\rho\psi}}\Vert y(U+\rho \psi)-y_d\Vert_2^2dt\vert.
\end{align*}
Thanks to  Corollary \ref{cor-rest-0}, we get $\displaystyle\lim_{\rho \to 0}I_1(\rho)=0$. More precisely we have
\begin{align*}
I_1(\rho) \leq \dfrac{C}{2\rho}\E\sup_{s\in [0, \tau_M^U\wedge\tau^{U+\rho\psi}_M]}\Vert y_\rho(s)-y(s)\Vert_V^2\leq C(M,L,T)\rho\E\int_0^{\tau_M^U\wedge\tau^{U+\rho\psi}_M}\Vert \psi(s)\Vert_{2}^2ds \overset{\rho\to 0}{\to} 0.
\end{align*}
According to   Proposition \ref{prop-gateux-diff}, we have $\dfrac{1}{\rho}[y(U+\rho\psi)-y(U)]-z(\psi)=\delta_\rho.$ Therefore, using Cauchy-Schwarz inequality and \eqref{representationz}, we deduce
\begin{align*}
	I_2(\rho)&=\vert \E \int_0^{\tau_M^U\wedge\tau^{U+\rho\psi}_M}(y(U)-y_d,\delta_\rho)dt\vert\\
	&\leq 2\sqrt{T} \Big(\E\int_0^{\tau_M^U}(\Vert y(U)\Vert_2^2 +\Vert y_d\Vert_{2}^2 )ds\Big)^{1/2} \Big(\E\sup_{s\in [0, \tau_M^U\wedge \tau^{U+\rho\psi}_M]}\Vert \delta_\rho(s)\Vert_{2}^2\Big)^{1/2}\overset{\rho\to 0}{\to} 0.
\end{align*}
On the other hand, the Cauchy-Schwarz inequality,
the Fubini's theorem and Lemma \ref{Lemma-equality-stopping} yield
\begin{align*}
	I_3(\rho)&=\vert \E\int_{\tau_M^U\wedge\tau^{U+\rho\psi}_M}^{\tau_M^U}( y(U)-y_d,z(\psi))dt\vert\\
	& \leq  \sqrt{T}\Big(\int_0^{T}P(\tau_M^U\wedge\tau^{U+\rho\psi}_M <t\leq \tau_M^U)[M^2+\Vert y_d\Vert_{2}^2] ds\Big)^{1/2} \Big(\E\sup_{s\in [0, \tau_M^U]}\Vert z(s)\Vert_{2}^2\Big)^{1/2}\overset{\rho\to 0}{\to} 0.
\end{align*}
  Finally,  $I_4$ and $I_5$ can be estimated as follows
\begin{align*}
	I_4(\rho)&=\vert \dfrac{1}{2\rho}\E\int_{\tau_M^U\wedge\tau^{U+\rho\psi}_M}^{\tau_M^U}\Vert y(U)-y_d\Vert_2^2dt\vert \leq \int_0^{T}\dfrac{P(\tau_M^U\wedge\tau^{U+\rho\psi}_M < \tau_M^U)}{\rho}[M^2+\Vert y_d\Vert_{2}^2] ds;\\
	I_5(\rho)&=\vert \dfrac{1}{2\rho} \E\int_{\tau_M^U\wedge\tau^{U+\rho\psi}_M}^{\tau_M^{U+\rho\psi}}\Vert y(U+\rho \psi)-y_d\Vert_2^2dt\vert \leq \int_0^{T}\dfrac{P(\tau_M^U\wedge\tau^{U+\rho\psi}_M < \tau_M^{U+\rho\psi})}{\rho}[M^2+\Vert y_d\Vert_{2}^2] ds.
	\end{align*}
Since $\displaystyle\lim_{\rho \to 0} (U+\rho\psi)=U$	in	$X$,
	Corollary \ref{cor-delta0} ensures  that 
	 $\displaystyle\lim_{\rho \to 0}I_4(\rho)+I_5(\rho)=0$. Consequently, 
	$ \displaystyle\lim_{\rho \to 0}\vert \dfrac{1}{2\rho}[S_1(U+\rho\psi)-S_1(U)]-\E\int_0^{\tau_M^U}(y-y_d,z)dt\vert=0,$
	and the G\^ateaux derivative of $S_1$  at $U\in L^p((\Omega_T,\mathcal{P}_T);(H^1(D))^d)$ in direction of $\psi\in L^p((\Omega_T,\mathcal{P}_T);(H^1(D))^d)$ is given by	$ \displaystyle	d^GS_1(U)[\psi]=
	\E\int_0^{\tau_M^U}(y-y_d,z)dt. $
	Notice that $S_2$ is given by the $p^{th}$-power of a norm in 
	$L^p(\Omega\times(0,T);(H^1(D))^d)$. Therefore, 
	by using 	standard arguments, we can show that
	the G\^ateaux derivative of $S_2$  at $U\in L^p((\Omega_T,\mathcal{P}_T);(H^1(D))^d)$ in direction of $\psi\in L^p((\Omega_T,\mathcal{P}_T);(H^1(D))^d)$ corresponds to
	$ \displaystyle	d^GS_2(U)[\psi]=\lambda\E\int_0^T \Vert U\Vert_{(H^1(D))^d}^{p-2}(U,\psi)_{(H^1(D))^d}dt. $
	It is clear that $d^GS_1(U)$ and $d^GS_2(U)$ are linear and continuous (bounded) functionals,	which	ends	the	proof.
	\end{proof}

\section{Stochastic backward adjoint equations }\label{Sec7}
This section is devoted to the formulation and  study of the so-called adjoint equation, which is crucial for writing the first order optimality conditions for the system \eqref{I}.
It is worth  to mention that the adjoint equation,  as well as the optimality conditions, can be deduced by formally applying the Lagrange multiplier method, as explained in \cite[Sect. 2.10]{Troltzsch}.
The equations, once formally deduced, should  be rigorously justified.  For this task, a suitable  integration by parts should be performed. Since the 3D and 2D analysis are different, we will distinguish the two cases, by studying each one separately.\\

Let	$g$ be 	a predictable function satisfying $g\in L^p(\Omega;L^2(0,T;H))$. In particular, we will set later	$g=y-y_d$,	where	$y$	is	the	solution	of	\eqref{I} and $y_d\in L^2(0,T;H)$.\\
First,	let us make some observations about the stochastic part of the problem.
\begin{remark}\label{Rmq-adjoint*noise}
	Let us consider the following mapping
	\begin{align*}
	\nabla_yG: (L^2(D))^d&\to L_2(\mathbb{H}; (L^2(D))^d),\\
	u &\mapsto  \nabla_yG(\cdot,y)u:
	\mathbb{H}\to (L^2(D))^d, \quad 
	\nabla_yG(\cdot,y)ue_\k=\nabla_y\sigma_\k(\cdot,y)u, \quad\forall \k \geq 1.
	\end{align*} 
	\begin{itemize}
		\item 	We recall that $(L^2(D))^d$ and $L_2(\mathbb{H}; (L^2(D))^d)$ are  Hilbert spaces. Since $\nabla_yG$ is linear and continuous (bouned) operator, there exists a linear and bounded operator $G^*:L_2(\mathbb{H}; (L^2(D))^d) \to (L^2(D))^d$ satisfying 
		\begin{align}\label{adjoint-noise}
		(\nabla_yG u,\mathcalb{q})_{L_2(\mathbb{H}; (L^2(D))^d)}=(u,G^*\mathcalb{q}); \quad \forall \mathcalb{q}\in L_2(\mathbb{H}; (L^2(D))^d), \forall u\in (L^2(D))^d.
		\end{align}
		\item 
		It follows from \eqref{adjoint-noise} that $G^*\mathcalb{q}$ is predictable too   for every  $L_2(\mathbb{H}; (L^2(D))^d)$-valued predictable process $\mathcalb{q}$. Moreover,	for	any	$  \mathcalb{q}\in L_2(\mathbb{H}; (L^2(D))^d)$,	$\;u\in (L^2(D))^d$,	we	have
		\begin{align*}
		(\nabla_yG u,\mathcalb{q})_{L_2(\mathbb{H}; (L^2(D))^d)}=\sum_{ \k\geq 1}(\nabla_y\sigma_\k(\cdot,y) u,\mathcalb{q}e_\k)
		&= \sum_{ \k\geq 1}(u,(\nabla_y\sigma_\k(\cdot,y))^T\mathcalb{q}e_\k)=(u,G^*\mathcalb{q}),
		\end{align*}
		where $B^T$ denotes the transpose of the matrix $B$.
	\end{itemize}
\end{remark}


	\subsection{Existence and uniqueness of solution 	to the  backward adjoint equation in 2D}
			  Here we  aim to show the well-posedness	 of solution  $(	\mathcalb{p},\mathcalb{q})$ to the following 2D	  adjoint system  
	\begin{align}\label{adjoint}
	\begin{cases}
	-d(v(\mathcalb{p}))+\mathbb{I}_{[0,\tau_M[}\big\{-\nu \Delta \mathcalb{p}
	-\text{curl }v(y)\times \mathcalb{p}+\text{curl }v(y\times 
	\mathcalb{p})-(\alpha_1+\alpha_2)\text{div}\big[A(y)A(\mathcalb{p})+A(\mathcalb{p})A(y)\big]&\\[1mm]
	\quad-\beta\text{div}\big[\vert A(y)\vert^2A(\mathcalb{p})\big]-2\beta\text{div}\big[(A(y): A(\mathcalb{p}))A(y)\big]\big\}dt&\\
	\qquad=\mathbb{I}_{[0,\tau_M[}\big\{g-\nabla \pi -\displaystyle\sum_{ \k\geq 1}(\nabla_y\sigma_\k(\cdot,y))^T\mathcalb{q}e_\k\big\}dt+\displaystyle\sum_{ \k\geq 1}v(\mathcalb{q}e_\k)d\beta_\k&\hspace*{-4.15cm}\text{in } D\times (0,T)\times\Omega,\\[1mm]
	\text{div}(\mathcalb{p})=0 \quad &\hspace*{-4.15cm}\text{in } D\times (0,T)\times\Omega,\\
	\mathcalb{p}\cdot \eta=0, \quad [\eta \cdot \mathbb{D}(\mathcalb{p})]\cdot \tau=0  \quad &\hspace*{-4.15cm}\text{on } \partial D\times (0,T)\times\Omega,\\
	\mathcalb{p}(T)=0 \quad &\hspace*{-4.15cm}\text{in } D\times\Omega,
	\end{cases}
	\end{align}
	where $y$ is the solution of \eqref{I}, in the sense of Definition \ref{Def-strong-sol-main} and $\tau_M$ is given by \eqref{stopping-time}. \\

Let us introduce the notion of solution to the system \eqref{adjoint}, the so-called adjoint state.
\begin{definition}\label{Def-adjoint}
	A pair $(\mathcalb{p},\mathcalb{q})$ of  stochastic processes is  a solution to \eqref{adjoint} if it satisfies the following properties	
		\begin{enumerate}
				\item[i)] $\mathcalb{p}$ and $\mathcalb{q}$ are predictable  processes with values in $W$ and $L_2(\mathbb{H}; W)$, respectively.
						\item[ii)]  $\mathcalb{p}\in  L^\infty(0,T;L^2(\Omega;W))$,  $\mathcalb{q}\in L^{2}(\Omega_T;L_2(\mathbb{H}; W))$ and $v(\mathcalb{p})\in \mathcal{C}_w([0,T];L^2(\Omega;(L^2(D))^2))$. 
			\item[iii)]  For any $t\in [0,T]$, P-a.s.  in $\Omega$ and for any $\phi \in W$, the following equation holds
	\begin{align}
	\label{22_1}
	(v(\mathcalb{p}(t)),\phi)&+2\nu\int_t^T\mathbb{I}_{[0,\tau_M[}(s)(\mathbb{D} \mathcalb{p},
	\mathbb{D}\phi)ds+2\beta \int_t^T\mathbb{I}_{[0,\tau_M[}(s)\big(\left(A(\mathcalb{p}):	A(y)\right)A(y),\nabla\phi\big) ds\notag\\[-0.15cm]&+\int_t^T\mathbb{I}_{[0,\tau_M[}(s)\Big\{-b(\phi,\mathcalb{p},v(y))+b(\mathcalb{p},\phi,v(y))
	 +b(\mathcalb{p},y,v(\phi))
	 -b(y,\mathcalb{p},v(\phi))\Big\}ds\notag\\[-0.15cm]
	 &+\int_t^T\mathbb{I}_{[0,\tau_M[}(s)\Big\{(\alpha_1+\alpha_2)\big(A(y)A(\mathcalb{p})+A(\mathcalb{p})A(y),\nabla \phi\big)  +\beta\big( |A(y)|^2A(\mathcalb{p}), \nabla\phi\big)\Big \} ds\notag\\[-0.15cm]
	 	 &= \int_t^T\mathbb{I}_{[0,\tau_M[}(s)(g-G^*\mathcalb{q},\phi)ds+\sum_{ \k\geq 1}\int_t^T(v(\mathcalb{q}e_\k),\phi)d\beta_\k(s).
	\end{align}
\end{enumerate} 
\end{definition}
\begin{theorem}\label{exis-adjoint+estim} There exists a unique pair $(\mathcalb{p},\mathcalb{q})$ in the sense of Definition \ref{Def-adjoint},
	satisfying 
			\begin{align}\label{adjoint-estimate}
	\sup_{s\in [0, T]}\E\Vert \mathcalb{p}(s)\Vert_{W}^{2} \leq 
	C(\alpha_1,\alpha_2,\beta,\nu,T,M)\text{	and	}
	\E\int_{0}^T\Vert \mathcalb{q}\Vert_{L_2(\mathbb{H}; W)}^2ds\leq C(M),
	\end{align}
	where  $ C(M)>0$ and $ C(\alpha_1,\alpha_2,\beta,\nu,T,M)>0$. Moreover, for each fixed $M\in \mathbb{N}$ we have
	\begin{align*}
		\E\sup_{s\in [\tau_M,T]}\Vert \mathcalb{p}(s)\Vert_{W}^{2}=0 \text{  and } \;\E\int_{\tau_M}^T\Vert \mathcalb{q}\Vert_{L_2(\mathbb{H}; W)}^2ds=0.
	\end{align*}
\end{theorem}
\subsubsection{Finite dimensional approximation}
\label{22_7}Let us consider the orthonormal basis $\{h_i\}$ in $V$,	given	by	\eqref{basis1} and the finite dimensional spaces $W_n$,  already introduced in Section \ref{S1_zn}.
Setting  $ p_n(t)=\sum_{i=1}^n d_i(t)h_i$, \,
$n\in\mathbb{N},\; d_i(t)\in \mathbb{R}$, and 
$q_n(t)\in L_2(\mathbb{H}, W_n),$
the  approximation for \eqref{adjoint} reads 
\begin{align}\label{approximation-adjoint}
&(v(p_n(t)),\phi)+\int_t^T\hspace*{-0.25cm}\mathbb{I}_{[0,\tau_M[}(s)\Big\{2\nu(\mathbb{D} p_n,
\mathbb{D}\phi)ds-b(\phi,p_n,v(y))\nonumber\\[-0.15cm]
&+b(p_n,\phi,v(y))+b(p_n,y,v(\phi))
-b(y,p_n,v(\phi))\Big\}ds\nonumber\\[-0.15cm]
&+\int_t^T\hspace*{-0.25cm}\mathbb{I}_{[0,\tau_M[}(s)\Big\{(\alpha_1+\alpha_2)\big(A(y)A(p_n)+A(p_n)A(y),\nabla \phi\big)  +\beta\big( |A(y)|^2A(p_n), \nabla\phi\big)\Big \} ds\nonumber\\[-0.15cm]
&+2\beta \int_t^T\hspace*{-0.25cm}\mathbb{I}_{[0,\tau_M[}(s)\big(\left(A(p_n):	A(y)\right)A(y),\nabla\phi\big) ds= \int_t^T\hspace*{-0.25cm}\mathbb{I}_{[0,\tau_M[}(s)(g-G^*q_n,\phi)ds\nonumber\\[-0.15cm]
&+\sum_{ \k\geq 1}\int_t^T(v(q_ne_\k),\phi)d\beta_\k(s), \quad \text{ for any }\phi \in W_n.
\end{align}
From   \cite[Prop. 6.20]{backward-SPDE},  there exists a  unique pair of predictable processes $(p_n,q_n)$ such that  \eqref{approximation-adjoint} holds  for any $t\in [0,T]$ and 
\begin{align}\label{regularity-approx-adjoint}
p_n\in L^r_{\mathcal{P}_T}(\Omega;\mathcal{C}([0,T],W_n)),\quad q_n \in L^r(\Omega;L^2(0,T;L_2(\mathbb{H},W_n)), \text{	for		all	}	1\leq r\leq 2(d+1).
\end{align}
\subsubsection{Uniform estimates}\label{apriori-estimate-pn2D}
\subsubsection*{Step 1.  Estimate in the space $V$}\label{subsub1} 
Let us consider  $0\leq t\leq T$. Setting $\phi=h_i$ in \eqref{approximation-adjoint}, we obtain
\begin{align*}
&(v(p_n(t)),h_i)+\int_t^T\hspace*{-0.25cm}\mathbb{I}_{[0,\tau_M[}(s)\Big\{(\alpha_1+\alpha_2)\big(A(y)A(p_n)+A(p_n)A(y),\nabla h_i\big)  +\beta\big( |A(y)|^2A(p_n), \nabla h_i\big)\Big \} ds\\
&+\int_t^T\hspace*{-0.25cm}\mathbb{I}_{[0,\tau_M[}(s)\Big\{2\nu(\mathbb{D} p_n,
\mathbb{D}h_i)-b(h_i,p_n,v(y))+b(p_n,h_i,v(y))+b(p_n,y,v(h_i))
-b(y,p_n,v(h_i))\Big\}ds\nonumber\\
&+2\beta \int_t^T\hspace*{-0.25cm}\mathbb{I}_{[0,\tau_M[}(s)\big(\left(A(p_n):	A(y)\right)A(y),\nabla h_i\big) ds\nonumber\\
&= \int_t^T\hspace*{-0.25cm}\mathbb{I}_{[0,\tau_M[}(s)(g-G^*q_n,h_i)ds+\sum_{ \k\geq 1}\int_t^T(v(q_ne_\k),h_i)d\beta_\k(s).\nonumber
\end{align*}
Applying the It\^o's formula, integrating over the time interval
$[t,T]$,
and summing from $i=1$ to $n$, we deduce
\begin{align}\label{2cases-eqn}
&\Vert p_n(t)\Vert_V^2+\int_t^{T}2\mathbb{I}_{[0,\tau_M[}(s)\Big\{(\alpha_1+\alpha_2)\big(A(y)A(p_n)+A(p_n)A(y),\nabla p_n\big)  +\beta\big( |A(y)|^2A(p_n), \nabla p_n\big)\Big \} ds\nonumber\\
&+\int_t^T\hspace*{-0.3cm}2\mathbb{I}_{[0,\tau_M[}(s)\Big\{2\nu(\mathbb{D} p_n,
\mathbb{D}p_n)+b(p_n,y,v(p_n))
-b(y,p_n,v(p_n))\Big\}ds\nonumber\\
&+4\beta \int_t^T\mathbb{I}_{[0,\tau_M[}(s)\big(\left(A(p_n):	A(y)\right)A(y),\nabla p_n\big) ds= \int_t^T2\mathbb{I}_{[0,\tau_M[}(s)(g-G^*q_n,p_n)ds\\
&+\sum_{ \k\geq 1}\int_t^T2(q_ne_k,p_n)_Vd\beta_\k(s)-\sum_{i=1}^n\sum_{ \k\geq 1}\int_t^T (q_ne_k,h_i)_V^2 ds.\nonumber
\end{align}
We notice that
$\displaystyle\sum_{i=1}^n\sum_{ \k\geq 1}\int_t^T(q_ne_\k,h_i)_V^2ds
=\sum_{ \k\geq 1}\int_t^T\Vert q_ne_\k\Vert _V^2ds= \int_t^T\Vert q_n\Vert_{L_2(\mathbb{H},V)}^2ds.$

On the other hand, taking into account \eqref{regularity-approx-adjoint}, we have
\begin{align}\label{martingale-property-Back}
 &\E\sum_{ \k\geq 1}\int_0^T(q_ne_k,p_n)_V^2ds <\infty
\text{	thus	}	  \E\sum_{ \k\geq 1}\int_t^T2(q_ne_k,p_n)_Vd\beta_\k(s)=0.
\end{align}


Thanks to Lemma \ref{tecnical-lemma} (see Section \ref{Technical-results}), we derive 
\begin{align*}
&\int_t^T\hspace*{-0.3cm}2\mathbb{I}_{[0,\tau_M[}(s)\Big\{2\nu(\mathbb{D} p_n,
\mathbb{D}p_n)+b(p_n,y,v(p_n))
-b(y,p_n,v(p_n))\Big\}ds\nonumber\\
&\leq  4\nu \int_t^T\hspace*{-0.3cm}\mathbb{I}_{[0,\tau_M[}(s)\Vert \mathbb{D} p_n\Vert_2^2ds+C\int_t^{T}\mathbb{I}_{[0,\tau_M[}(s)\Vert y\Vert_{W^{2,4}}\Vert p_n\Vert_{V}^2ds\\
&\leq    4\nu \int_t^T\hspace*{-0.3cm}\mathbb{I}_{[0,\tau_M[}(s)\Vert \mathbb{D} p_n\Vert_2^2ds+CM\int_t^{T}\mathbb{I}_{[0,\tau_M[}(s)\Vert p_n\Vert_{V}^2ds.
\end{align*}
Now, using similar arguments to those used in \cite[Section 4.2]{yas-fer}, we deduce the following estimates
\begin{align*}
&\hspace*{-1cm}\left\vert\int_t^{T}2\mathbb{I}_{[0,\tau_M[}(s)\Big\{(\alpha_1+\alpha_2)\big(A(y)A(p_n)+A(p_n)A(y),\nabla p_n\big) \Big \} ds\right\vert \leq  C_2  \int_t^{T}\mathbb{I}_{[0,\tau_M[}(s)\Vert p_n\Vert_{H^1}^2\Vert  y\Vert_{W^{1,\infty}}ds\\
\text{		 and 	} &\vert 4\beta\int_t^T\mathbb{I}_{[0,\tau_M[}(s)\big(\left(A(p_n):	A(y)\right)A(y),\nabla p_n\big) ds+ \int_t^T\mathbb{I}_{[0,\tau_M[}(s) \beta\big( |A(y)|^2A(p_n), \nabla p_n\big)ds\vert \\
&\leq C_3  \int_t^T\mathbb{I}_{[0,\tau_M[}(s)\Vert p_n\Vert_{H^1}^2\Vert  y\Vert_{W^{1,\infty}}^2ds\leq C_3   \int_t^T\mathbb{I}_{[0,\tau_M[}(s)\Vert p_n\Vert_{H^1}^2\Vert  y\Vert_{W^{2,4}}^2ds
\end{align*}
Gathering the previous estimates, we show that there exists $C>0$, independent of $n$,  such that
\begin{align}\label{V-estim-adjoint1}
\Vert p_n(t)\Vert_V^2&+\int_t^T\Vert q_n\Vert_{L_2(\mathbb{H};V)}^2ds+4\nu \int_t^T\hspace*{-0.3cm}\mathbb{I}_{[0,\tau_M[}(s)\Vert \mathbb{D} p_n\Vert_2^2ds\leq C(M^2+1)\int_t^{T}\mathbb{I}_{[0,\tau_M[}(s)\Vert p_n\Vert_{V}^2ds \notag\\
& +\int_t^T2\mathbb{I}_{[0,\tau_M[}(s)(g-G^*q_n,p_n)ds+\sum_{ \k\geq 1}\int_t^T2(q_ne_\k,p_n)_Vd\beta_\k(s).
\end{align}
Let $\delta >0$. Thanks to Remark \ref{Rmq-adjoint*noise}	and	since $V\hookrightarrow (L^2(D))^2$, one has  
\begin{align*}
&\vert \int_t^T\mathbb{I}_{[0,\tau_M[}(s)(g-G^*q_n,p_n)ds\vert \leq C\int_t^{\tau_M}\Vert g\Vert_2^{2}ds+C_p\int_t^{\tau_M}\Vert p_n\Vert_2^{2}ds+ C \int_t^{\tau_M}\vert(G^*q_n,p_n)\vert ds\\
&\leq C\int_t^{T}\mathbb{I}_{[0,\tau_M[}(s)\Vert g\Vert_2^{2}ds+ \delta C( G^*)[ \int_t^{T} \Vert q_n\Vert_{L_ 2(\mathbb{H};V)}^2 ds+ C_\delta(T)\int_t^{T}\mathbb{I}_{[0,\tau_M[}(s) \Vert p_n\Vert_V^{2} ds].
\end{align*} By taking the expectation in \eqref{V-estim-adjoint1},	using	\eqref{martingale-property-Back}	and	an appropriate choice of $\delta ,$ we are able to deduce for any $0\leq t< \tau_M$ that
 \begin{align*}
 \E\Vert p_n(t)\Vert_V^{2}+\E\big[\int_t^{T}\Vert q_n\Vert_{L_2(\mathbb{H};V)}^2ds\big]&\leq C( G^*)(M^2+1)\E\int_t^{T}\mathbb{I}_{[0,\tau_M[}(s)\Vert p_n\Vert_{V}^{2}ds\\
 &\quad+ C\E\int_t^{T}\mathbb{I}_{[0,\tau_M[}(s)\Vert g\Vert_2^{2}ds.
 \end{align*}
 Therefore, Gr\"onwall's inequality ensures the existence of $C(M) >0$ such that
 \begin{align}\label{estimate-adjoint-V-L2}
 \sup_{r\in [t,T]}\E\Vert p_n(r)\Vert_V^{2}+\E\big[\int_t^{T}\Vert q_n\Vert_{L_2(\mathbb{H};V)}^2ds\big]\leq Ce^{C( G^*)(M^2+1)T}\E\int_t^{T}\mathbb{I}_{[0,\tau_M[}(s)\Vert g\Vert_2^{2}ds:=C(M).
 \end{align} 
\subsubsection*{Step 2. Estimate in the space $W$}
For any $t\in [0,T]$, the equation \eqref{approximation-adjoint} gives
	\begin{align}\label{adjoint-W}
	(p_n(t),h_i)_V&+\int_t^{T}\mathbb{I}_{[0,\tau_M[}(s)(F(p_n),h_i)ds\\&=\int_t^T\hspace*{-0.25cm}\mathbb{I}_{[0,\tau_M[}(s)(B(q_n),h_i)ds+\sum_{ \k\geq 1}\int_t^{T}(q_ne_\k,h_i)_Vd\beta_\k, \quad\forall 1\leq i\leq n,\nonumber\\
		\text{	where	}	F(u)&:= -\nu \Delta u
		-\text{curl }v(y)\times u+\text{curl }v(y\times 
		u)-(\alpha_1+\alpha_2)\text{div}\big[A(y)A(u)+A(u)A(y)\big]&\notag\\[1mm]
		&\quad-\beta\text{div}\big[\vert A(y)\vert^2A(u)\big]-2\beta\text{div}\big[(A(y): A(u))A(y)\big];\;
		B(q_n):=g-G^*q_n.\notag
	\end{align}

Let $\widetilde{F}_n$ and $\widetilde{B}_n$ be the solutions of \eqref{Stokes} for $f=F(p_n)$ and $f=B(q_n)$. Therefore
\begin{equation}\label{V-W-relation-adjoint}
(F(p_n),h_i)=(\widetilde{F}_n,h_i)_V, \quad (B(q_n),h_i)=(\widetilde{B}_n,h_i)_V ;\quad \forall 1\leq i \leq n
\end{equation}
Multiplying \eqref{adjoint-W} by   $\mu_i$ and using \eqref{basis1}, we deduce	for	$1\leq i\leq n$
\begin{align*}
(p_n(t),h_i)_W&+\int_t^{T}\mathbb{I}_{[0,\tau_M[}(s)(\widetilde{F}_n,h_i)_Wds=\int_t^T\hspace*{-0.25cm}\mathbb{I}_{[0,\tau_M[}(s)(\widetilde{B}_n,h_i)_Wds+\sum_{ \k\geq 1}\int_t^{T}(q_ne_\k,h_i)_Wd\beta_\k.	
\end{align*}
Now, using Ito's formula for the function  $u\mapsto u^2$ and the process   $(p_n(t),h_i)_W$, next  multiplying  by $\frac{1}{\mu_i}$ and summing over $i=1,\cdots,n$,  we are able to infer that
\begin{align*}
&\Vert p_n(t)\Vert_W^2+2\int_t^{T}\mathbb{I}_{[0,\tau_M[}(s)(\widetilde{F}_n,p_n)_Wds\\&=2\int_t^T\hspace*{-0.25cm}\mathbb{I}_{[0,\tau_M[}(s)(\widetilde{B}_n,p_n)_Wds+2\sum_{ \k\geq 1}\int_t^{T}(q_ne_\k,p_n)_Wd\beta_\k-\sum_{i=1}^n\sum_{ \k\geq 1}\int_t^{T}(q_ne_\k,\tilde{h}_i)_W^2ds.
\end{align*}
In addition, we  observe that 
$\displaystyle
	\sum_{i=1}^n\sum_{ \k\geq 1}\int_t^{T}(q_ne_\k,\tilde{h}_i)_W^2ds
	=\sum_{ \k\geq 1}\int_t^T\Vert q_ne_\k\Vert _W^2ds= \int_t^T\Vert q_n\Vert_{L_2(\mathbb{H};W)}^2ds.$
\\
Taking into account the definition of  inner product in the space $W$, we deduce
\begin{align*}
&\Vert p_n(t)\Vert_W^2+\int_t^T\Vert q_n\Vert_{L_2(\mathbb{H};W)}^2ds=-2\int_t^{T}\mathbb{I}_{[0,\tau_M[}(s)(F(p_n),p_n)ds-2\int_t^{T}\mathbb{I}_{[0,\tau_M[}(s)(F(p_n),\mathbb{P}v(p_n))ds\\&+2\int_t^T\hspace*{-0.25cm}\mathbb{I}_{[0,\tau_M[}(s)(g-G^*q_n,p_n)ds+2\int_t^T\hspace*{-0.25cm}\mathbb{I}_{[0,\tau_M[}(s)(g-G^*q_n,\mathbb{P}v(p_n))ds+2\sum_{ \k\geq 1}\int_t^{T}(q_ne_\k,p_n)_Wd\beta_\k\\&=I_1+I_2+I_3+I_4+I_5.
\end{align*}
Let	us focus on estimating  $I_i$, $i=1,\dots,5.$ 
 By using the boundary conditions, we get
\begin{align*}
&	-I_1=2\int_t^{T}\mathbb{I}_{[0,\tau_M[}(s)(F(p_n),p_n)ds=\int_t^T\hspace*{-0.3cm}2\mathbb{I}_{[0,\tau_M[}(s)\Big\{2\nu(\mathbb{D} p_n,
\mathbb{D}p_n)+b(p_n,y,v(p_n))
-b(y,p_n,v(p_n))\Big\}ds\\ &+\int_t^{T}2\mathbb{I}_{[0,\tau_M[}(s)\Big\{(\alpha_1+\alpha_2)\big(A(y)A(p_n)+A(p_n)A(y),\nabla p_n\big)  +\beta\big( |A(y)|^2A(p_n), \nabla p_n\big)\Big \} ds\nonumber\\
	&+4\beta \int_t^T\mathbb{I}_{[0,\tau_M[}(s)\big(\left(A(p_n):	A(y)\right)A(y),\nabla p_n\big) ds.
\end{align*}
Arguments already detailed in  Subsubsection \ref{subsub1} yields the existence of $C, C_\delta>0$ such that
\begin{align*}
\vert I_1\vert &\leq 	C(M^2+1)\int_t^{T}\mathbb{I}_{[0,\tau_M[}(s)\Vert p_n\Vert_{V}^2,\\
\vert I_3\vert&\leq  C\int_t^{T}\mathbb{I}_{[0,\tau_M[}(s)\Vert g\Vert_2^{2}ds+ \delta C[ \int_t^{T} \Vert q_n\Vert_{L_ 2(\mathbb{H};V)}^2 ds+ C_\delta\int_t^{T}\mathbb{I}_{[0,\tau_M[}(s) \Vert p_n\Vert_V^{2} ds], \quad\forall \delta >0.
\end{align*}
By using \eqref{regularity-approx-adjoint}, we get that 
 $\E(I_5)=0$.
Thanks to Remark \ref{Rmq-adjoint*noise}, there exists $C>0$ such that 
\begin{align*}
\vert 	I_4\vert=&2\vert \int_t^T\mathbb{I}_{[0,\tau_M[}(s)(g-G^*q_n,\mathbb{P}v(p_n))ds\vert\\
& \leq C\int_t^T\mathbb{I}_{[0,\tau_M[}(s)\Vert g\Vert_2^{2}ds+C\int_t^T\mathbb{I}_{[0,\tau_M[}(s)\Vert \mathbb{P}v(p_n)\Vert_2^{2}ds+ C \int_t^T\mathbb{I}_{[0,\tau_M[}(s)\vert(G^*q_n,\mathbb{P}v(p_n))\vert ds\\
&\leq C\int_t^{T}\mathbb{I}_{[0,\tau_M[}(s)\Vert g\Vert_2^{2}ds+ \delta C[ \int_t^{T} \Vert q_n\Vert_{L_ 2(\mathbb{H};W)}^2 ds+ \dfrac{C}{\delta}\int_t^{T}\mathbb{I}_{[0,\tau_M[}(s) \Vert p_n\Vert_W^{2} ds],\quad \forall \delta>0,
\end{align*}
where we used  $W\hookrightarrow (L^2(D))^2$. Finally, we write
\begin{align*}
	I_2&=	-2\int_t^{T}\mathbb{I}_{[0,\tau_M[}(s)(-\nu \Delta p_n
-\text{curl }v(y)\times p_n+\text{curl }v(y\times 
p_n)-(\alpha_1+\alpha_2)\text{div}\big[A(y)A(p_n)+A(p_n)A(y)\big]\\[1mm]
&\qquad-\beta\text{div}\big[\vert A(y)\vert^2A(p_n)\big]-2\beta\text{div}\big[(A(y): A(p_n))A(y)\big],\mathbb{P}v(p_n))ds
=I_2^1+I_2^2+I_2^3+I_2^4.
\end{align*}
We have
$	\vert I_2^1\vert = \vert 2\int_t^{T}\mathbb{I}_{[0,\tau_M[}(s)(-\nu \Delta p_n,\mathbb{P}v(p_n))ds\vert \leq 2\nu \int_t^{T}\mathbb{I}_{[0,\tau_M[}(s)\Vert p_n\Vert_{W}^2ds.$
By performing standard calculations, we show that there exists $C>0$ such that
\begin{align*}
\vert I_2^4\vert&= \vert 2\int_t^{T}\mathbb{I}_{[0,\tau_M[}(s)(-(\alpha_1+\alpha_2)\text{div}\big[A(y)A(p_n)+A(p_n)A(y)\big]-\beta\text{div}\big[\vert A(y)\vert^2A(p_n)\big]\\&\qquad\qquad-2\beta\text{div}\big[(A(y): A(p_n))A(y)\big],\mathbb{P}v(p_n))ds\vert \\
&\leq C \int_t^{T}\mathbb{I}_{[0,\tau_M[}(s)\Vert y\Vert_{W^{2,4}}\Vert p_n\Vert_{W}^2+ C\int_t^{T}\mathbb{I}_{[0,\tau_M[}(s) \Vert y\Vert_{W^{2,4}}^2\Vert p_n\Vert_{W}^2\leq  C(1+M^2)\int_t^{T}\mathbb{I}_{[0,\tau_M[}(s)\Vert p_n\Vert_{W}^2.
\end{align*}
Concerning $I_2^2+I_2^3$,  we invoke  Lemma \ref{IPP2D-2} (see Section \ref{Technical-results}) to get $C>0$ such that
\begin{align*}
&\vert I_2^2+I_2^3\vert= \vert 2\int_t^{T}\mathbb{I}_{[0,\tau_M[}(s)(\text{curl }v(y\times p_n)-\text{curl }v(y)\times p_n,\mathbb{P}v(p_n))ds\vert\\
&\leq C(1+\alpha_1)\int_t^{T}\mathbb{I}_{[0,\tau_M[}(s)\Vert y\Vert_{W^{2,4}}\Vert p_n \Vert_{W}\Vert \mathbb{P}v(p_n)\Vert_{2}ds+2\int_t^{T}\mathbb{I}_{[0,\tau_M[}(s)\vert b(y,v(p_n)-\mathbb{P}v(p_n),\mathbb{P}v(p_n))\vert ds.
\end{align*}
Thanks to \cite[Lemma 5]{Bus-Ift-2} and  $W^{2,4} \hookrightarrow L^\infty$ , there exists $C>0$ depending only on $D$ such that 
\begin{align*}
	\vert b(y,v(p_n)-\mathbb{P}v(p_n),\mathbb{P}v(p_n))\vert \leq \Vert y\Vert_{\infty}\Vert v(p_n)-\mathbb{P}v(p_n)\Vert_{H^1}\Vert\mathbb{P}v(p_n) \Vert_2 \leq C \Vert y\Vert_{W^{2,4}}\Vert p_n\Vert_{H^2}\Vert\mathbb{P}v(p_n) \Vert_2.
\end{align*}
Hence,  there exists $C>0$ such that
\begin{align*}
\vert I_2^2+I_2^3\vert
\leq C(1+\alpha_1)\int_t^{T}\mathbb{I}_{[0,\tau_M[}(s)\Vert y\Vert_{W^{2,4}}\Vert p_n \Vert_{W}\Vert \mathbb{P}v(p_n)\Vert_{2}ds\leq  CM(1+\alpha_1)\int_t^{T}\mathbb{I}_{[0,\tau_M[}(s)\Vert p_n \Vert_{W}^2ds.
\end{align*}
By taking the expectation and gathering the previous estimates, we obtain
\begin{align*}
&\E\Vert p_n(t)\Vert_W^2+\E\int_t^T\Vert q_n\Vert_{L_2(\mathbb{H};W)}^2ds\leq  C(\alpha_1,\alpha_2,\beta)(1+M^2)\E\int_t^{T}\mathbb{I}_{[0,\tau_M[}(s)\Vert p_n\Vert_{W}^2ds\\
&+C\E\int_t^{T}\mathbb{I}_{[0,\tau_M[}(s)\Vert g\Vert_2^{2}ds
+ \delta C[ \E\int_t^{T} \Vert q_n\Vert_{L_ 2(\mathbb{H};W)}^2 ds+ \dfrac{C}{\delta}\E\int_t^{T}\mathbb{I}_{[0,\tau_M[}(s) \Vert p_n\Vert_W^{2} ds],\quad \forall \delta>0.
\end{align*}
An appropriate choice of $\delta$ gives 
\begin{align*}
\E\Vert p_n(t)\Vert_W^2+\E\int_t^T\Vert q_n\Vert_{L_2(\mathbb{H};W)}^2ds\leq  C(1+M^2)\E\int_t^{T}\mathbb{I}_{[0,\tau_M[}(s)\Vert p_n\Vert_{W}^2ds+C\E\int_t^{T}\mathbb{I}_{[0,\tau_M[}(s)\Vert g\Vert_2^{2}ds.
\end{align*}
Now, Gronwall's inequality ensure
\begin{align}\label{pWestimate}
&\E\Vert p_n(t)\Vert_W^2+\E\int_t^T\Vert q_n\Vert_{L_2(\mathbb{H};W)}^2ds\leq  Ce^{C(\alpha_1,\alpha_2,\beta)(1+M^2)T}\E\int_t^{T}\mathbb{I}_{[0,\tau_M[}(s)\Vert g\Vert_2^{2}ds\leq  C(M), 
\end{align}
$\forall t\in [0,T]$. Let  $t\geq \tau_M$, it follows from \eqref{pWestimate} that 
$
\E\Vert p_n(t)\Vert_W^2+ \E\int_t^T\Vert q_n\Vert_{L_2(\mathbb{H};W)}^2ds=0, \; \forall t\geq \tau_M.
$
As a conclusion, we state the next proposition.
\begin{prop}\label{prop-estimate-pn}
There exist $ C(M)>0$ and $ C(\alpha_1,\alpha_2,\beta,\nu,T,M)>0$ such that
	\begin{align}
\sup_{s\in [0, T]}\E\Vert p_n(s)\Vert_{W}^{2} \leq 
C(\alpha_1,\alpha_2,\beta,\nu,T,M),\text{		 and 	}
\E\int_{0}^T\Vert q_n\Vert_{L_2(\mathbb{H}; W)}^2ds\leq C(M).
\end{align}
 Moreover, for fixed $M\in \mathbb{N}$ we have
$
\displaystyle\E\sup_{s\in [\tau_M,T]}\Vert p_n(s)\Vert_{W}^{2}=0 \text{  and } \E\int_{\tau_M}^T\Vert q_n\Vert_{L_2(\mathbb{H}; W)}^2ds=0.
$
\end{prop}

\subsubsection{Proof of Theorem \ref{exis-adjoint+estim}}\label{Subse-exis-adjoint-2D}
Here, we will adapt the reasoning previously used to prove the Theorem
\ref{exis-z+estim}.
Due to  Proposition \ref{prop-estimate-pn}, 
we deduce the existence of subsequence (denoted by the same way) $(p_n,q_n)$ and $(\mathcalb{p},\mathcalb{q}) \in  L^2(\Omega, L^2(0,T;W))\times L^2(\Omega, L^2(0,T;L_2(\mathbb{H}; W)))$, where $\mathcalb{p}$ and $\mathcalb{q}$  are  predictables on $\Omega_T$ such that 
\begin{align}
p_n &\rightharpoonup \mathcalb{p} \text{ weakly in } L^2(\Omega, L^2(0,T;W)),\label{cv-pn-weak-1}\\
q_n &\rightharpoonup \mathcalb{q} \text{ weakly  in } L^2(\Omega, L^2(0,T;L_2(\mathbb{H}; W)))\label{cv-qn-weak-1}.
\end{align} 
According to Remark \ref{Rmq-adjoint*noise} we also know that $G^*$ defined by \begin{align*}
G^*:L^2(\Omega\times[0,T]; L_2(\mathbb{H}; (L^2(D))^d)) \to L^2(\Omega\times[0,T]; \mathbb{R});\quad q \mapsto (G^*q,\phi) 
\end{align*}
is linear and bounded.	Recall	that	the	stochastic	It\^o	integral	is linear and continuous. 
 Thus, using \cite[Prop. 21.27 p. 261]{Zeidler}  and \eqref{cv-qn-weak-1} we show that 
\begin{align}
\sum_{ \k\geq 1}\int_t^T(v(q_ne_\k),\phi)d\beta_\k(s)  &\rightharpoonup \sum_{ \k\geq 1}\int_t^T(v(\mathcalb{q}e_\k),\phi)d\beta_\k(s) \text{  in } L^2(\Omega\times [0,T]),\quad \forall t\in [0,T],\label{stochastic-cv-pn}\\
(G^*q_n,\phi)&\rightharpoonup  (G^*\mathcalb{q},\phi) \text{  in } L^2(\Omega\times [0,T]),\quad \forall t\in [0,T]\label{adjoint-cv-qn}.
\end{align}

Setting  $M_n(t) =v(p_n(t))-\sum_{ \k\geq 1}\int_t^Tv(q_ne_\k)d\beta_\k(s)$, 
 $t\in [0,T]$, the relation 
 \eqref{approximation-adjoint} gives 
\begin{align}\label{pn-eqn-1}
&\dfrac{d}{dt}(M_n(t),\phi)=\mathbb{I}_{[0,\tau_M[}\Big\{2\nu(\mathbb{D} p_n,
\mathbb{D}\phi)ds-b(\phi,p_n,v(y))+b(p_n,\phi,v(y))+b(p_n,y,v(\phi))
\nonumber\\[-0.15cm]
&-b(y,p_n,v(\phi))\Big\}+\mathbb{I}_{[0,\tau_M[}\Big\{(\alpha_1+\alpha_2)\big(A(y)A(p_n)+A(p_n)A(y),\nabla \phi\big)  +\beta\big( |A(y)|^2A(p_n), \nabla\phi\big)\Big \}\nonumber\\[-0.15cm]
&+2\beta\mathbb{I}_{[0,\tau_M[}\big(\left(A(p_n):	A(y)\right)A(y),\nabla\phi\big) - \mathbb{I}_{[0,\tau_M[}(g-G^*q_n,\phi).  
\end{align}
With the help of \eqref{cv-pn-weak-1}, \eqref{stochastic-cv-pn} and \eqref{adjoint-cv-qn}, we are able  to pass to the limit in the weak sense in	this equation, as $n\to \infty$,    and deduce
\begin{align}
\label{trace-time-0pn}
&\dfrac{d}{dt}(M(t),\phi)=\mathbb{I}_{[0,\tau_M[}\Big\{2\nu(\mathbb{D} \mathcalb{p},
\mathbb{D}\phi)ds-b(\phi,\mathcalb{p},v(y))+b(\mathcalb{p},\phi,v(y))+b(\mathcalb{p},y,v(\phi))
\nonumber\\[-0.15cm]
&-b(y,\mathcalb{p},v(\phi))\Big\}+\mathbb{I}_{[0,\tau_M[}\Big\{(\alpha_1+\alpha_2)\big(A(y)A(\mathcalb{p})+A(\mathcalb{p})A(y),\nabla \phi\big)  +\beta\big( |A(y)|^2A(\mathcalb{p}), \nabla\phi\big)\Big \}\nonumber\\[-0.15cm]
&+2\beta\mathbb{I}_{[0,\tau_M[}\big(\left(A(\mathcalb{p}):	A(y)\right)A(y),\nabla\phi\big) - \mathbb{I}_{[0,\tau_M[}(g-G^*\mathcalb{q},\phi),  
\end{align}
where 
$M(t) =v(\mathcalb{p}(t))-\sum_{ \k\geq 1}\int_t^Tv(\mathcalb{q}e_\k)d\beta_\k(s)$.
Using Proposition \ref{prop-estimate-pn}, we can verify that
the distributional derivative 
$
\frac{dM}{dt}$ belongs to $
L^2(0,T;L^2(\Omega;W^\prime)).
$
On the other hand,  Proposition \ref{prop-estimate-pn} ensures that $M \in L^\infty(0,T;L^2(\Omega;(L^2(D))^d))$, then we infer that
$
M\in \mathcal{C}([0,T];L^2(\Omega;W^\prime)).
$
Taking into account the properties of the stochastic  integral, we conclude that $  v(\mathcalb{p})\in \mathcal{C}([0,T];L^2(\Omega;W^\prime)).$ Thus $v(\mathcalb{p})\in \mathcal{C}_w([0,T];L^2(\Omega;(L^2(D))^d)),$ thanks to \cite[Lemma. 1.4 p. 263]{Temam77}.	To verify that $\mathcalb{p}(T)=0$, let $\phi \in  \overline{W_n}$ and $\xi \in \mathcal{C}^\infty([t,T])$ for $t\in [0,T]$ and note that
\begin{align}\label{IPPtimep}
&\int_t^T\langle	\dfrac{dM(s)}{ds},\phi\xi\rangle_{W^*,W} ds=-\int_t^T[(M(s),\phi)\dfrac{d\xi}{ds}]ds
+(M(T),\phi)_V\xi(T).
\end{align}
First,	we	multiply \eqref{pn-eqn-1} by $\mathbb{I}_A\xi$,	$\xi \in \mathcal{C}^\infty([0,t])$ for $t\in ]0,T]$	and integrate over $\Omega_T$.	Then,	we	pass	to	the	limit	as	$n\to\infty$	and	we	use	\eqref{IPPtimep}.	It	follows	by		standard	arguments	(see	e.g.	\cite[Prop.	3]{Val-Zim19}) that, for all $t\in [0,T]$, $v(p_n(t)) \rightharpoonup v(\mathcalb{p}(t))$ in $L^2(\Omega,(L^2(D))^d)$. Hence the  proof of Theorem \ref{exis-adjoint+estim} is completed.

\begin{remark}
\label{22_3} 
$i)$	Let us stress here that the main difference between the adjoint equation in 2D and 3D relies in the analysis of the term $b(\mathcalb{p},y,v(\phi))-b(y,\mathcalb{p},v(\phi))$.  
In the former case, we can take advantage of the equality
$
b(\mathcalb{p},y,v(\phi))-b(y,\mathcalb{p},v(\phi))=(\text{curl}v(y\times \mathcalb{p}), \phi),
$
which  plays a crucial role  in the deduction of the  
$W$-estimate for the solution of 
\eqref{adjoint} (see  \eqref{adjoint-W}-\eqref{V-W-relation-adjoint}). Instead, in 3D we have the more complicated  relation
	\begin{align*}
	b(\mathcalb{p},y,v(\phi))-b(y,\mathcalb{p},v(\phi)) = (\text{curl}v(y\times \mathcalb{p}), \phi)+I_{\partial D},
	\end{align*}
unfortunately, the boundary terms $I_{\partial D}$	(see	\eqref{boundary-3D})
resulting from integration by parts do not vanish, and 
are very difficult to handle  in order to derive the $W$-estimate.\\
$ii)$ We recall that the  derivation of the first order optimality conditions can be performed by using the formal Lagrange method, where the   \textit{"formal Lagrangian function"} $L(y,U,\mathcalb{p})$  becomes   "meaningful " after  only one integration by parts, and $D_yL(y,U,\mathcalb{p})$\footnote{$D_vL$ denotes the derivative of $\mathcal{L}$ with respect to the variable $v$.} leads to the variational equation  \eqref{22_1} and the  "formal" optimality condition 
\begin{align}\label{Optimality-formal}
	D_UL(y(\widetilde{U_M}),\widetilde{U_M},\mathcalb{p})(u-\widetilde{U_M}) \geq 0, \quad\forall u \in \mathcal{U}_{ad}^p.
	\end{align} Therefore, in 2D it is the variational equation  \eqref{22_1}, which will play the most  important role to rigorously prove the necessary first order optimality conditions \eqref{Optimality-formal} (see Section \ref{Section-Optimality-condition}).  This observation motivates the natural extension of the 2D adjoint equation  to the 3D setting.
\end{remark}


\subsection{Existence of solution to the backward adjoint equation in 3D}
Taking into account  Remark \ref{22_3}, it  is expected that the passage from the 2D to the 3D framework is accompanied by some loss  of regularity for the solution of the adjoint system. 

\begin{definition}\label{Def-adjoint3D}
A stochastic processes $(\mathbf{p},\mathbf{q})$ is  said to be 
	an adjoint state for the control problem  if the following properties hold
	\begin{enumerate}
		\item[i)] 	$\mathbf{p}$ and $\mathbf{q}$ are predictables  processes with values in $V$ and $L_2(\mathbb{H}; V)$, respectively.
		
		\item[ii)]  	 $\mathbf{p}\in  L^\infty(0,T;L^2(\Omega;V))$,  $\mathbf{q}\in L^{2}(\Omega_T;L_2(\mathbb{H}; V))$ and $\mathbf{p}\in \mathcal{C}_w([0,T];L^2(\Omega;V)).$ 
		\item[iii)] 	  For any $t\in [0,T]$, P-a.s.  in $\Omega$, the following adjoint equation is satisfied
			\begin{align}
		\label{adjoint3D}
			&(\mathbf{p}(t),\phi)+2\alpha_1(\mathbb{D}\mathbf{p}(t),\mathbb{D}\phi)\notag+2\beta \int_t^T\mathbb{I}_{[0,\tau_M[}(s)\big(\left(A(\mathbf{p}):	A(y)\right)A(y),\nabla\phi\big) ds\\
		&+\int_t^T\mathbb{I}_{[0,\tau_M[}(s)\Big\{2\nu(\mathbb{D} \mathbf{p},
		\mathbb{D}\phi)ds-b(\phi,\mathbf{p},v(y))+b(\mathbf{p},\phi,v(y))+b(\mathbf{p},y,v(\phi))
		-b(y,\mathbf{p},v(\phi))\Big\}ds\notag\\[-0.15cm]
		&+\int_t^T\mathbb{I}_{[0,\tau_M[}(s)\Big\{(\alpha_1+\alpha_2)\big(A(y)A(\mathbf{p})+A(\mathbf{p})A(y),\nabla \phi\big)  +\beta\big( |A(y)|^2A(\mathbf{p}), \nabla\phi\big)\Big \} ds\notag\\[-0.15cm]
		&\quad= \int_t^T\mathbb{I}_{[0,\tau_M[}(s)(g-G^*\mathbf{q},\phi)ds+\sum_{ \k\geq 1}\int_t^T(\mathbf{q}e_\k,\phi)_Vd\beta_\k(s),\quad\forall\phi \in W,
		\end{align}
		where $y$ is the solution of \eqref{I}, in the sense of Definition \ref{Def-strong-sol-main} and $\tau_M$ is given by \eqref{stopping-time}.
	\end{enumerate} 
\end{definition}
\begin{theorem}\label{exis-adjoint+estim3D}
There exists, at least, a pair $(\mathbf{q},\mathbf{q})$,
	which is a solution of  the adjoint equation \eqref{adjoint3D}, according to the 
	Definition \ref{Def-adjoint3D}, verifying 
	 the properties
	\begin{align}\label{adjoint-estimate3D}
\exists C(M)>0	\text{	such	that	}	\sup_{s\in [0,T]}\E\Vert \mathbf{p}(s)\Vert_{V}^{2} \leq C(M)\text{		 and 	}
	\E\int_{0}^T\Vert \mathbf{q}\Vert_{L_2(\mathbb{H}; V)}^2ds\leq C(M).
	\end{align}
 Moreover, for fixed $M\in \mathbb{N}$ we have
$	\displaystyle\E\sup_{s\in [\tau_M,T]}\Vert \mathbf{p}(s)\Vert_{V}^{2}=0 \text{  and } \E\int_{\tau_M}^T\Vert \mathbf{q}\Vert_{L_2(\mathbb{H}; V)}^2ds=0.
$
	\end{theorem}

\subsubsection{Proof of Theorem \ref{exis-adjoint+estim3D}}
As in the 2D case (see Section \ref{22_7}) we consider 
$ p_n(t)=\sum_{i=1}^n d_i(t)h_i$	and 
$q_n(t)\in L_2(\mathbb{H}, W_n),$	  $t\in [0,T]$.
The approximated problem for \eqref{adjoint3D} reads 
\begin{align}\label{approximation-adjoint3D}
&(p_n(t),v(\phi))+\int_t^T\hspace*{-0.25cm}\mathbb{I}_{[0,\tau_M[}(s)\Big\{2\nu(\mathbb{D} p_n,
\mathbb{D}\phi)ds-b(\phi,p_n,v(y))+b(p_n,\phi,v(y))+b(p_n,y,v(\phi))
-b(y,p_n,v(\phi))\Big\}ds\nonumber\\[-0.15cm]
&+\int_t^T\hspace*{-0.25cm}\mathbb{I}_{[0,\tau_M[}(s)\Big\{(\alpha_1+\alpha_2)\big(A(y)A(p_n)+A(p_n)A(y),\nabla \phi\big)  +\beta\big( |A(y)|^2A(p_n), \nabla\phi\big)\Big \} ds\\[-0.15cm]
&+2\beta \int_t^T\hspace*{-0.25cm}\mathbb{I}_{[0,\tau_M[}(s)\big(\left(A(p_n):	A(y)\right)A(y),\nabla\phi\big) ds= \int_t^T\hspace*{-0.25cm}\mathbb{I}_{[0,\tau_M[}(s)(g-G^*q_n,\phi)ds+\sum_{ \k\geq 1}\int_t^T(q_ne_\k,v(\phi))d\beta_\k(s),  
\nonumber
\end{align}
for any $\phi \in W_n.$  Arguments already detailed ensure the existence 
  of a  unique pair of predictable processes $(p_n,q_n)$ such that  \eqref{approximation-adjoint3D} holds  for any $t\in [0,T]$	and
 \begin{align}\label{regularity-approx-adjoint-3D}
p_n\in L^r(\Omega;\mathcal{C}([0,T],W_n)),\quad q_n \in L^r(\Omega;L^2(0,T;L_2(\mathbb{H},W_n)),\qquad \forall  1\leq r\leq 2(d+1).
\end{align}
Arguments already detailed (see Subsubsection \ref{apriori-estimate-pn2D})  yield the existence of $C(M) >0$ such that
\begin{align*}
\sup_{r\in [t,T]}\E\Vert p_n(r)\Vert_V^{2}+\E\big[\int_t^{T}\Vert q_n\Vert_{L_2(\mathbb{H};V)}^2ds\big]\leq Ce^{C( G^*)(M^2+1)T}\E\int_t^{T}\mathbb{I}_{[0,\tau_M[}(s)\Vert g\Vert_2^{2}ds, \forall t\in [0,T].
\end{align*}
	It follows  that 
$
	\E\Vert p_n(t)\Vert_V^2+ \E\int_t^T\Vert q_n\Vert_{L_2(\mathbb{H};V)}^2ds=0, \quad \forall t\geq \tau_M.
$
Finally,	the proof follows from	straightforward adaptation   of the proof of Theorem  \ref{exis-adjoint+estim}.	On	the 	one	hand,	we notice	that	the	\textit{uniform	estimates}	allows	to	pass	to	the	limit	in	\eqref{approximation-adjoint3D}.	On	the	other	hand,	by	standard	arguments		we	get:
	$\forall	t\in [0,T]$, $(p_n(t),\phi)_V \rightharpoonup (\mathbf{p}(t),\phi)_V$ in $L^2(\Omega)$, for $\phi \in V$,
	which	allows	to	identify	the	terminal	condition.
\section{Duality relation and  optimality condition}\label{Section-Optimality-condition}
\begin{prop}\label{Duality-approx}
Let $\psi \in L^{p}((\Omega_T,\mathcal{P}_T),(H^1(D))^d)
$ and $g=y-y_d$.	Let $z_n$ be the solution of \eqref{approximation},
$(p_n,q_n)$ be the solution of  \eqref{approximation-adjoint} in	2D	and $(\mathbf{p}_n,\mathbf{q}_n)$ be the solution of \eqref{approximation-adjoint3D}	in	3D. Then, we have 
	\begin{align}
	\label{duality-approx}
	\E\int_0^{T}\mathbb{I}_{[0,\tau_M[}(\psi,p_n)ds=\E
	\int_0^T\hspace*{-0.25cm}\mathbb{I}_{[0,\tau_M[}(s)(g,z_n)ds=	\E \int_0^T\hspace*{-0.25cm}\mathbb{I}_{[0,\tau_M[}(s)(\psi,\mathbf{p}_n)dt.
	\end{align}
\end{prop}
\begin{proof}
 Set $\phi=h_i$ in \eqref{approximation} and use that $(v(z_n(\cdot)),h_i)=(z_n(\cdot),h_i)_V$, we obtain
\begin{align}
\label{approximation-duality}
&(z_n(t),h_i)_V+\int_0^{t}\hspace*{-0.15cm}\mathbb{I}_{[0,\tau_M[}\{2\nu(\mathbb{D} z_n,\mathbb{D}h_i)+b(y,v(z_n),h_i)+b(z_n,v(y),h_i)+b(h_i,y,v(z_n))+b(h_i,z_n,v(y))\}ds\nonumber\\
&+\int_0^{t}\mathbb{I}_{[0,\tau_M[}\{(\alpha_1+\alpha_2)\big(A(y)A(z_n)+A(z_n)A(y),\nabla h_i\big) 
+\beta\big( |A(y)|^2A(z_n), \nabla h_i\big)\}ds\\&+\int_0^{t}2\beta \mathbb{I}_{[0,\tau_M[}\big((A(z_n):
A(y))A(y),\nabla h_i\big)ds= \int_0^{t}\mathbb{I}_{[0,\tau_M[}(\psi,h_i)ds+\int_0^{t}\mathbb{I}_{[0,\tau_M[}(\nabla_yG(\cdot,y)z_n,h_i)d\mathcal{W}(s).\notag \end{align}
On the other hand,  let  $(\mathbf{p}_n,\mathbf{q}_n)$ be the solution of \eqref{approximation-adjoint}	( or  \eqref{approximation-adjoint3D}). By setting $\phi=h_i$ in \eqref{approximation-adjoint} ( or  \eqref{approximation-adjoint3D}) and using that    $(\mathbf{p}_n(\cdot),v(h_i))=(\mathbf{p}_n(\cdot),h_i)_V$, $(\mathbf{q}_ne_\k,v(h_i))=(\mathbf{q}_ne_\k,h_i)_V$, we get
\begin{align}\label{approximation-adjoint3D-duality}
&(\mathbf{p}_n(t),h_i)_V\notag\\
&+\int_t^T\hspace*{-0.25cm}\mathbb{I}_{[0,\tau_M[}\Big\{2\nu(\mathbb{D} \mathbf{p}_n,
\mathbb{D}h_i)ds-b(h_i,\mathbf{p}_n,v(y))+b(\mathbf{p}_n,h_i,v(y))+b(\mathbf{p}_n,y,v(h_i))
-b(y,\mathbf{p}_n,v(h_i))\Big\}ds\nonumber\\[-0.15cm]
&+\int_t^T\hspace*{-0.25cm}\mathbb{I}_{[0,\tau_M[}\Big\{(\alpha_1+\alpha_2)\big(A(y)A(\mathbf{p}_n)+A(\mathbf{p}_n)A(y),\nabla h_i\big)  +\beta\big( |A(y)|^2A(\mathbf{p}_n), \nabla h_i\big)\Big \} ds\\[-0.15cm]
&+2\beta \int_t^T\hspace*{-0.25cm}\mathbb{I}_{[0,\tau_M[}\big(\left(A(\mathbf{p}_n):	A(y)\right)A(y),\nabla h_i\big) ds= \int_t^T\hspace*{-0.25cm}\mathbb{I}_{[0,\tau_M[}(g-G^*\mathbf{q}_n,h_i)ds+\sum_{ \k\geq 1}\int_t^T(\mathbf{q}_ne_\k,h_i)_Vd\beta_\k(s). \notag
\end{align}
The  Itô formula and the symmetry of matrices 
$A(y)$,   $A(\mathbf{p}_n)$, $A(z_n)$ yield
\begin{align*}
&(z_n(T),\mathbf{p}_n(T))_V-(z_n(0),\mathbf{p}_n(0))_V = \int_0^{T}\mathbb{I}_{[0,\tau_M[}(\psi,\mathbf{p}_n)ds+\int_0^{T}\mathbb{I}_{[0,\tau_M[}(\nabla_yG(\cdot,y)z_n,\mathbf{p}_n)d\mathcal{W}(s)\\
&- \int_0^T\hspace*{-0.25cm}\mathbb{I}_{[0,\tau_M[}(s)(g,z_n)ds+\int_0^T\hspace*{-0.25cm}\mathbb{I}_{[0,\tau_M[}(s)(G^*\mathbf{q}_n,z_n)ds-\sum_{ \k\geq 1}\int_0^T(\mathbf{q}_ne_\k, z_n)_Vd\beta_\k(s)\\
&-\sum_{ \k\geq 1}\int_0^T\mathbb{I}_{[0,\tau_M[}(s)(\nabla_y\sigma_\k(\cdot,y)z_n,\mathbf{q}_ne_\k)ds.
\end{align*}

By	using  Remark \ref{Rmq-adjoint*noise} and knowing that    $p_n(T)=0=z_n(0)$,	we obtain 
\begin{align*}
	&\int_0^{T}\mathbb{I}_{[0,\tau_M[}(\psi,\mathbf{p}_n)ds+\int_0^{T}\mathbb{I}_{[0,\tau_M[}(\nabla_yG(\cdot,y)z_n,\mathbf{p}_n)d\mathcal{W}(s)\\&=
	 \int_0^T\hspace*{-0.25cm}\mathbb{I}_{[0,\tau_M[}(s)(g,z_n)ds+\sum_{ \k\geq 1}\int_0^T(\mathbf{q}_ne_\k, z_n)_Vd\beta_\k(s).
	\end{align*}
Taking the expectation, we deduce
$
\displaystyle\E\int_0^{T}\mathbb{I}_{[0,\tau_M[}(\psi,\mathbf{p}_n)ds=\E
\int_0^T\hspace*{-0.25cm}\mathbb{I}_{[0,\tau_M[}(s)(g,z_n)ds.
$
\end{proof}
By passing to the limit, as $n\to \infty$, in
\eqref{duality-approx},  we establish the next result.
\begin{cor}\label{Duality}
	Consider $\psi \in L^{p}((\Omega_T,\mathcal{P}_T),(L^2(D))^d).$
		Let $z$ be the solution of \eqref{Linearized}, $(\mathcalb{p},\mathcalb{q})$ be the solution of the adjoint equations in 2D  \eqref{adjoint}(or the solution of adjoint equations in 3D  \eqref{adjoint3D},	which	is denoted	by	$(\mathbf{p},\mathbf{q})$). Then, we have 
	\begin{align*}
	\E\int_0^{T}\mathbb{I}_{[0,\tau_M[}(\psi,\mathcalb{p})ds=\E
	\int_0^T\hspace*{-0.25cm}\mathbb{I}_{[0,\tau_M[}(s)(y-y_d,z)ds=
	\E\int_0^{T}\mathbb{I}_{[0,\tau_M[}(\psi,\mathbf{p})ds.
	\end{align*}
\end{cor}
\subsection{A necessary optimality condition for \eqref{problem}}
Let $(\widetilde{U_M}, y(\widetilde{U_M}))$ be the optimal control pair. Consider $\psi \in \mathcal{U}_{ad}^p$ and  define $U_\rho=\widetilde{U_M}+\rho(\psi-\widetilde{U_M})$.  Thanks to 
 Proposition \ref{vari-cost}, we	get   
that the  G\^ateaux derivative of the cost functional $J$  is given by
\begin{align*}
	&\lim_{\rho\to 0}\dfrac{J_M(U_\rho,y_\rho)-J_M(\widetilde{U_M},y(\widetilde{U_M}))}{\rho}\\
	&\quad=\lambda\E\int_0^T \Vert \widetilde{U_M}\Vert_{(H^1(D))^d}^{p-2} (\widetilde{U_M},\psi-\widetilde{U_M})_{(H^1(D))^d}dt+\E\int_0^{\tau_M^{\widetilde{U_M}}}(y(\widetilde{U_M})-y_d,z)dt \geq 0,
	\end{align*}
where $z$ is the unique solution to the linearized problem \eqref{Linearized} with $\psi$ replaced by $\psi-\widetilde{U_M}$.
\vspace{2mm}\\
Let $(\tilde{\mathbf{p}}, \tilde{\mathbf{q}})$ be  the unique solution of \eqref{adjoint} (or \eqref{adjoint3D}	in	3D) with $g=y(\widetilde{U_M})-y_d$. The application of  Corollary \ref{Duality} yields	$
\displaystyle\E
\int_0^T\mathbb{I}_{[0,\tau_M[}(s)(y(\widetilde{U_M})-y_d,z)ds=	\E \int_0^T\mathbb{I}_{[0,\tau_M[}(s)(\psi-\widetilde{U_M},\tilde{\mathbf{p}})dt.
$
Finally, we obtain the following optimality condition,  for any   $\psi \in \mathcal{U}_{ad}^p:$\\
$
\displaystyle\E\int_0^T 	\big(\lambda\Vert \widetilde{U_M}\Vert_{(H^1(D))^d}^{p-2} (\widetilde{U_M},\psi-\widetilde{U_M})_{(H^1(D))^d}+ \mathbb{I}_{[0,\tau_M[}(\psi-\widetilde{U_M},\tilde{\mathbf{p}})\big)ds\geq 0.
	$
\section{Cost functional with derivatives}\label{Sec-V-control}

The control of the evolution of the velocity field	derivatives   is relevant  in the study of	turbulence, hydrodynamics and	combustion	theory. Therefore, it is important to consider  cost functionals	depending on the derivatives of the quantities of	interest.	For	example, enstrophy $\mathcal{E}$	is	one	of	the	quantities	that	play 	a crucial role	in	the	control	of	turbulent flows	and is given by $\mathcal{E}(y):=\Vert	\nabla	y\Vert_2^2.$	We	refer	to \cite{Sritharan} for approaches that  minimize the enstrophy. Our aim in this section is to propose an extension of our analysis to cost functionals including first-order derivatives of the velocity field. 
Let us consider the following problem
 \begin{align}\label{problem-V}
 \mathcal{\widetilde{P}}\begin{cases}
 \displaystyle\min_{U}
 \{ \dfrac{1}{2}\E\int_0^{\tau^U_M}\Vert y-y_d\Vert_V^2dt+  \dfrac{\lambda}{p}\E\int_0^T\Vert U\Vert_{(¨H^1(D))^d}^pdt, \quad\lambda>0:\\ \quad\qquad U \in \mathcal{U}_{ad}^p \text{ and } y \text{ is the solution of \eqref{I} for the minimizing }  U \in \mathcal{U}_{ad}^p,  
 \end{cases}
 \end{align}
with a desired target field $y_d \in L^2(0,T;W)$ and $(\tau_M^U)_{M\in \mathbb{N}}$ given by \eqref{stopping-time}.	A	similar	raisoning  to that of Theorem \ref{main-thm-V}	yields	the	following	result.
	\begin{theorem}\label{main-thm-V} 
	Assume $\mathcal{H}_1$. Then the control problem \eqref{problem-V}  admits a unique   optimal solution 
	$(\widetilde{U_M},\tilde{y}) \in \mathcal{U}_{ad}^p \times L^p(\Omega;L^p(0,T;\widetilde{W})),$
	where $\tilde{y}:=y(\widetilde{U_M})$ is the unique solution of \eqref{I} with $U=\widetilde{U_M}$. Moreover, under the assumption $\mathcal{H}_0$: 	
	\begin{itemize} 
		\item  there exists a unique solution   $\tilde{z}$  of the linearized equations  \eqref{Linearized}, in 2D and 3D with $y=\tilde{y}$ and $\psi=\psi-\widetilde{U_M}$;
		\item if $d=2$,  there exists a unique solution  $(\tilde{\mathcalb{p}},\tilde{\mathcalb{q}})$ of the stochastic backward adjoint equation \eqref{adjoint}, with force $g=v(\tilde{y}-y_d)$;
		\item if $d=3$, there exists,  at least, one  solution   $(\tilde{\mathbf{p}},\tilde{\mathbf{q}})$ of the stochastic backward adjoint equation \eqref{adjoint3D}, with force $g=v(\tilde{y}-y_d)$.
	\end{itemize} In addition,    the duality property 
	\begin{align*}
	\E\int_0^{T}\mathbb{I}_{[0,\tau_M[}(\psi-\widetilde{U_M},\tilde{\mathcalb{p}})ds=\E
	\int_0^T\hspace*{-0.25cm}\mathbb{I}_{[0,\tau_M[}(s)(y(\widetilde{U_M})-y_d,z)_Vds=	\E \int_0^T\hspace*{-0.25cm}\mathbb{I}_{[0,\tau_M[}(s)(\psi-\widetilde{U_M},\tilde{\mathbf{p}})dt
	\end{align*}
	is valid for any  $\psi \in \mathcal{U}_{ad}^p$,  and the following optimality condition 	holds,  for any   $\psi \in \mathcal{U}_{ad}^p$
	\begin{align*}
	&\E\int_0^T 	\big(\lambda\Vert \widetilde{U_M}\Vert_{(H^1(D))^d}^{p-2} (\widetilde{U_M},\psi-\widetilde{U_M})_{(H^1(D))^d}+ \mathbb{I}_{[0,\tau_M[}(\psi-\widetilde{U_M},\tilde{\mathbf{p}})\big)ds\geq 0;\\
	&\E\int_0^T 	\big(\lambda\Vert \widetilde{U_M}\Vert_{(H^1(D))^d}^{p-2} (\widetilde{U_M},\psi-\widetilde{U_M})_{(H^1(D))^d}+ \mathbb{I}_{[0,\tau_M[}(\psi-\widetilde{U_M},\tilde{\mathcalb{p}})\big)ds\geq 0.
	\end{align*}
\end{theorem}

\section{Technical lemmas} \label{Technical-results}
In this section, we  establish some lemmas which play an important role in the estimation of the nonlinear terms	and	the	derivation	of	the	optimality	system.
We recall that $D  \subset  \mathbb{R}^d, d=2,3$ is a  bounded  and simply connected domain  with regular boundary $\partial D$ and  $\eta=(\eta_k)_{k=1}^d$, $\tau=(\tau_k)_{k=1}^d$  denote the outward normal and the unitary tangent  to the boundary $\partial D$, respectively.

We start with a result, which can be deduced by closely following the analysis in  \cite[Appendix]{Busuioc}. 
\begin{lemma}\label{tecnical-lemma}
 There exists $C:=C(D,\eta)>0$ such that the following inequalities hold
	\begin{align*}
	\vert b(\delta,y,v(\delta))\vert \leq C \Vert y \Vert_{W^{2,4}}\Vert \delta\Vert_{V}^2	
\text{		 and 	}	\vert b(y,\delta,v(\delta))\vert \leq C \Vert y \Vert_{W^{2,4}}\Vert \delta\Vert_{V}^2, \quad \forall y \in \widetilde{W}, \forall \delta\in W.
	\end{align*}
\end{lemma}

\begin{lemma}\label{interpolation-estimate-lem}
Let $t>0$ and $0<\epsilon\leq 1$. There exists $C_\epsilon(D)>0$ such that
	\begin{align*}
\E\int_0^t \vert b(y,v(\phi),\mathbb{P}v(y))\vert^q &\leq 	\E\int_0^t \big(\Vert y\Vert_\infty \Vert \phi\Vert_{H^3}\Vert y\Vert_{W}\big)^q ds\\&\leq C_\epsilon(D) \bigl(\E\int_0^t\Vert y\Vert_{W}^{2q}ds+\Vert \phi\Vert_{L^{2q(d+1+\epsilon)}(\Omega\times(0,t);H^3)}^{2q(d+\epsilon)}\Vert y\Vert_{L^{2q(d+1+\epsilon)}(\Omega\times(0,t);V)}^{2q}\bigr),
	\end{align*}
$ \forall q\in [1,\infty[$,  $y\in L^{2q}(\Omega\times(0,t);W) \cap  L^{(d+1+\epsilon)2q}(\Omega\times(0,t);V)$ and $\phi \in L^{(d+1+\epsilon)2q}(\Omega\times(0,t);(H^3(D))^d)$.
\end{lemma}
\begin{proof} Let us consider $(q,\epsilon) \in [1,\infty[ \times ]0,1]$, and recall that 
$b(y,v(\phi),\mathbb{P}v(y))=\displaystyle\sum_{i,j=1}^d\int_Dy^i\dfrac{\partial v(\phi)^j}{\partial x_i}\mathbb{P}v(y)^jdx.$
Applying the H\"older inequality, we derive
	\begin{align}\label{apendix-1}
		\vert b(y,v(\phi),\mathbb{P}v(y))\vert\leq  C\Vert y\Vert_\infty \Vert \phi\Vert_{H^3}\Vert y\Vert_{W} \leq C_D\Vert y\Vert_{W^{1,d+\epsilon}} \Vert \phi\Vert_{H^3}\Vert y\Vert_{W},	\end{align}
		where $C_D >0$ is related to  the Sobolev embedding $W^{1,d+\epsilon}(D) \hookrightarrow L^\infty(D)$  (see \cite[Thm. 1.20]{Roubicek}). Applying again the  H\"older inequality, and the Sobolev embedding $H^2(D)\hookrightarrow W^{1,2(d-1+\epsilon)}(D)$, we deduce 
		\begin{align}\label{inter-est}
		\Vert y\Vert_{W^{1,d+\epsilon}} \leq  C\Vert y\Vert_{H^1}^{\frac{1}{d+\epsilon}}\Vert y\Vert_{W^{1,2(d-1+\epsilon)}}^{\frac{d-1+\epsilon}{d+\epsilon}}\leq  C\Vert y\Vert_{H^1}^{\frac{1}{d+\epsilon}}\Vert y\Vert_{H^2}^{\frac{d-1+\epsilon}{d+\epsilon}}.
		\end{align}
	Inserting \eqref{inter-est} in  \eqref{apendix-1} and next taking the $q^{th}$ power of the resulting inequality, we obtain
			\begin{align*}
		\vert b(y,v(\phi),\mathbb{P}v(y))\vert^q &
		\leq C_D\Vert y\Vert_{V}^{\frac{q}{d+\epsilon}} \Vert \phi\Vert_{H^3}^q\Vert y\Vert_{W}^{q\frac{2d+2\epsilon-1}{d+\epsilon}}\leq C_\epsilon(D)\bigl( \Vert y\Vert_{W}^{2q}+	 \Vert y\Vert_{V}^{2q}\Vert \phi\Vert_{H^3}^{2(d+\epsilon)q}\bigr),\end{align*}
		where we used Young inequality with $\gamma=\dfrac{2(d+\epsilon)}{2(d+\epsilon)-1} >1$.  Now,  the H\"older inequality yields
		\begin{align*}
			\E\int_0^t \Vert y\Vert_{V}^{2q}\Vert \phi\Vert_{H^3}^{2(d+\epsilon)q} ds &\leq \bigl(\E\int_0^t \Vert \phi\Vert_{H^3}^{2q(d+1+\epsilon)}ds\bigr)^{\frac{d+\epsilon}{d+1+\epsilon}} \bigl(\E\int_0^t \Vert y\Vert_{V}^{2q(d+1+\epsilon)}ds\bigr)^{\frac{1}{d+1+\epsilon}}\\
			&\leq  \Vert \phi\Vert_{L^{2q(d+1+\epsilon)}(\Omega\times(0,t);H^3)}^{2q(d+\epsilon)}\Vert y\Vert_{L^{2q(d+1+\epsilon)}(\Omega\times(0,t);V)}^{2q}
.		\end{align*}
\end{proof}


Convenient modifications in the proof of \cite[Lemma 5.3.]{Arada-Cip} yield
\begin{lemma}\label{IPP2D-2}
	Let $d=2$,  $y,\psi \in \widetilde{W}$ and $\phi \in V$. Then there exists $C>0$ depending only on $D$ such that
	$	\vert (\text{curl}(v(y\times\psi))-\text{curl} v(y)\times \psi,\phi)\vert \leq C(1+\alpha_1)\Vert y\Vert_{W^{2,4}}\Vert \psi \Vert_{W}\Vert \phi \Vert_{2}+\alpha_1\vert b(y,\Delta\psi,\phi)\vert. 
	$
\end{lemma}

\begin{lemma}\label{IPP1}
	Let $y,p\in \widetilde{W}$ and $\phi \in W$. Then
	\begin{itemize}
		\item 2D case: we have 
		$	(curl v(y\times p),\phi)=b(p,y,v(\phi))-b(y,p,v(\phi)). 
		$	\item 3D case: we have
	$		(curl v(y\times p),\phi)=b(p,y,v(\phi))-b(y,p,v(\phi))+\alpha_1I_{\partial D},	$
		 \begin{align}\label{boundary-3D}
	\text{		 where 	}\qquad	I_{\partial D}&=
		\int_{\partial D}(\text{curl} (p\cdot \nabla y-y\cdot\nabla p)\cdot (\tau\times \eta) (\phi\cdot \tau) dS\nonumber\\
		&-2\int_{\partial D}(p\cdot \nabla y-y\cdot\nabla p)\cdot\sum_{ k=1}^3 \tau_k(\eta\times \nabla )\eta_k )(\phi\cdot \tau)dS.
		\end{align}			\end{itemize}
\end{lemma}
\begin{proof}
	Let $y,p\in \widetilde{W}$ and $\phi \in W$. Integrating twice by parts, we deduce
	\begin{align*}
	(curl v(y\times p),\phi)&=b(p,y,v(\phi))-b(y,p,v(\phi))\\&+\alpha_1[\int_{\partial D}(\text{curl} \text{curl }(y\times p)\times  \phi)\cdot \eta dS+\int_{\partial D}( \text{curl }(y\times p)\times \text{curl } \phi)\cdot \eta dS].
	\end{align*} 
	Since $div(y)=div(p)=0$, one has  $\text{curl }(y\times p)=p\cdot\nabla y-y\cdot\nabla p$ and therefore
	\begin{align*}
	\text{curl} \text{curl }(y\times p)&=(d-1)[p\cdot\nabla curl(y)-y\cdot\nabla curl(p)]\\
	&-[curl(y)\cdot\nabla p-curl(p)\cdot\nabla y]+2\sum_{ k=1}^d\nabla p_k\times\nabla y_k.
	\end{align*}
	
We wish to draw the reader’s attention to the
well known explicit relation  between the normal and tangent vectors to the boundary in 2D:	$
\eta=(\eta_1,\eta_2)$	and  $\tau=(-\eta_2,\eta_1),$
 which will play  a crucial role to show that boundary terms vanish	 in 2D,	we refer to \cite[Lemma 3.6.]{Chem-Cip-2018} for more details.	Concerning	the	3D	case,
 the situation is more delicate. First, the term $curl(y)\cdot\nabla p-curl(p)\cdot\nabla y$ does not vanish	as	in	2D. Using \cite[Prop. 2]{Busuioc}, we are able to infer
		\begin{align*}
		F:&=\int_{\partial D}( \text{curl }(y\times p)\times \text{curl } \phi)\cdot \eta dS=-2\int_{\partial D}(p\cdot \nabla y-y\cdot\nabla p)\cdot(\eta \times\sum_{ k=1}^3\phi_k(\eta\times \nabla )\eta_k )dS.
		\end{align*}
			On the other hand, we notice that
	$	B:=	\int_{\partial D}(\text{curl} \text{curl }(y\times p)\times  \phi)\cdot \eta dS=\int_{\partial D}\text{curl} \text{curl }(y\times p)\cdot ( \phi\times \eta) dS.$
		Since $\phi$ is tangent to $\partial D$, we  derive
	$		B=\displaystyle\int_{\partial D}\text{curl} \text{curl }(y\times p)\cdot (\tau\times \eta) (\phi\cdot \tau) dS.
	$
		Thus
		\begin{align*}
		F+B&= \int_{\partial D}(\text{curl} \text{curl }(y\times p)\cdot (\tau\times \eta) (\phi\cdot \tau) d
		S-2\int_{\partial D}(p\cdot \nabla y-y\cdot\nabla p)\cdot(\eta \times\sum_{ k=1}^3\phi_k(\eta\times \nabla )\eta_k )dS.
		\end{align*}
		 Since $\phi\cdot \eta =0$, we write   $\phi=(\phi\cdot \tau) \tau$ and 
		$\sum_{ k=1}^3\phi_k(\eta\times \nabla )\eta_k = (\phi\cdot \tau)\sum_{ k=1}^3 \tau_k(\eta\times \nabla )\eta_k ).$
		Finally, 
		\begin{align*}
		F+B&= \int_{\partial D}(\text{curl} (p\cdot \nabla y-y\cdot\nabla p)\cdot (\tau\times \eta) (\phi\cdot \tau) dS-2\int_{\partial D}(p\cdot \nabla y\\
		&-y\cdot\nabla p)\cdot\sum_{ k=1}^3 \tau_k(\eta\times \nabla )\eta_k )(\phi\cdot \tau)dS.
		\end{align*}
\end{proof}
\section*{Acknowledgment}
This work is funded by national funds through the FCT - Funda\c c\~ao para a Ci\^encia e a Tecnologia, I.P., under the scope of the projects UIDB/00297/2020 and UIDP/00297/2020 (Center for Mathematics and Applications).


 \end{document}